\documentclass[11pt,english]{article}
\usepackage[sc]{mathpazo}
\usepackage{geometry}
\usepackage{color}
\usepackage{array}
\usepackage{bbding}
\usepackage{multirow}
\usepackage[fleqn]{amsmath}
\usepackage{amssymb}
\usepackage{amsthm}
\usepackage{graphicx}
\usepackage{natbib}
\usepackage{lscape}
\usepackage{comment}
\usepackage{bm}
\usepackage[title]{appendix}
\usepackage{longtable}
\usepackage{mathtools}
\usepackage{chngcntr}
\usepackage{authblk}

\usepackage[colorlinks,
            linkcolor=blue,
            anchorcolor=blue,
            citecolor=blue
            ]{hyperref}

\makeatletter
\allowdisplaybreaks

\newcommand{\tabincell}[2]{\begin{tabular}{@{}#1@{}}#2\end{tabular}}

\@ifundefined{showcaptionsetup}{}{%
 \PassOptionsToPackage{caption=false}{subfig}}
\usepackage{subfig}
\makeatother

\usepackage{babel}
\usepackage{enumitem}

\newtheorem{prop}{Proposition}

\newtheorem{defi}{Definition}

\DeclarePairedDelimiter{\abs}{\lvert}{\rvert}

\begin{document}\sloppy

\title{Rhythmic Control of Automated Traffic - Part II: Grid Network Rhythm and Online Routing}
\author{Xi Lin\textsuperscript{a}\hspace{2em} Meng Li\textsuperscript{a}\hspace{2em} Zuo-jun Max Shen\textsuperscript{b,c}\hspace{2em} Yafeng Yin\textsuperscript{d,e}\hspace{2em} Fang He\textsuperscript{f}\footnote{Corresponding author. E-mail address: \textcolor{blue}{fanghe@tsinghua.edu.cn}.}}
\affil{\small\emph{\textsuperscript{a}Department of Civil Engineering, Tsinghua University, Beijing 100084, P.R. China}\normalsize}
\affil{\small\emph{\textsuperscript{b}Department of Industrial Engineering and Operations Research, University of California Berkeley, Berkeley, CA 94720, United States}\normalsize}
\affil{\small\emph{\textsuperscript{c}Department of Civil and Environmental Engineering, University of California Berkeley, Berkeley, CA 94720, United States}\normalsize}
\affil{\small\emph{\textsuperscript{d}Department of Civil and Environmental Engineering, University of Michigan, Ann Arbor, MI 48109-2125, USA}\normalsize}
\affil{\small\emph{\textsuperscript{e}Department of Industrial and Operations Engineering, University of Michigan, Ann Arbor, MI 48109-2125, USA}\normalsize}
\affil{\small\emph{\textsuperscript{f}Department of Industrial Engineering, Tsinghua University, Beijing 100084, P.R. China}\normalsize}
\date{\today}
\maketitle

\begin{abstract}
\noindent Connected and automated vehicle (CAV) technology is providing urban transportation managers tremendous opportunities for better operation of urban mobility systems. However, there are significant challenges in real-time implementation, as the computational time of the corresponding operations optimization model increases exponentially with increasing vehicle numbers. Following the companion paper (Chen et al., 2020) which proposes a novel automated traffic control scheme for isolated intersections, this study proposes a network-level real-time traffic control framework for CAVs on grid networks. The proposed framework integrates a rhythmic control (RC) method with an online routing algorithm to realize collision-free control of all CAVs on a network and achieve superior performance in average vehicle delay, network traffic throughput, and computational scalability. Specifically, we construct a preset network rhythm that all CAVs can follow to move on the network and avoid collisions at all intersections. Based on the network rhythm, we then formulate online routing for the CAVs as a mixed integer linear program, which optimizes the entry times of CAVs at all entrances of the network and their time-space routings in real time. We provide a sufficient condition that the linear-programming relaxation of the online routing model yields an optimal integer solution. Extensive numerical tests are conducted to show the performance of the proposed operations management framework under various scenarios. It is illustrated that the framework is capable of achieving negligible delays and increased network throughput. Furthermore, the computational time results are also promising. The CPU time for solving a collision-free control optimization problem with 2,000 vehicles is only 0.3 s on an ordinary personal computer. \par
\hfill\break%
\noindent\textit{Keywords}: Rhythmic control; Connected and automated vehicles; Online routing; Computational scalability
\end{abstract}

\section{Introduction} \label{intro_sec}

Disruptive technologies always lead to rapid and significant revolutions in urban mobility systems (Qi and Shen, 2019). Recently, the substantial integration of advances in mobile internet and smartphone technologies into transportation systems has led to the emergence of real-time ride-sharing mobility services, substantially revolutionizing the taxi industry, and laying the foundation for a more connected and centrally-controlled mobility service operations system (see, e.g., Zha et al., 2016; Bertsimas et al., 2019). The next major wave of technology innovation seeks to utilize vehicle-based innovation to impact the urban mobility system in a profound, disruptive, and far-reaching way (Mahmassani, 2016). Automated vehicles, integrated with connected vehicle technologies, are capable of gathering information, autonomously performing all driving functions, and communicating with environments in a collaborative and real-time manner for an entire trip through an interconnected network of moving vehicular units and stationary infrastructure units (NHTSA, 2013). Consequently, the underlying operations mode of urban traffic could be significantly impacted, and automated transportation that eliminates human-driven vehicles altogether will enable robotic-like control of vehicular flow. As automated vehicles gradually become ubiquitous in the coming years, a centrally-controlled management system (CCMS) can be used to leverage the accuracy and consistency of such vehicle motion control. In particular, the CCMS can take responsibility for operating urban network traffic, and maximize network traffic efficiency through jointly optimizing traffic control and individual vehicle operations strategies in road networks and implementing coordinated strategies, such as speed harmonization and coordinated cruise control. In this regime, an urban network transportation system will eliminate unreliable and selfish human-driven environments entirely, and the operations management system will shift from isolated and uncoordinated control to being connected and centrally-controlled, indicating tremendous potential for enhancing urban mobility. \par

In a CCMS, one key focus is routing and controlling connected and automated vehicle (CAV) traffic on an urban road network, such that collisions are avoided at all intersections, and such that the network traffic throughput (vehicle delay) is significantly increased (decreased). This research question is straightforward, but quite challenging to solve. The decisions of routing need to be seamlessly integrated with vehicle management at intersections, because routing is directly affected by management strategies at intersections. Consequently, such integration requirements will inevitably lead to high computational complexity in the CCMS. Furthermore, urban traffic demands could be massive during peak hours. Therefore, the CCMS needs to satisfy a very time-constrained optimization schedule, with probably less than a few seconds to solve a network-level cooperative control problem involving the routing and intersection management of thousands of CAVs. \par

Previous studies have primarily investigated a single type of CAV control problems, such as automated management of an isolated intersection (Dresner and Stone, 2008; Levin et al., 2016;
Yu et al., 2018), centralized routing (Zhang and Nie, 2018; Li et al., 2018),, and vehicle scheduling (Muller et al., 2016; Soleimaniamiri and Li, 2019). Specifically, the companion paper (Chen et al., 2020) proposes an innovative control at isolated intersections with redesigned intersection layout. Few studies have considered the joint optimization of routing and intersection management strategies for CAV traffic in an urban road network, i.e., routing all of the CAVs in a road network to maximize traffic network throughput and minimize overall delay, while resolving inter-vehicle conflicts at intersections, which is particularly critical for CCMS performance. Moreover, the exact algorithms proposed for optimizing CAV control in a real-time manner rarely scale past a few dozens of vehicles. For instance, as reported by Levin and Rey (2017), the proposed mixed-integer linear program (MIP) that optimizes vehicle ordering at conflict points of an isolated intersection can only be solved in real-time for up to 30 vehicles. Lu et al. (2019) optimized multiple CAV trajectories while resolving inter-vehicle conflicts on a multi-lane single-link network. However, in some instances, the optimization model cannot be solved in real-time, even for only 10 vehicles. \par

This paper intends to fill this void by proposing a novel and real-time traffic control framework for CAVs on urban road networks, called \textit{rhythmic control (RC)}. It is named as "rhythmic" because the entire network traffic is organized into an regularly recurring and quickly-changing manner, which shares substantial similarity with the rhythm in music. As one of the first steps towards investigating network control of CAV traffic, this study focuses on grid networks. Similarly to many downtown areas in cities such as New York, Chicago, Barcelona, and Glasgow, we set the grid networks to be one-way, as the intersections of one-way streets have fewer conflict points and simpler streamlines than those of two-way streets. Thus, they have potential for improving traffic throughput and alleviating computational complexity in an automated control system. One drawback of a one-way street in a human-driven-vehicle environment is that many turning movements en route may increase the human’s driving burden. This, however, is not a concern for CAVs, as they can be programmed to remember and execute all maneuvers in a journey. Moreover, in Section \ref{network_sec}, we will further justify the utilization of a one-way network by analyzing its connectivity and detour ratio. The latter is defined as the ratio of shortest-path distance over Manhattan distance. In a one-way grid network, to realize real-time operations management of CAV traffic, we propose a two-level control framework. At the lower level, we adopt a RC method for resolving collisions at intersections. Specifically, we develop a pre-set \textit{network rhythm}, such that all of the virtual platoons can follow the rhythm to flow within the network without any stops or collisions. The virtual platoons on each vertical or horizontal road of the network may not necessarily consist of vehicles, but represent a spatiotemporal cell that is available for being occupied by vehicles. Then, a CAV can fulfill its trip and avoid collisions by repeating the process of joining a virtual platoon, moving forward in a virtual platoon, and leaving for joining another virtual platoon. In this case, the RC framework requires very modest computational efforts, but has great potential for reducing vehicle delays. This is because the network rhythm is pre-determined, and once a CAV starts to follow the rhythm and joins a virtual platoon, it will not incur stops. More importantly, based on the RC framework, the online routing problem for CAVs can be significantly simplified. At the upper level of the control framework, we formulate online routing for CAVs as a MIP problem, which optimizes the CAVs' entry time on the perimeter of the grid network and their time-space routings in real-time. Owing to the RC framework, the routing model does not need to optimize individual vehicle spatiotemporal trajectories or to explicitly include collision-free constraints, but rather can realize conflict-free routing for CAVs by only optimizing their choices in joining and leaving different virtual platoons. We further provide a sufficient condition that the linear-programming relaxation of the online routing model yields an optimal integer solution. As expected and to be demonstrated by numerical results, the proposed two-level control framework achieves superior performance on average vehicle delay and network traffic throughput, and makes conflict-free routing optimization formulations tractable at the largest practical scales, involving tens of thousands of CAVs per hour. \par

The remainder of this paper is organized as follows. Section \ref{network_sec} introduces the details of the one-way road network design, and analyzes the networks’ global accessibility and average detour ratio. Section \ref{RC_sec} presents the two-level rhythmic control framework. Section \ref{Extension} introduces some extensions of the RC framework to consider network rhythm length choice, imbalanced demands, temporary within-network waiting as well as heterogeneous block sizes. Section \ref{Nume_sec} reports the results of numerical experiments, and the last section concludes the paper with some suggestions for future research.

\section{One-way grid network} \label{network_sec}

To reduce the impacts of inter-vehicle conflicts on network traffic throughput and alleviate computational complexity, we set the grid networks to be one-way. Within each parallel group, the directions of the streets are arranged in an alternating fashion. Figure \ref{OWGC} shows a one-way grid network with four horizontal and vertical streets. As illustrated, besides the main streets, there are four elements on the network, i.e., entrances, exits, crossroads, and junctions, introduced as follows.\par

\noindent \textbf{Entrances}: These are the entrances to the one-way street network from outside roads. As shown in Figure \ref{Components}(a), at each entrance, there are extra waiting lanes for the automated vehicles, to allow a temporary delay; such waiting lanes are crucial for the further routings of incoming CAVs, discussed in detail in the latter section. \\
\textbf{Exits}: These are the exits from the one-way street network to the outside roads. No specific design is required for such areas. \\
\textbf{Crossroads}: These are the intersections between two perpendicular one-way streets. A detailed geometry layout is shown in Figure \ref{Components}(b). Turning lanes are arranged on both sides, to allow CAVs to turn from one street to its perpendicular street; we observe that turning lanes are extended to some distance away from the crossroads to ensure that turning vehicles can switch to these lanes in advance, then decelerate and finally accelerate. \\
\textbf{Junctions}: These are the connecting points between the road network and zones with other functionalities, e.g., parking facilities, pickup areas, buildings, and garages. One typical design is shown in Figure \ref{Components}(c). CAVs can both leave and enter the road network at junctions. Similarly to crossroads, there is an extra lane extended to some distance away to allow vehicles to decelerate or accelerate. Note that in reality there may not be a connecting road linking the network and in-building facilities, physically; it may only be a curbside lane where vehicles can cruise or park temporarily, and trip demands could be randomly generated at these spots.\par

\begin{figure}[!ht]
	\centering
	\includegraphics[width=0.8\textwidth]{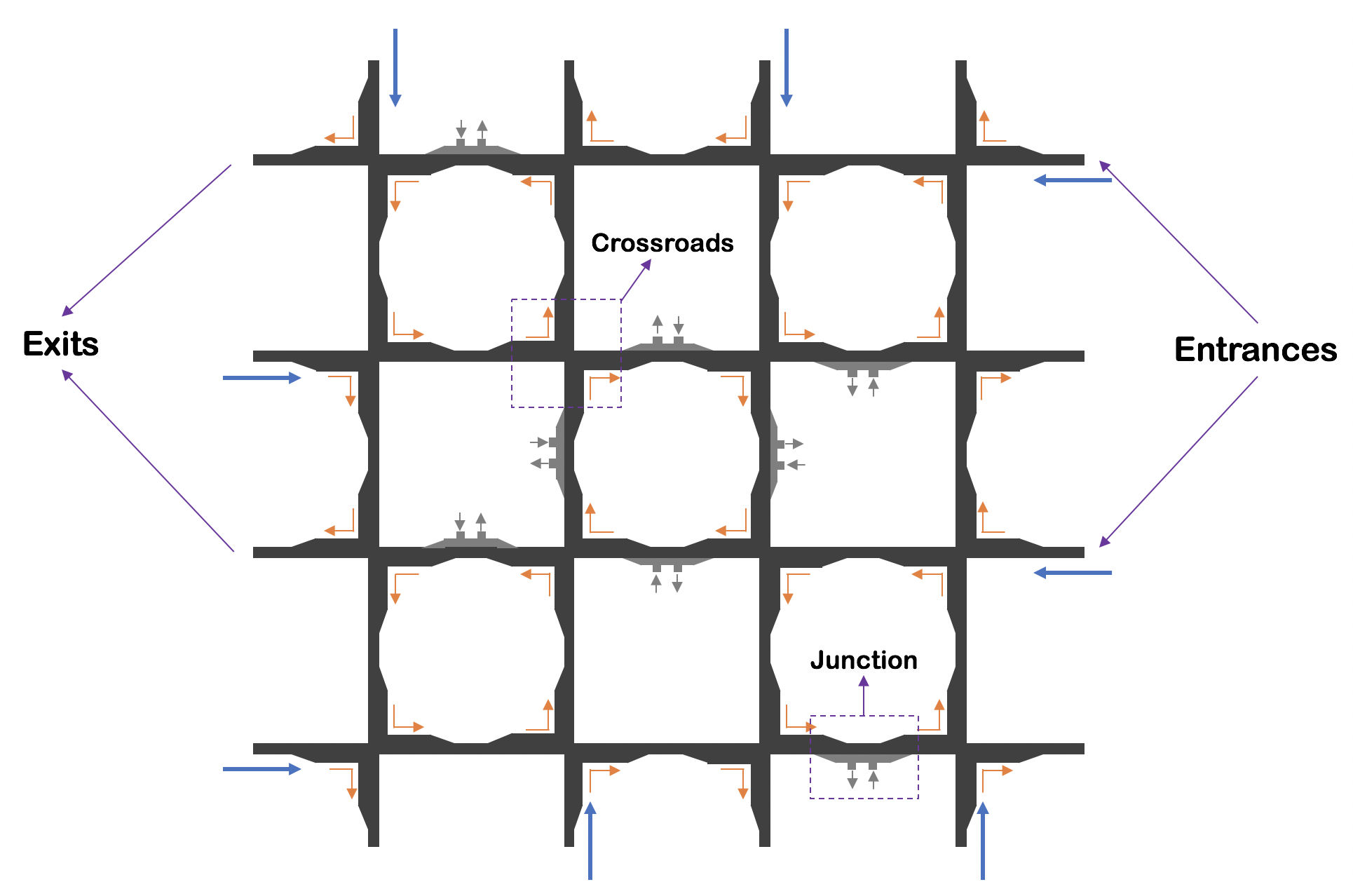}
	\caption[]{One-way grid network} 
	\label{OWGC}
\end{figure}

\begin{figure}[!ht]
	\centering
	\subfloat[][]{\includegraphics[width=0.7\textwidth]{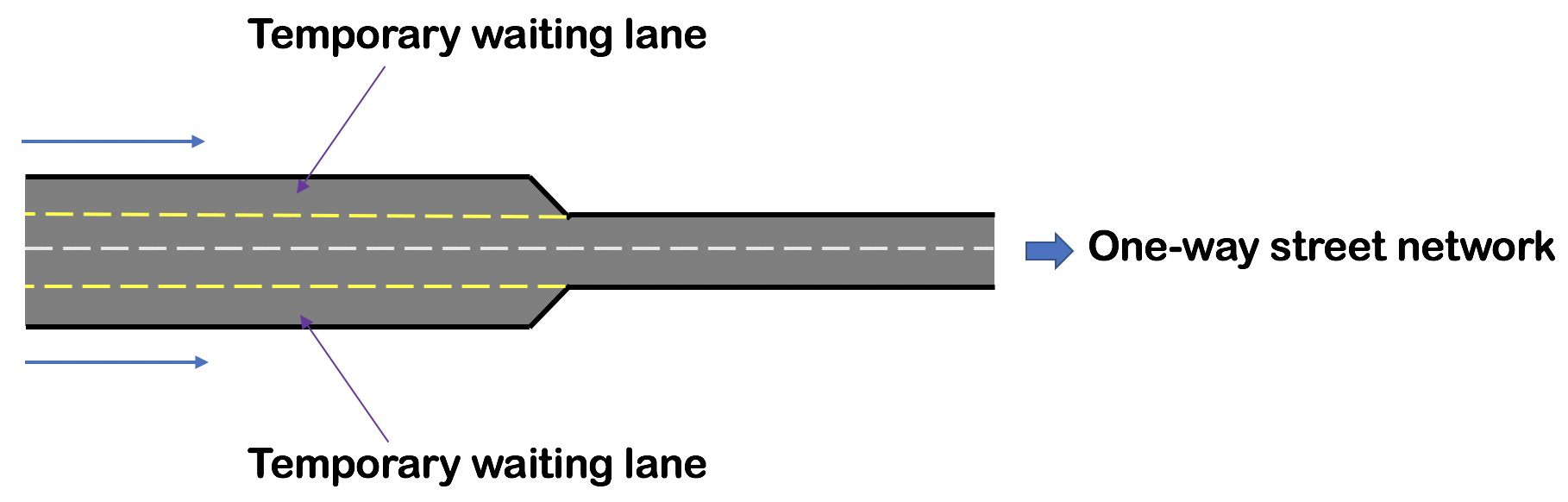}}\\
	\subfloat[][]{\includegraphics[width=0.55\textwidth]{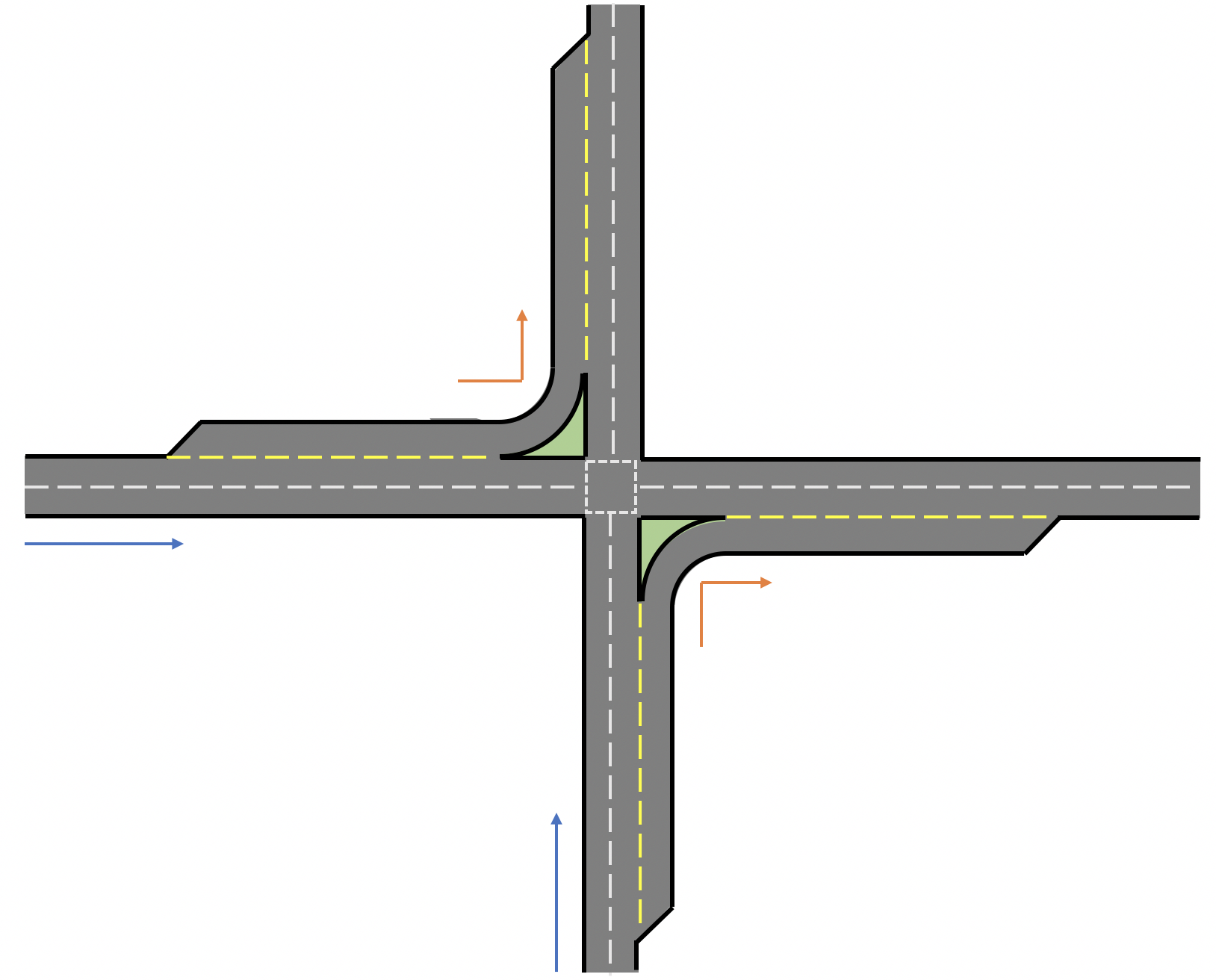}}\\
	\subfloat[][]{\includegraphics[width=0.7\textwidth]{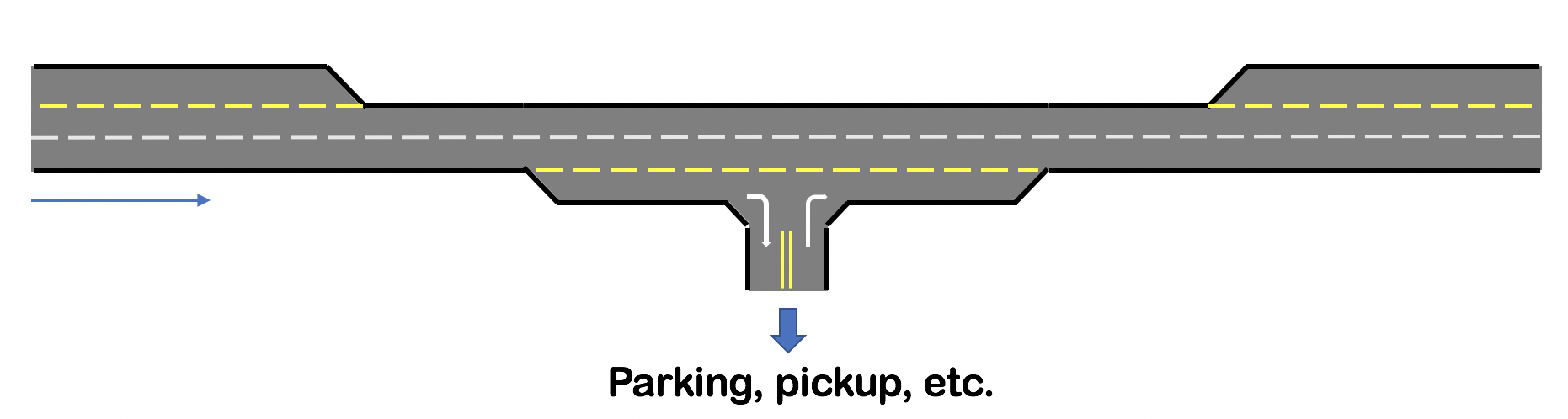}}
	\caption[]{Elements in one-way grid network: (a) Entrance (b) Crossroads (c) Junctions}
	\label{Components}
\end{figure}

Below, we prove a crucial property for the network, namely, global accessibility. We define an $m \times n$ network as a network with $m$ horizontal streets and $n$ vertical streets. The numbering rule for horizontal streets is from the bottom up, and that for vertical streets is from left to right.

\begin{prop}\label{prop1}
In an $m \times n$ one-way street network where $m, n$ are even numbers, a vehicle entering from an entrance or a junction can reach any exit or junction.
\end{prop}

\begin{proof}
See Appendix A.
\end{proof}

One concern in adopting a one-way network is that it may lead to extra detours of vehicles. To resolve it, we establish a rather meaningful result regarding the detour ratios of the aforementioned one-way grid network. The detour ratio is defined as the ratio of the actual driving distance (shortest driving distance in this study) to the corresponding Manhattan distance; high detour ratios could undermine the level-of-service of the network for travelers. We estimate the average detour ratios of a one-way grid network of size $m \times n$ by assuming that all origin-destination (O-D) pairs share the same travel demands, and the equation is given by:

\begin{align} \label{Detour}
& \bar{r} = \frac{\sum_{w \in \mathcal{W}} \hat{s}_w }{\sum_{w \in \mathcal{W}} s_w}
\end{align}

\noindent Here, $\bar{r}$ is the average detour ratio, $s_w$ is the Manhattan distance of O-D pair $w$, $\hat{s}_w$ is the shortest-path distance of O-D pair $w$, and $\mathcal{W}$ is the set of all O-D pairs. The results are shown in Table \ref{AveDR} below.

\begin{table}[!ht]
  \centering
  \caption{Average detour ratios of one-way grid networks with different sizes}
  \label{AveDR}
  \small
  \begin{tabular}{ l | c  c  c  c  c  c  c  c }
   \hline
   $m \backslash n$ & \tabincell{c}{2} & 4 & 6 & 8 & 10 & 12 & 14 & 16 \\
		\hline
		2	& 1.668 & 1.528 & 1.423 & 1.353 & 1.303 & 1.266 & 1.238 & 1.213 \\
		4   & 1.528 & 1.371 & 1.296 & 1.251 & 1.221 & 1.198 & 1.181 & 1.166 \\
		6   & 1.423 & 1.296 & 1.237 & 1.202 & 1.179 & 1.163 & 1.150 & 1.140 \\
		8   & 1.353 & 1.251 & 1.202 & 1.172 & 1.152 & 1.139 & 1.129 & 1.121 \\
		10  & 1.303 & 1.221 & 1.179 & 1.152 & 1.135 & 1.122 & 1.114 & 1.111 \\
		12  & 1.266 & 1.198 & 1.163 & 1.139 & 1.122 & 1.110 & 1.102 & 1.096 \\
		14  & 1.238 & 1.181 & 1.150 & 1.129 & 1.114 & 1.102 & 1.093 & 1.087 \\
		16  & 1.213 & 1.166 & 1.140 & 1.121 & 1.111 & 1.096 & 1.087 & 1.081 \\
   \hline
  \end{tabular}
\end{table}

From Table \ref{AveDR}, we see that when the network size is large enough (e.g., $m, n \ge 10$ ), the average detour ratio is generally around 1.1, and the average detour ratio decreases with increased network size. Since the Manhattan distance is equivalent to the shortest-path distance for the two-way grid network, the numbers shown in Table \ref{AveDR} actually show the additional detour induced by the one-way network compared to its two-way counterpart. Therefore, Table \ref{AveDR} indicates that the detour brought by one-way structure is subtle, especially for large-scale networks.\par

Note that we propose the use of one-way street network to reduce the complexity of conflicts. In the literature, some scholars have suggested other forms of network lane configurations with fewer (or even zero) conflicts (e.g., Eichler et al., 2013; Boyles et al., 2014); similarly, they reduced the conflicts on a network at the cost of increased vehicle detours. As demonstrated by Eichler
et al. (2013), the vortex-based zero-conflict network design induces larger detours than one-way networks in general. For simplicity, this study only considers the use of one-way streets for supporting our traffic control framework where the conflict relation is simple enough and the induced additional detours are sufficiently mild.

\section{Rhythmic control framework} \label{RC_sec}

In this section, we develop a two-level RC framework to realize online routing for CAVs, while resolving inter-vehicle conflicts at intersections. Before elaborating on the framework details, we first show the following motivating example for the RC.

\subsection{Motivating example} \label{motivate_subsec}

Consider an intersection with one conflicting point comprised of two intersecting lanes, as shown in Figure \ref{Twolane_fig}(a). Leveraging the accuracy and consistency of CAVs’ motion control, we are capable of controlling vehicles to pass through the intersection in a regularly recurring sequence, such that vehicles of the two intersecting lanes pass through the intersection in an alternating and conflict-free manner at a constant speed without any stop, as shown in Figure \ref{Twolane_fig}(b). We henceforth refer to this regularly recurring sequence as a \textit{rhythm}.

\begin{figure}[!ht]
	\centering
	\subfloat[][]{\includegraphics[width=0.4\textwidth]{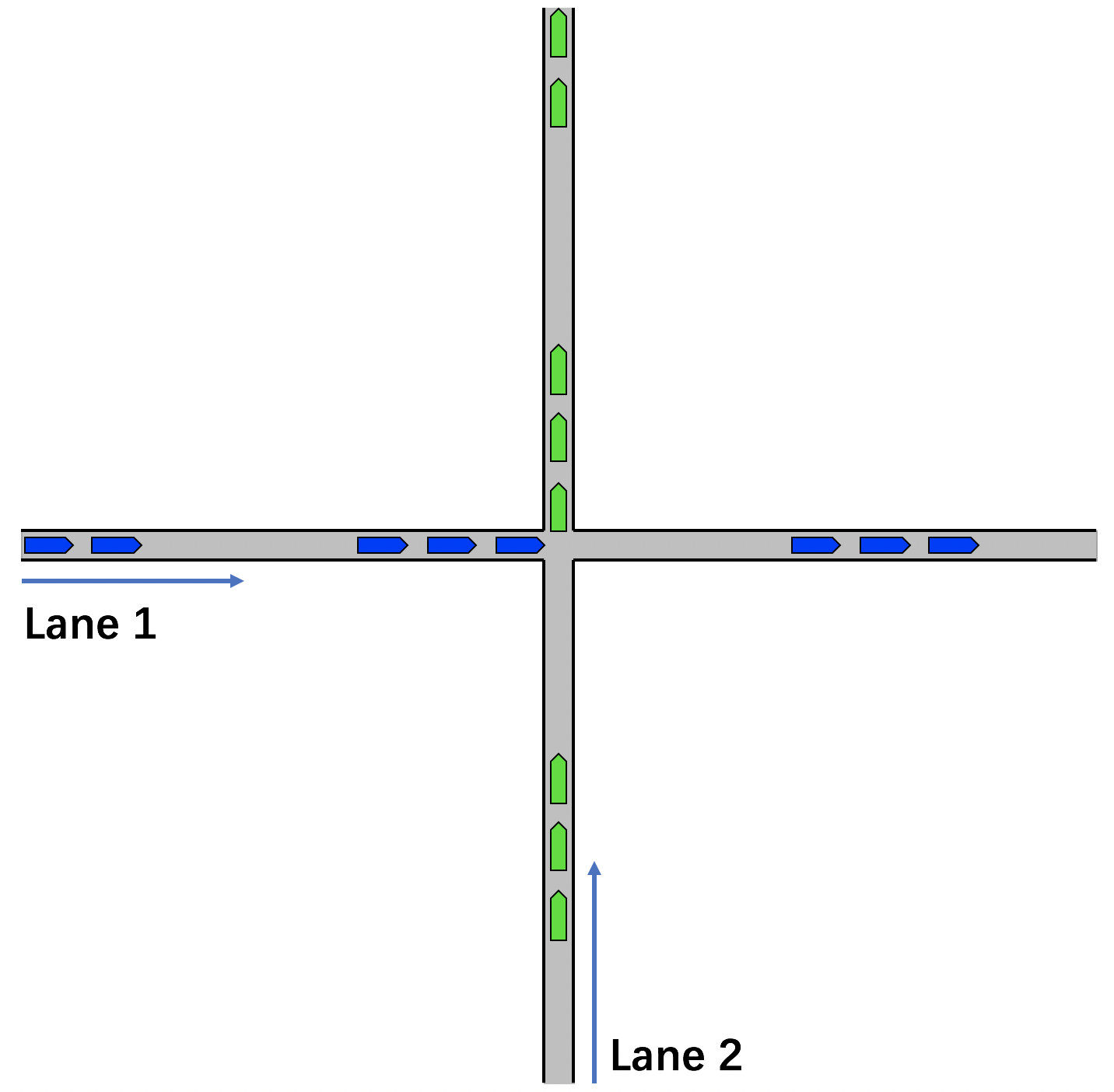}} \hspace{1em}
	\subfloat[][]{\includegraphics[width=0.4\textwidth]{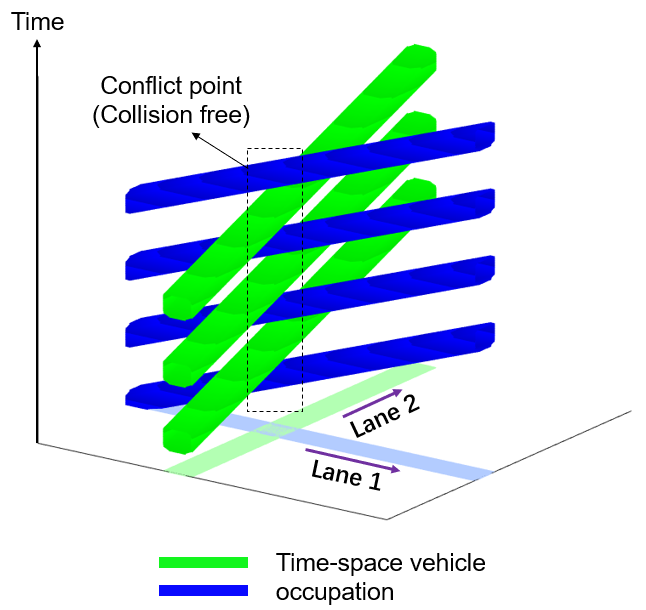}}
	\caption[]{Two-lane intersection. (a) Intersection layout. (b) Time-space vehicle occupations. } 
	\label{Twolane_fig}
\end{figure}

Inspired by this example, we propose an RC framework that enables CAVs to flow within a grid network at a constant speed without any stop by letting them follow a preset and coordinated rhythm. Consider two terminologies repeatedly utilized in this study. A \textit{virtual platoon} denotes a reserved time-space occupation with a predetermined trajectory along an arterial. The platoon is “virtual” because, in practice, it is possible that no vehicle is contained therein, especially in low-demand scenarios. For a one-way grid network, a \textit{network rhythm} represents a regularly recurring sequence of virtual platoons on arterials entering the network, such that all of the virtual platoons can follow the rhythm to flow within the network without any stops or collisions. Under a network rhythm, on each arterial, the time between two virtual platoons consecutively formed and allowed to enter the network is denoted by $\hat{t}$. The general concept of the proposed RC framework is described as follows. \par

\noindent i) Virtual platoons are admitted to enter the network by following a network rhythm. \\
ii) A CCMS solves an online routing problem to determine when a CAV can enter the network and which virtual platoons it will successively join and leave to fulfill its trip.

\subsection{Network rhythm} \label{rhythm_subsec}

The RC framework must prevent vehicles arriving from different directions from colliding at crossroads. Our proposed RC framework adopts a network rhythm to achieve this, i.e., by designing pre-set entry times for newly-formed virtual platoons of different roads within each interval $\hat{t}$, rather than relying on dynamically controlling the movements of vehicles/platoons to avoid conflicts in their space-time trajectories, as this latter method requires complicated computational efforts and possibly massive amounts of accelerations/decelerations. Consequently, all virtual platoons on the network will be perfectly coordinated, i.e., they will proceed on the network with a constant speed in a conflict-free manner. \par

We first propose a rigorous conflict-free statement. In what follows, we use $l_p$ to denote the maximum length of a virtual platoon, $v$ to denote the constant platoon speed, and $t_c$ to denote the travel time required for a virtual platoon to pass through a segment between two consecutive crossroads. Proposition \ref{prop2} below provides the details of the statement.

\begin{prop}\label{prop2}
On an $m \times n$ one-way street network where $m, n$ are even numbers, as long as $\frac{l_p}{v} \le \frac{1}{2}\hat{t}$ and $t_c = a \hat{t}, a \in \mathbb{N}^+$, we can set the entry times of virtual platoons on horizontal streets to be $k\hat{t}, k \in \mathbb{N}$ and on vertical streets to be $(k+\frac{1}{2})\hat{t}, k \in \mathbb{N}$ to avoid collisions at crossroads.
\end{prop}

\begin{proof}
We denote the time required for a virtual platoon to drive from an entrance to the first crossroads as $t_e$. For an arbitrary crossroads, a virtual platoon on a horizontal street entering at $k_h\hat{t}$ will arrive at the crossroads at $k_h\hat{t} + t_e + c_h t_c = (k_h + a c_h)\hat{t} + t_e$, and a virtual platoon on the vertical street entering at $(k_v + \frac{1}{2})\hat{t}$ will arrive at the crossroads at $(k_v + \frac{1}{2})\hat{t} + t_e + c_v t_c = (k_v + \frac{1}{2} + a c_v)\hat{t} + t_e$; in that regard, $c_h$ and $c_v$ are the number of road segments between two consecutive crossroads that a vehicle on the horizontal/vertical street (respectively) must pass through before reaching a target crossroads after the first crossroads on the street; $c_h, c_v \in \mathbb{N}$. The arrival time difference between two platoons from different directions is computed by $\Delta t = [k_h - k_v - \frac{1}{2} + a (c_h - c_v)] \hat{t} = (k' - \frac{1}{2})\hat{t}$, where $k' = k_h - k_v + a (c_h - c_v) \in \mathbb{Z}$; thus we know $\abs{\Delta t} \ge \frac{1}{2}\hat{t}$. As $\frac{l_p}{v} \le \frac{1}{2}\hat{t}$, we then know that no platoons collide at the crossroads. Proof completes.
\end{proof}

From Proposition \ref{prop2}, we see that the keys for the successful coordination of virtual platoon movements are: i) preset entry times of virtual platoons at entrances; ii) $t_c = a \hat{t}, a \in \mathbb{N}^+$. As $t_c$ is determined by the speed of the virtual platoons, the second condition requires a synergy between the virtual platoon speed and the network rhythm. To provide an example, assuming that the distance between two consecutive crossroads is 150 m and that the platoon speed is 15 m/s, then $t_c = 10$ s. and we can therefore set $\hat{t}$ to be 10 s, 5 s, $\frac{10}{3}$ s, etc. By the proof of Proposition \ref{prop2}, it is clear that virtual platoons from different directions pass through crossroads in alternating fashion, i.e., just after a virtual platoon passes through the crossroads, another platoon arrives from the perpendicular street. A network rhythm is pre-determined, and once a CAV starts to follow the network rhythm and joins a virtual platoon, it will not incur stops. In this case, a network rhythm requires very modest computational efforts, but has great potential for reducing vehicle delays. Figures \ref{Fourlane_fig}(a)-\ref{Fourlane_fig}(b) below extend the two-lane intersection in Figure \ref{Twolane_fig} into a four-lane network to demonstrate the concept of network rhythm. We note that the concept of network rhythm shares some similarities with the network-level signal coordination (e.g., Gartner et al., 1975; Yan et al., 2019); for example, the length of a rhythm can be treated as an analog to the common control cycle through the network, and the virtual platoons can be seen as the green-wave bands from a moving perspective. Proposition \ref{prop2} indicates that there exists an offset design that can "perfectly" coordinate all intersections without any displacement between consecutive green-wave bands. Nevertheless, utilizing the controllability of CAVs, we can do more than just setting the offsets and letting the vehicles follow the traffic lights (like the traditional signal control); the following will demonstrate how to organize all vehicle movements into those virtual platoons for maintaining safety and efficiency at the same time.

\begin{figure}[!ht]
	\centering
	\subfloat[][]{\includegraphics[width=0.4\textwidth]{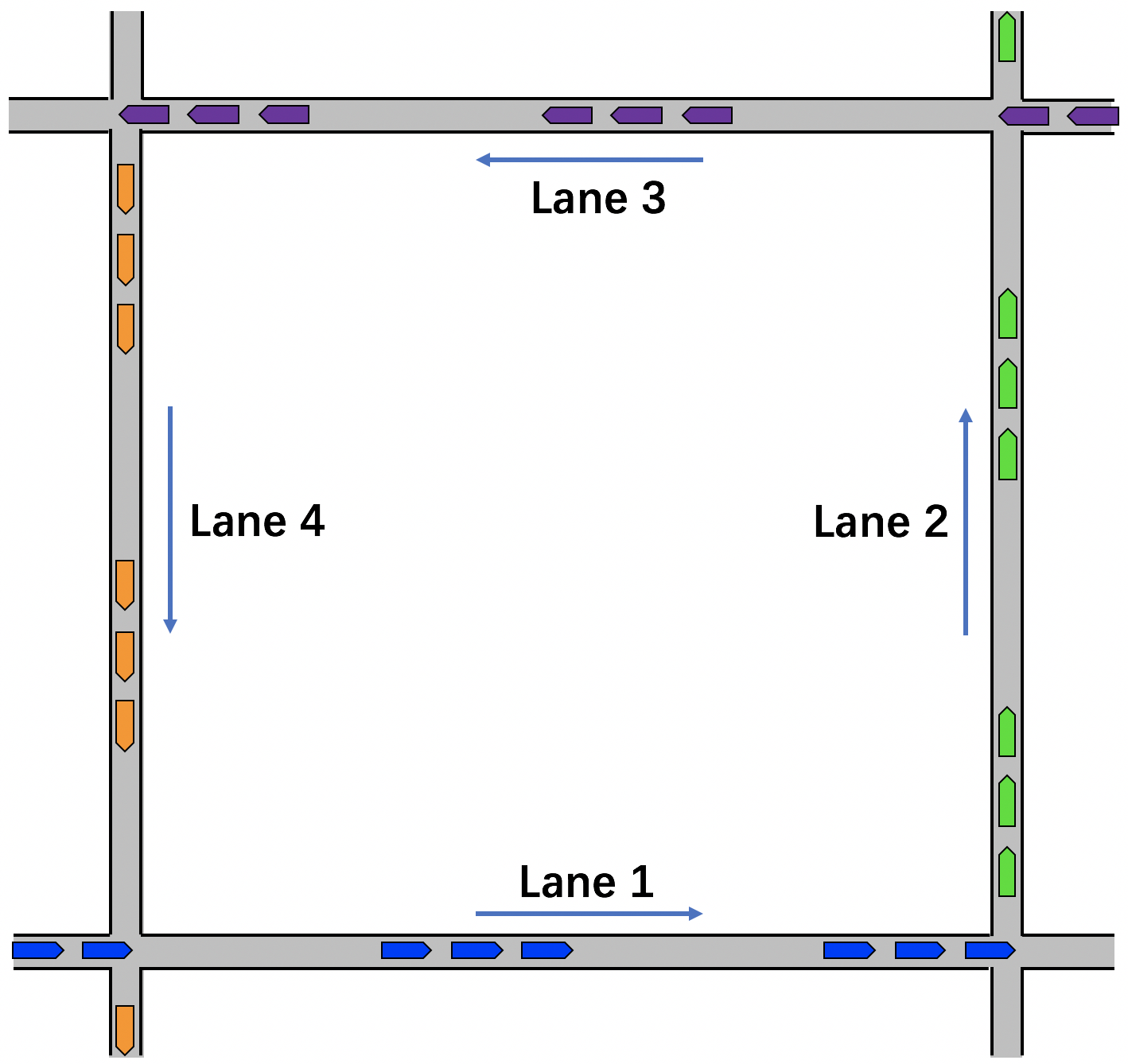}} \hspace{1em}
	\subfloat[][]{\includegraphics[width=0.4\textwidth]{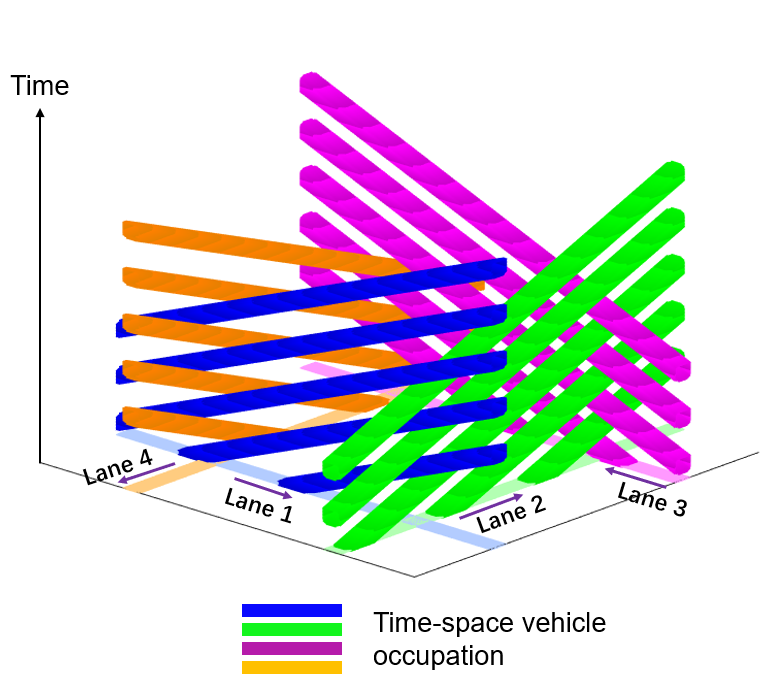}}
	\caption[]{Four-lane network. (a) Network layout. (b) Time-space vehicle occupations. } 
	\label{Fourlane_fig}
\end{figure}
\par

On a one-way grid network, and given the network rhythm $\hat{t}$, the lower-level RC framework works as follows. Within each interval of $\hat{t}$, each incoming vehicle at an entrance receives an order from the CCMS to be platooned with nearby vehicles, or to wait in the temporary waiting lane for the next platooning order. Once a virtual platoon is formed at an entrance of the network, it moves forward to the exit on the same street. Within its presence on the network, there could be new vehicles joining the platoon in succession, either from junctions or other perpendicular streets. Similarly, vehicles could also leave the platoon for other junctions or perpendicular streets. The number of vehicles in one virtual platoon has a uniform upper bound $\bar{N}_p$, and the online routing protocol, introduced in the latter section, will ensure that the vehicle number in each virtual platoon is always within the upper bound. The vehicles with turning movements in their routes will switch from one platoon to another until their trips are finished. Figures \ref{NetRhy}(a)-\ref{NetRhy}(b) graphically illustrate the virtual platoons and vehicle movements under the proposed RC, in which Figure \ref{NetRhy}(a) demonstrates two consecutive rhythms with the newly entered virtual platoons labelled, and \ref{NetRhy}(b) shows one possible vehicle movement of passing through two virtual platoons (both labelled in the pictures).

\begin{figure}[!ht]
	\centering
	\subfloat[][]{\includegraphics[width=0.8\textwidth]{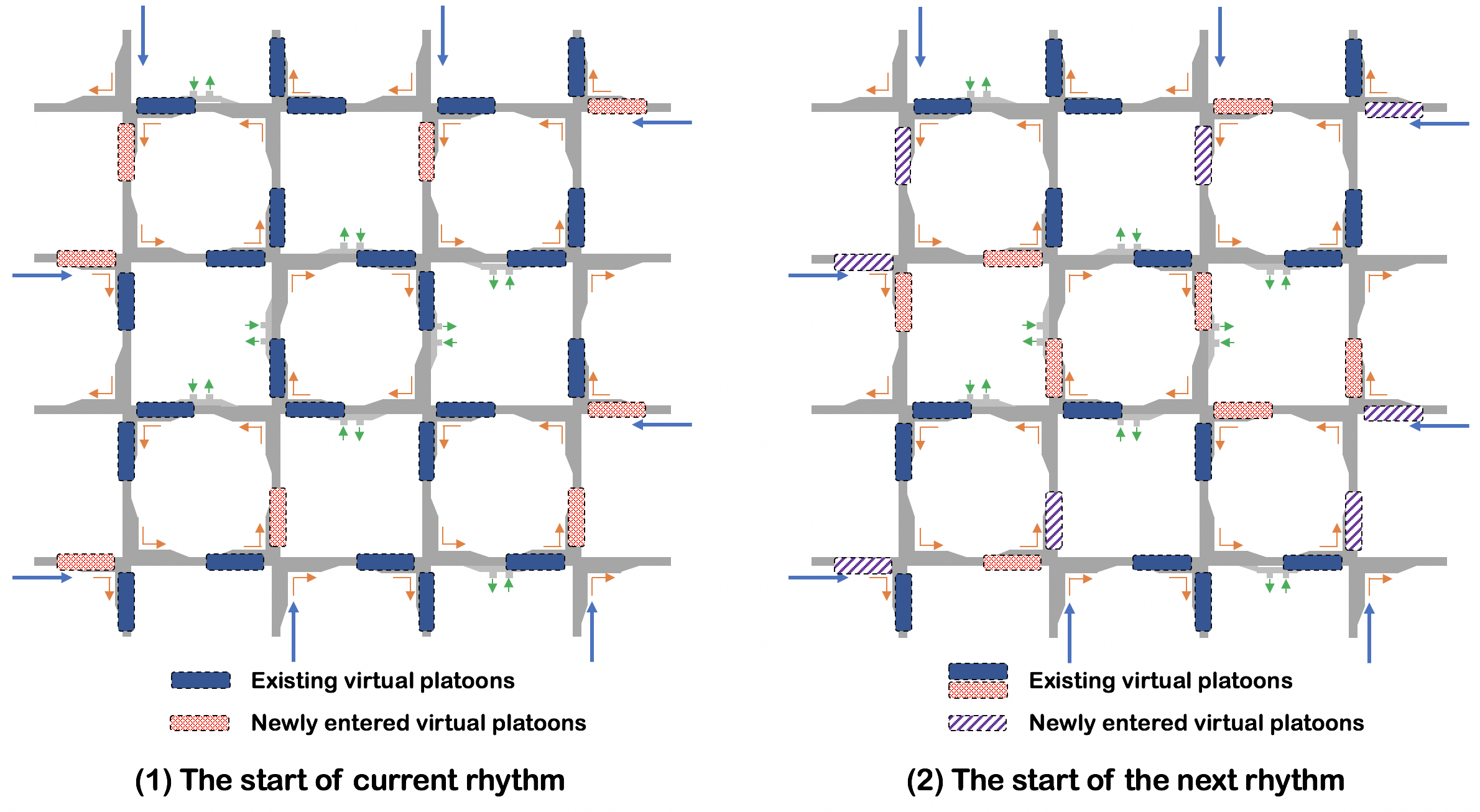}}\\
	\subfloat[][]{\includegraphics[width=0.8\textwidth]{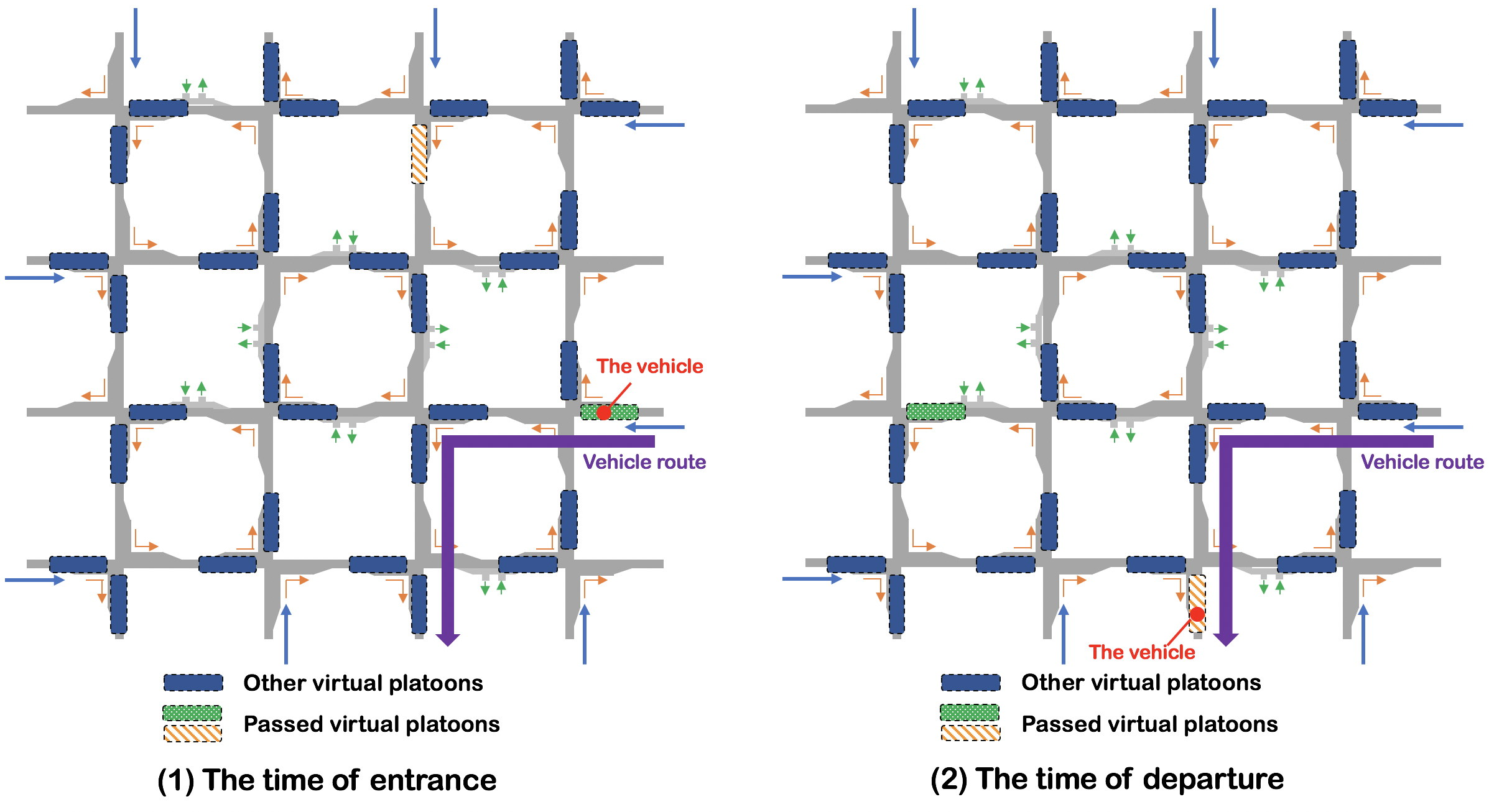}}
	\caption[]{Lower-level rhythmic control. (a) Network rhythms and virtual platoons. (b) Vehicle movements. } 
	\label{NetRhy}
\end{figure}

\begin{figure}[!ht]
	\centering
	\subfloat[][]{\includegraphics[width=0.45\textwidth]{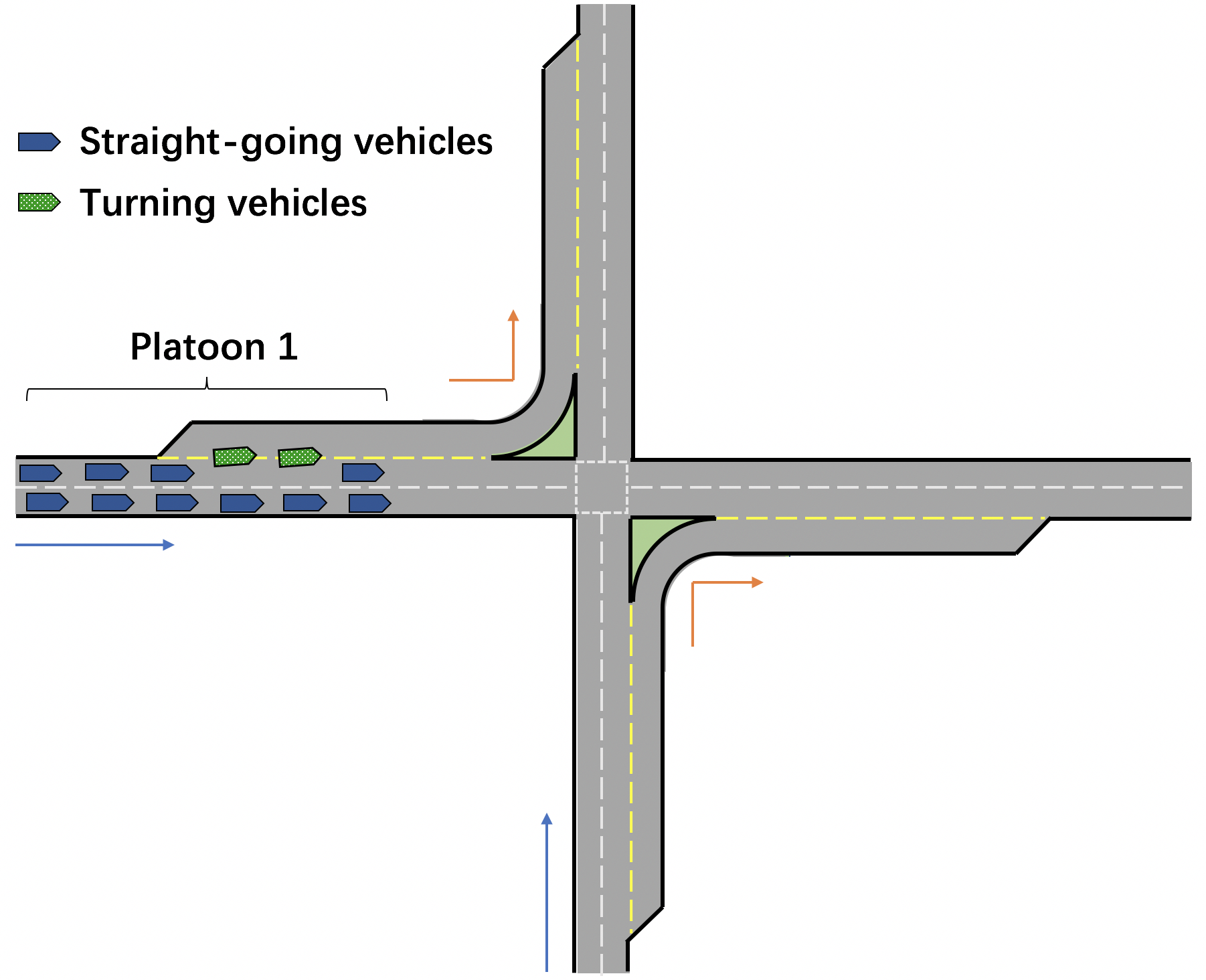}}\\
	\subfloat[][]{\includegraphics[width=0.45\textwidth]{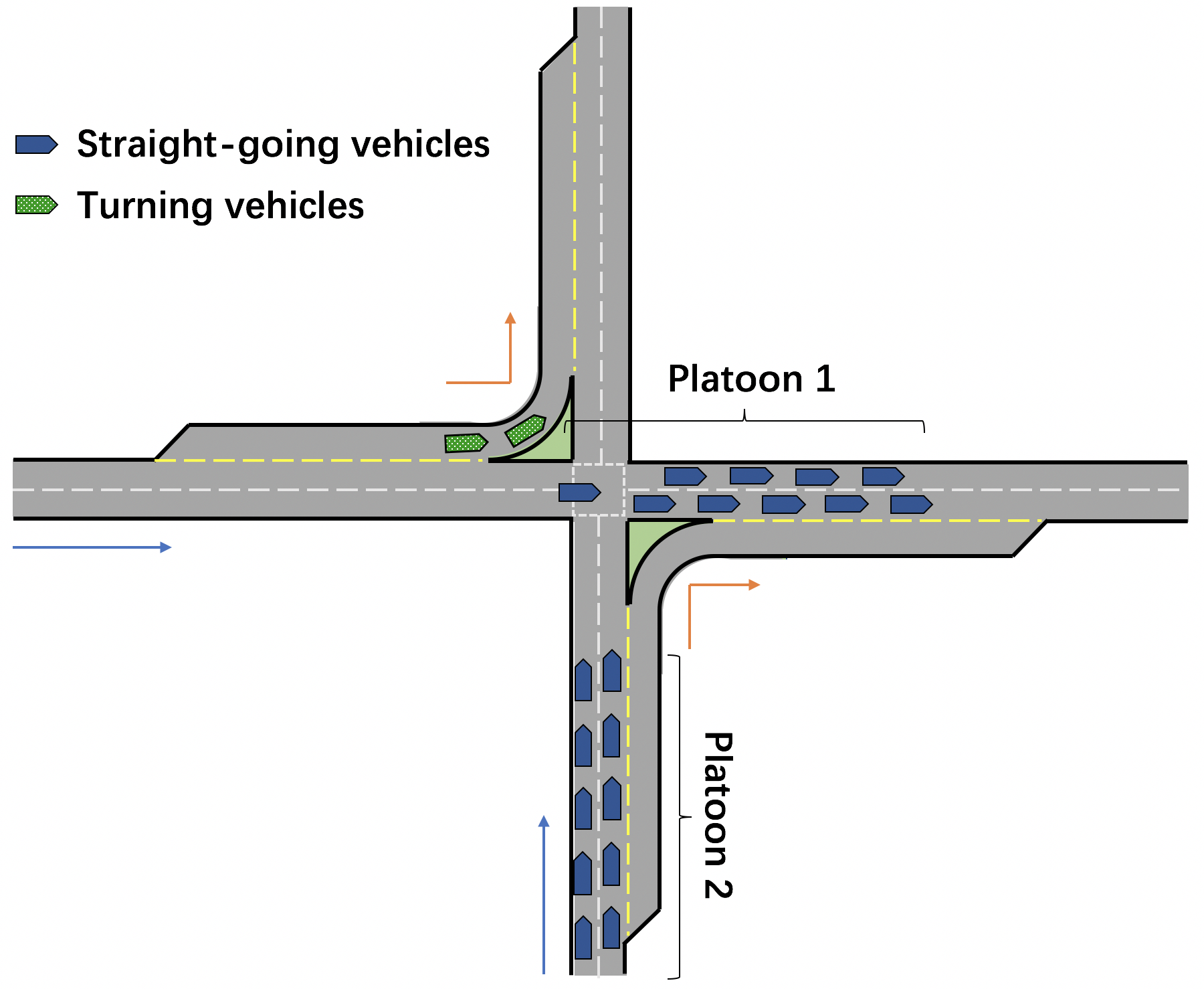}}\\
	\subfloat[][]{\includegraphics[width=0.45\textwidth]{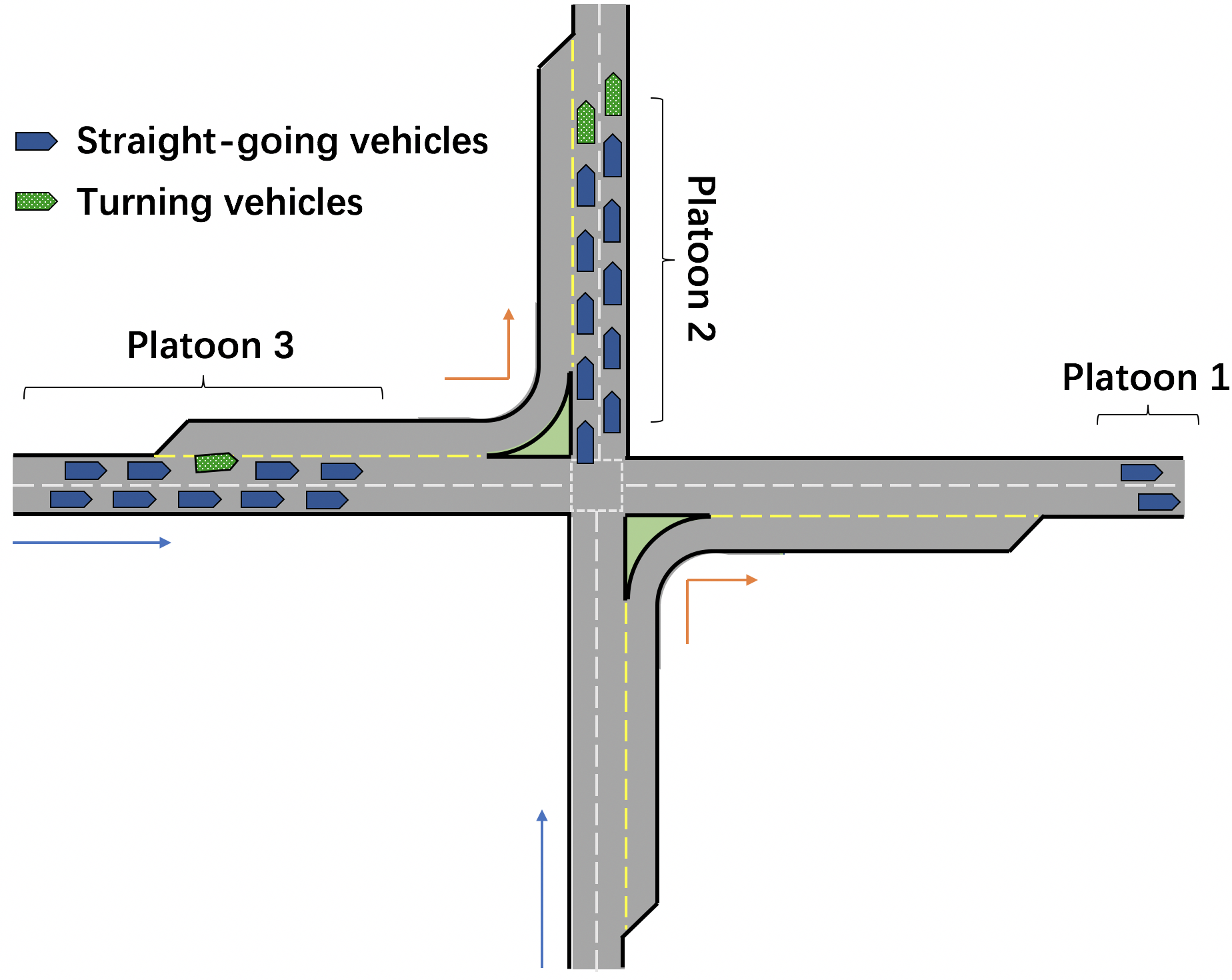}}
	\caption[]{Process of turning at crossroads. (a) Leaving the current platoon. (b) Turning. (c) Joining the next platoon. } 
	\label{Turning}
\end{figure}

Next, we elaborate on the turning processes of CAVs at the crossroads of a one-way grid network. As shown in Figure \ref{Turning}, vehicles preparing to turn at a crossroads first leave their current virtual platoon for the turning lane (Figure \ref{Turning}(a)), then decelerate to pass through the turning lane (Figure \ref{Turning}(b)), and finally accelerate after turning to join the next virtual platoon on the target street (Figure \ref{Turning}(c)). As the virtual platoons from different streets arrive at the crossroads in an alternating fashion, there is a time difference $\hat{t}$ between two consecutive platoons, which allows the turning vehicles to decelerate and then accelerate. \par

For the turning movements, one additional point worth mentioning regards mandatory lane changing, which could happen when the number of lanes on a street is larger than one. For example, in a two-lane case as in Figure \ref{Turning}, if a vehicle in virtual platoon 1 driving on the right lane plans to turn left at the crossroads, then a mandatory lane changing to the left lane is required before reaching the crossroads. As vehicles may constantly join or leave the platoon on the network, and because different vehicles may have different target turning directions, a control module is needed to dynamically arrange the lane changing maneuvers to satisfy requests from all vehicles. Such microscopic dynamic maneuvers can be integrated into a platoon control module. As this study focuses on the macroscopic design of a control framework that facilitates conflict-free online routing for CAVs, we leave the microscopic platoon control module to future investigation. 

\subsection{Online routing} \label{routing_subsec}

In a one-way grid network, and based on the pre-set network rhythm, the upper-level RC framework will conduct online routing for CAVs. In this section, we propose two methods for the online routing to minimize the average vehicle delay. As mentioned in the previous section, virtual platoons are formed one after another with interval $\hat{t}$ at entrances. Hence, in our proposed online routing model, we assume a routing interval equal to the network rhythm, and within each routing interval, all vehicles waiting at entrances and junctions will either be assigned a route or will be kept waiting for another routing interval in temporary waiting lanes (or in the waiting zone around junctions). Once assigned a route, the virtual platoons that a vehicle will join/leave are also determined. \par

To model an $m \times n$ one-way grid network where $m$ and $n$ are even numbers, we construct a directed graph $\mathcal{G}(\mathcal{N}, \mathcal{A})$ where $\mathcal{N}$ is the set of all nodes and $\mathcal{A}$ is the set of all links. A link $a \in \mathcal{A}$ represents a road connecting two crossroads or an entrance and a crossroad. Let $\mathcal{W}$ represent the set of all O-D pairs and $\mathcal{E}$ represent the set of all routing intervals. To simplify the notation without loss of generality, we assume that $t_c = \hat{t}$ when deriving the online routing model formulations. Then, based on Proposition \ref{prop2}, in the interval $e \in \mathcal{E}$, there is only one virtual platoon on each link of the network. Accordingly, we can denote a particular virtual platoon by a pair of interval number and link number $(a, e), a \in \mathcal{A}, e \in \mathcal{E}$; we also refer to $(a, e)$ as a temporal link. Under the proposed RC, CAVs will join different virtual platoons to complete their trips. Therefore, a CAV’s path can be defined as a collection of virtual platoons, and we let $\delta_{a,r}^{e,w}$ represent the virtual-platoon-path incidence, which equals one if path $r \in \mathcal{R}^w$ connecting O-D pair $w$ includes virtual platoon $(a, e)$, and zero otherwise. Because the network rhythm is fixed and each virtual platoon’s entry time and travel times are both fixed based on Proposition \ref{prop2}, for an routing interval $e$ lasting within $[t_e, t_e + \hat{t})$, one can go through the following steps to calculate $\delta_{a,r}^{e,w}$ for any path $r$ starting in this interval. \par

\noindent \textbf{Step 1}: In routing interval $[t_e, t_e + \hat{t})$, for a given path $r$, obtain the preset entry time at its originating entrance or junction, i.e., $t_r^s$; \\
\textbf{Step 2}: For the first link on path $r$, i.e., $a_1$, compute the arrival time by $t_1 = t_r^s$; for all other links in path $r$, i.e., $(a_2, a_3, ..., a_k)$, where the permutation of $a_i, 2 \le i \le k$, is by the sequence from origin to destination, compute the arrival time by $t_{i+1} = t_i + \tau_{a_i}$, where $\tau_{a_i}$ is the travel time on link $a_i$; \\
\textbf{Step 3}: For each link $a_i, 1 \le i \le k$, if the arrival time is in routing interval $e_i$ with duration $[t_{e_i}, t_{e_i} + \hat{t})$, then $\delta_{a_i,r}^{e_i,w}=1$, and for any $\bar{e} \ne e_i$ we have $\delta_{a_i,r}^{\bar{e},w} = 0$. For any link $a$ that path $r$ does not traverse, we have $\delta_{a,r}^{\bar{e},w} = 0, \forall \bar{e} \in \mathcal{E}$.

\subsubsection{Formulations and algorithms} \label{SPR_subsubsec}

We first introduce the basic routing protocol, i.e., shortest-path routing (SPR), which requires only a modest computational cost. In each routing interval, the SPR either assigns vehicles to one shortest path or keeps them waiting for a current interval; as there is no alternative route for vehicles of the same O-D pair, SPR can be treated as a protocol for assigning limited entrance permission to each vehicle. The constraint is that the number of vehicles in a platoon at any moment cannot exceed its preset upper bound. Let $\bar{\mathcal{R}}$ denote the set of shortest paths for all O-D pairs. For the routing interval $\bar{e} \in \mathcal{E}$, the SPR problem is formulated as follows: 

\begin{align}
\textbf{SPR} \nonumber \\
& \min_{\mathbf{f} \in \mathbb{N}^{\abs{\bar{\mathcal{R}}}}} \sum_{w \in \mathcal{W}} (\bar{d}^w - f_r^w)\hat{c}^w \nonumber \\
& \textrm{s.t.} \quad \sum_{w \in \mathcal{W}} \delta_{a,r}^{e,w} f_r^w \le N_a^e & \forall a \in \mathcal{A}, e \in \mathcal{E} \label{CapCons_SPR}\\
&  f_r^w \le \bar{d}^w & \forall r \in \bar{\mathcal{R}}, w\in \mathcal{W} \label{FlowCons_SPR}
\end{align}
\par

In SPR, $f_r^w$ is a decision variable representing the traffic flow on the shortest path  belonging to O-D pair $w \in \mathcal{W}$, and they must be non-negative integers; $\bar{d}^w$ is the travel demand for O-D pair $w \in \mathcal{W}$ in this routing interval, and includes both newly-arrived CAVs and CAVs kept waiting in temporal waiting lanes in previous intervals; $\hat{c}^w$ represents the cost that a CAV of O-D pair $w \in \mathcal{W}$ incurs if it is forced to wait for the current routing interval; and $N_a^e$ is the remaining room for the virtual platoon that passes through link $a \in \mathcal{A}$ within routing interval $e \in \mathcal{E}$. We note that $N_a^e$ is not always equal to the upper bound of platoon capacity $\bar{N}_p$, because at routing interval $e \in \mathcal{E}$, some platoons on some links are already “reserved” by vehicles entering before $e$. The objective function of SPR is to minimize the total cost of delayed vehicles from all O-D pairs. Constraint (\ref{CapCons_SPR}) requires that no virtual platoon can support a number of vehicles more than the preset upper bound. Constraint (\ref{FlowCons_SPR}) implies that the number of vehicles of O-D pair $w \in \mathcal{W}$ permitted to enter in routing interval $\bar{e} \in \mathcal{E}$ cannot exceed the associated demand. On a grid network, there could be multiple shortest paths between an O-D pair, and it is inefficient to use the same shortest path set in every network rhythm. One straightforward technique to alleviate this issue is to randomly select one shortest path for each O-D pair in each routing interval to avoid the undesirable case that certain paths are repetitively utilized. \par

Here, we briefly discuss the value choice for $\hat{c}^w$. As our routing protocol works in every routing interval, and all vehicles kept waiting at interval $\bar{e}$ will be rerouted at interval $\bar{e}+1$, then it is intuitive to set $\hat{c}^w = \hat{t}$. However, this setting may result in an unsatisfactory consequence: vehicles in some specific O-D pairs may be constantly delayed, because allowing them to enter is not beneficial for the minimization of SPR. To alleviate this defect, in a practical implementation we can alter the expression of the delay penalty as $\hat{c}^w = (1 + l^w)\hat{t}$, where $l^w$ is a number representing that for $l^w$ consecutive routing intervals, the queue of this O-D pair has not been cleared. With such an alteration, vehicles in those O-D pairs become “more important,” and they will be assigned entry permissions with higher priority. \par

As formulated, SPR is an integer program (IP); when the problem scale is large, exactly solving an IP could be time-consuming, and we require the solution process to be quick enough to be implemented in practical situation. Therefore, we propose a simple approximation algorithm to solve SPR.  Consider the following approximation algorithm (AA-SPR): \par

\noindent \underline{\textbf{AA-SPR}} \\
\textbf{Step 1}: Eliminate all irrelevant variables and constraints in SPR to obtain an equivalent problem of smaller size, denoted by SPR*; \\
\textbf{Step 2}: Solve the LP-relaxation of SPR* to obtain a solution $\tilde{\mathbf{f}}$; \\
\textbf{Step 3}: If $\tilde{\mathbf{f}}$ is integral, then set it as the optimal integer solution and terminate AA-SPR; otherwise, find the “most fractional” element in $\tilde{\mathbf{f}}$, i.e., $\max_{r \in \bar{\mathcal{R}}} \min (\tilde{f}_r - \lfloor \tilde{f}_r \rfloor, \lceil \tilde{f}_r \rceil - \tilde{f}_r)$, adding a constraint $f_r \le \lfloor \tilde{f}_r \rfloor$ if $\tilde{f}_r - \lfloor \tilde{f}_r \rfloor \le \lceil \tilde{f}_r \rceil - \tilde{f}_r$, and $f_r \ge \lceil \tilde{f}_r \rceil$ otherwise in SPR*, and return to Step 2. \par

We note that the coefficient matrix of Constraints (\ref{CapCons_SPR})-(\ref{FlowCons_SPR}) of SPR has $\abs{\bar{\mathcal{R}}}$ columns (equal to the size of O-D pairs) and $\abs{\mathcal{E}} \abs{\mathcal{A}} + \abs{\bar{\mathcal{R}}}$ rows, where $\abs{\mathcal{E}} \abs{\mathcal{A}}$ is the row size of Constraint (\ref{CapCons_SPR}), and $\abs{\bar{\mathcal{R}}}$ is the row size of Constraint (\ref{FlowCons_SPR}). However, the relevant row size in Constraint (\ref{CapCons_SPR}) is generally much smaller than $\abs{\mathcal{E}} \abs{\mathcal{A}}$, because most platoons are not affected by entering vehicles in the current routing interval; the relevant column size is also generally smaller than $\abs{\bar{\mathcal{R}}}$, because most O-D pairs may exhibit no demand at the current routing interval. Thus, in step 1 of AA-SPR, we eliminate all irrelevant columns and rows in Constraints (\ref{CapCons_SPR})-(\ref{FlowCons_SPR}), and denote the resulting coefficient matrix as $\mathbf{C}_{SP}$. We also note that the steps 2-3 of AA-SPR require rounding some fractional elements in the current solution to the nearest integer, and solving a series of restricted LP-relaxations. In practical implementation, this rounding procedure can be largely accelerated, as the optimal LP-relaxation solution of a new restricted problem can be obtained using a few simplex iterations from the solution obtained in the previous step (e.g., Balev et al.,
2008). Additionally, AA-SPR gives a natural optimality gap measure; the objective function value of the first round of LP-relaxation is a lower bound of the SPR, and the final solution is an upper bound. We can utilize the above bounds to evaluate the associated solution quality. \par

The previously-introduced SPR protocol only utilizes one shortest path for each O-D pair in each routing interval. When the level of demands on a network becomes high, there is a need to activate additional "non-shortest" paths to increase the capacity of the network, at the cost of increased computational time. This section provides the model and solution algorithm for a multiple-path routing (MPR) protocol, which allows the system to assign vehicles to multiple paths between the associated O-D pair. \par

In MPR, we denote the set of eligible paths between O-D pair $w \in \mathcal{W}$ as $\mathcal{R}^w_l$, and $\mathcal{R}_l = \cup_{w \in \mathcal{W}} \mathcal{R}^w_l$. When constructing eligible path sets, we define $\hat{c}_m$ as a preset detour limit, such that only paths satisfying $c_r^w \le \hat{c}_m$ are stored. The choice of $\hat{c}_m$ is a trade-off between computational efficiency and the capability of handling heavy traffic; setting $\hat{c}_m = \min_{r \in \mathcal{R}^w_l} c_r^w$ indicates that we only allow the use of all shortest paths \footnote{Note that this case is not the same as SPR, since the latter only allows one shortest path for each O-D pair.}, while setting $\hat{c}_m$ as a large number suggests we allow long detours, even containing loops. Other notations of sets and incidence factors are the same as those in the SPR. The mathematical model of MPR is then formulated as: \\
\begin{align}
\textbf{MPR} \nonumber \\
& \min_{\mathbf{f, \hat{f}} \in \mathbb{N}} \sum_{w \in \mathcal{W}} \left( \sum_{r \in \mathcal{R}^w_l} c_r^w f_r^w + \hat{c}^w \hat{f}^w \right) \nonumber \\
& \textrm{s.t.} \quad \sum_{w \in \mathcal{W}} \sum_{r \in \mathcal{R}^w_l} \delta_{a,r}^{e,w} f_r^w \le N_a^e & \forall a \in \mathcal{A}, e \in \mathcal{E} \label{CapCons_FPR}\\
&  \hat{f}^w + \sum_{r \in \mathcal{R}^w_l}f_r^w  = \bar{d}^w & \forall w\in \mathcal{W} \label{FlowCons_FPR}
\end{align}

\par

In the model MPR, there are two sets of decision variables: $f_r^w$ represents the flow on path $r \in \mathcal{R}^w_l$ between O-D pair $w \in \mathcal{W}$, and $\hat{f}^w$ represents the number of vehicles being kept waiting for this routing interval between O-D pair $w \in \mathcal{W}$; both the decision variables must be non-negative integers. In the objective function, $c_r^w$ is the travel time on path $r \in \mathcal{R}^w_l$, and $\hat{c}^w$ is the penalty of delay; the objective function is to minimize the total cost in this routing interval. Constraint (\ref{CapCons_FPR}) is a platoon capacity constraint, and is similar to Constraint (\ref{CapCons_SPR}) in SPR. Constraint (\ref{FlowCons_FPR}) states that the sum of the number of vehicles assigned to all eligible paths between O-D pair $w \in \mathcal{W}$ and those kept waiting should be equal to the associated demand of vehicles ready for entry. \par

As formulated, MPR is also an integer program. Similarly to SPR, to guarantee a real-time implementation, we also adopt an LP-relaxation-based approximation algorithm to solve MPR (AA-MPR). \par

\noindent \underline{\textbf{AA-MPR}} \\
\textbf{Step 1}: According to current $\hat{c}^w$, extract all paths such that $c_r^w \le \hat{c}_m$ to form the eligible path set $\mathcal{R}_l$, and formulate MPR based on $\mathcal{R}_l$; \\
\textbf{Step 2}: Solve the LP-relaxation of MPR to obtain a solution $(\mathbf{f}^{\circ}, \hat{\mathbf{f}}^{\circ})$; \\
\textbf{Step 3}: If $(\mathbf{f}^{\circ}, \hat{\mathbf{f}}^{\circ})$ is integral, then set it as the optimal integer solution and terminate AA-MPR; otherwise, find the “most fractional” element in $(\mathbf{f}^{\circ}, \hat{\mathbf{f}}^{\circ})$ as in AA-MPR, and suppose the element is $f_r^{\circ}$. Then, add a constraint $f_r \le \lfloor f_r^{\circ} \rfloor$ if $f_r^{\circ} - \lfloor f_r^{\circ} \rfloor \le \lceil f_r^{\circ} \rceil - f_r^{\circ}$ and $f_r \ge \lceil f_r^{\circ} \rceil$ otherwise in MPR. Return to Step 3.\par

\subsubsection{Power of LP-relaxation} \label{LPR_subsubsec}

In this subsection, we provide some insights and simulation test results on the effectiveness of AA-SPR and AA-MPR presented above. Specifically, we identify a mild condition that the LP relaxation leads to an integral optimal solution, and then demonstrate some specific examples for the embodiment of the mild condition. A Monte-Carlo simulation with AA-SPR is conducted to show that the LP-relaxation-based algorithm achieves superior solution quality. \par

\begin{defi} \label{defi_LS}
$\mathbf{C}_{SP}$ is called \textbf{locally separable} if for any $k$ different paths $r_1, r_2, ..., r_k$ (the set is denoted by $\mathcal{R}^k$) and $k$ different temporal links $(a_1,e_1), (a_2,e_2), ..., (a_k,e_k)$ (the set is denoted by $\mathcal{A}_{\mathcal{E}}^k$), one of the following three conditions hold: \\
(i) Two paths in $\mathcal{R}^k$ traverse the same set of temporal links in $\mathcal{A}_{\mathcal{E}}^k$; \\
(ii) Two temporal links in $\mathcal{A}_{\mathcal{E}}^k$ are traversed by the same set of paths in $\mathcal{R}^k$;\\
(iii) There exists a $k_0 \in \{ 1,2,...,k-1 \}$ such that there are $k_0$ paths in $\mathcal{R}^k$ never traversing $k - k_0$ temporal links in $\mathcal{A}_{\mathcal{E}}^k$.
\end{defi}

\begin{prop}\label{prop3}
If $\mathbf{C}_{SP}$ is locally separable, then the LP-relaxations of SPR and MPR have integral optimal solutions.
\end{prop}

\begin{proof}
See Appendix A.
\end{proof}
\par

The locally-separable property actually states that, for any $k$ paths and $k$ temporal links, either there are two paths or two temporal links that are “locally identical,” or we can “separate” a portion of those paths with another portion of temporal links, such that they are uncorrelated. We describe some special cases where the locally-separable property holds below. \par

\begin{itemize}
    \item \textit{No paths in set $\bar{\mathcal{R}}$ are interconnected}. In such a case, any temporal link is passed through by at most one path, so the locally-separable property holds. We should point out that in SPR, the path-interconnection is much scarcer than in a static network, because two interconnected paths should not only pass through the same link, but also arrive at it at the same time.

    \item \textit{All paths are straight}. In such a case, no turning movement exists, and vehicles entering from different origins will never cross; it is not difficult to prove that the locally-separate property holds under such a setting. In reality, this requirement may be difficult to achieve, but in many cases, straight traffic streams form the majority of network traffic, and a minor proportion of the turning movements is likely to maintain the locally-separable property.

\begin{figure}[!ht]
	\centering
	\subfloat[][]{\includegraphics[width=0.4\textwidth]{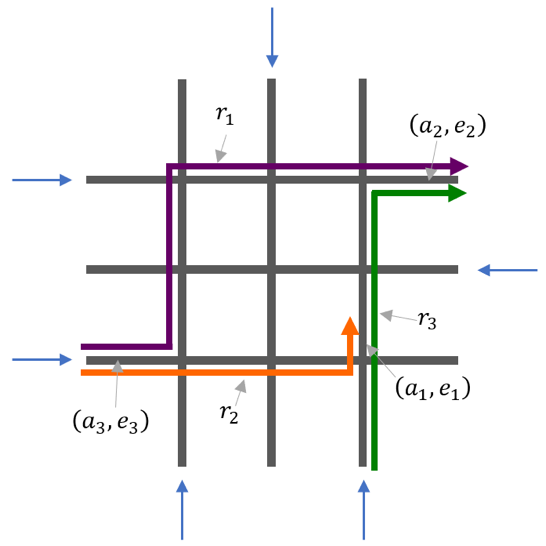}} \hspace{1em}
	\subfloat[][]{\includegraphics[width=0.4\textwidth]{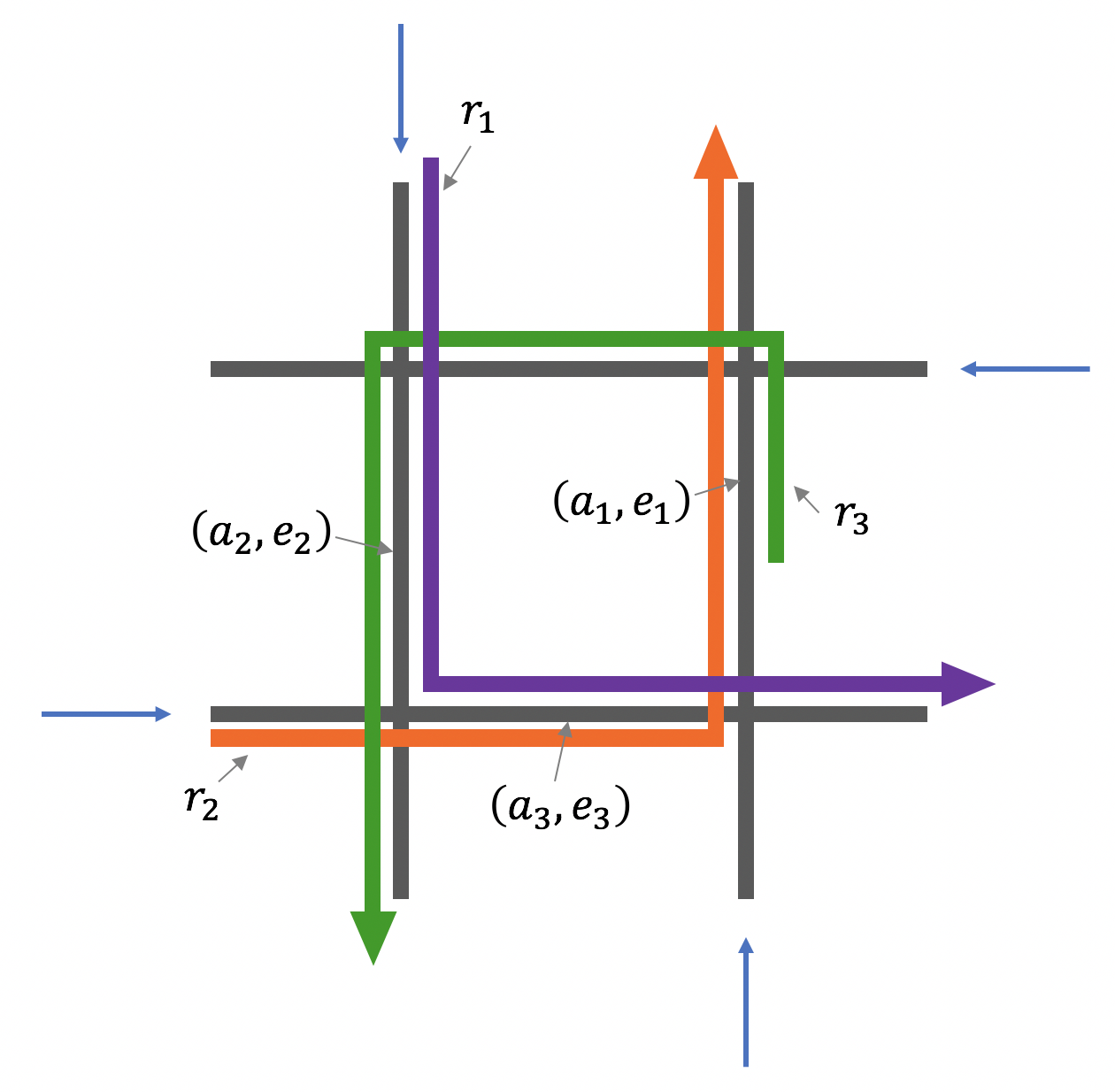}}
	\caption[]{(a) Interconnection loop. (b) No interconnection loop.} 
	\label{NonInt}
\end{figure}
\par
    
     \item \textit{No interconnection loop exists}. An interconnection loop is defined as a set of paths $r_1,r_2,...,r_k, k\ge3$ such that $r_i$ and $r_{i+1}$ are interconnected, $i\in{1,2,...,k}, \ r_{k+1}=r_1$. The non-existence of an interconnection loop is actually equivalent to a case where any square sub-matrix of $\mathbf{C}_{SP}$ with a size larger than two has at least one row that has at most one non-zero element (there exists one temporal link crossed by at most one path).This satisfies condition (iii) in the definition of a locally-separable property. Here, we point out that the existence of an interconnection loop actually contributes to the non-integrality of LP-relaxation solution (if exists), and a simple example is shown below. Consider three paths $(r_1,r_2,r_3)$ as shown in Figure \ref{NonInt}(a), where $r_1$ and $r_2$ intersect at temporal link $(a_3,e_3)$, $r_2$ and $r_3$ intersect at temporal link $(a_1,e_1)$, and $r_3$ and $r_1$ intersect at temporal link $(a_2,e_2)$; clearly, $r_1,r_2,r_3$ forms an interconnection loop, and it does not satisfy the locally-separable property. If we extract those three paths and three temporal links, the associated square sub-matrix is $\left[ \begin{matrix} 0 & 1 & 1 \\ 1 & 0 & 1 \\ 1 & 1 & 0 \end{matrix} \right]$, and the determinant is 2, i.e., it is not unimodular. Assuming that the demands on the three paths are one, the capacities on all temporal links are one, and the objective is to maximize the number of entered vehicles, it is straightforward that only one vehicle is allowed to enter, but for the LP-relaxation of the problem, the best solution is $(f_1,f_2,f_3) = (0.5,0.5,0.5)$, resulting in a total flow of 1.5. Thus, there is no integral optimal solution for the LP-relaxation. Here we also note that, unlike Figure \ref{NonInt}(a), the case in Figure \ref{NonInt}(b) does not form an interconnection loop; for it, if $r_1$ and $r_2$ intersect at $(a_3,e_3)$, $r_2$ and $r_3$ intersect at $(a_1,e_1)$, then $r_3$ and $r_1$ cannot intersection at $(a_2,e_2)$ because the vehicles on $r_1$ have already passed through link $a_2$ when the vehicles on $r_3$ arrive at $a_2$. This example suggests that the existence of interconnection loop in RC routing is substantially more limited than in an ordinary multi-commodity flow problem due to the presence of time horizon.
\end{itemize}

The rare existence of interconnection loop in the RC supports the effectiveness of the LP-relaxation because the locally-separable property is most likely satisfied. To further demonstrate it, we numerically calculate the existence probability of an interconnection loop in an arbitrary gird network. As one can imagine, this work is cumbersome as the possible interconnection loops can be of various forms. As the number of paths to form an interconnection loop increases, it becomes rather difficult to calculate its existence probability. The following table shows the existence probability of an interconnection loop constituted by three paths on grid networks with sizes of no more than $10 \times 10$. The detailed computational procedure is shown in Appendix B.

\begin{table}[!ht]
  \centering
  \caption{Existence of interconnection loops with three paths}
  \label{Loop_table}
  \footnotesize
  \begin{tabular}{ l  c  c  c  c  c}
   \hline
		Network size	& $2 \times 2$ & $4 \times 4$ & $6 \times 6$ & $8 \times 8$ & $10 \times 10$ \\
		\hline
		Number of shortest paths   & 60 & 1,000 & 5,124 & 16,272 & 39,820  \\
		Number of 3-path combinations   & $3.42 \times 10^4$ & $1.66 \times 10^8$ & $2.24 \times 10^{10}$ & $7.18 \times 10^{11}$ & $1.05 \times 10^{13}$ \\
		Number of 3-path interconnection loops  & 0 & 0 & 1,472 & 42,568 & 389,324 \\
		Probability of 3-path interconnection loops & 0 & 0 & $6.57 \times 10^{-8}$ & $5.93 \times 10^{-8}$ & $3.70 \times 10^{-8}$ \\
   \hline
  \end{tabular}
\end{table}
\par

From Table \ref{Loop_table}, we observe that the existence of a 3-path interconnection loop is rather limited, i.e., if we randomly extract three different paths from the set of the shortest paths, then the probability of forming an interconnection loop is no more than $10^{-7}$ (one over ten million); moreover, a network with a size of no more than $4 \times 4$ contains no interconnection loop constituted by three paths. Considering that a loop with more than three paths is more difficult to be formed, we thus draw an empirical conclusion that the locally-separable property is ubiquitous, especially in SPR. Furthermore, even when a case violating the property occurs, it is possible that the associated constraints are non-binding, i.e., it does not affect the integrality of the optimal solution to the LP-relaxation.  To further support the above analyses in the integrality of solution to the LP-relaxation, we conduct a thorough Monte-Carlo simulation with a $6 \times 6$ one-way street network to illustrate the solution quality of AA-SPR; the details are given in Appendix B. Among 10,000 simulation tests, we observe that the probability that the LP-relaxation of SPR directly gives an optimal integral solution is $99.86\%$. For the remaining $0.14\%$ of cases, we find that the objective function value gap between the solution to the LP-relaxation and the rounded integral solution is less than $0.02\%$. 

\section{Some extensions} \label{Extension}

In this section, we discuss some extensions to the RC for CAV traffic on one-way grid networks. Those extensions are aimed at handling various traffic and infrastructure conditions, so as to improve the usability and robustness. 

\subsection{Buffer and choice of rhythm length} \label{Buffer_subsec}

In reality, especially in near future when the CAV technology is not mature enough, the control should always be able to endure some level of uncertainties in vehicle movements. For RC, such uncertainties imply that some vehicles may go beyond the allowable space-time ranges provided by virtual platoons, and consequently induce potential collision risks at crossroads. For managing such risks, we introduce a "buffer" in a virtual platoon. The concept is graphically illustrated in Figure \ref{Buffer_fig}. As we can see, parts of the head and the tail of a virtual platoon are set as the buffer zones, and only the middle part (called valid zone) is allowed to contain vehicles in normal situations. When the platoon is running, if the vehicles experience some disturbances in speed, then they could go beyond the valid zone and fall into the buffer zones; but falling into the buffer zones will not cause collision risks at crossroads as the whole virtual platoon does not have space-time interference with other virtual platoons based on Proposition \ref{prop2}. Furthermore, for ensuring safety, some monitoring facilities should be located at roadsides; once identifying some vehicles falling into the buffer zones, the CCMS should send orders to them and push them back to the valid zone.

\begin{figure}[!ht]
	\centering
	\includegraphics[width=0.65\textwidth]{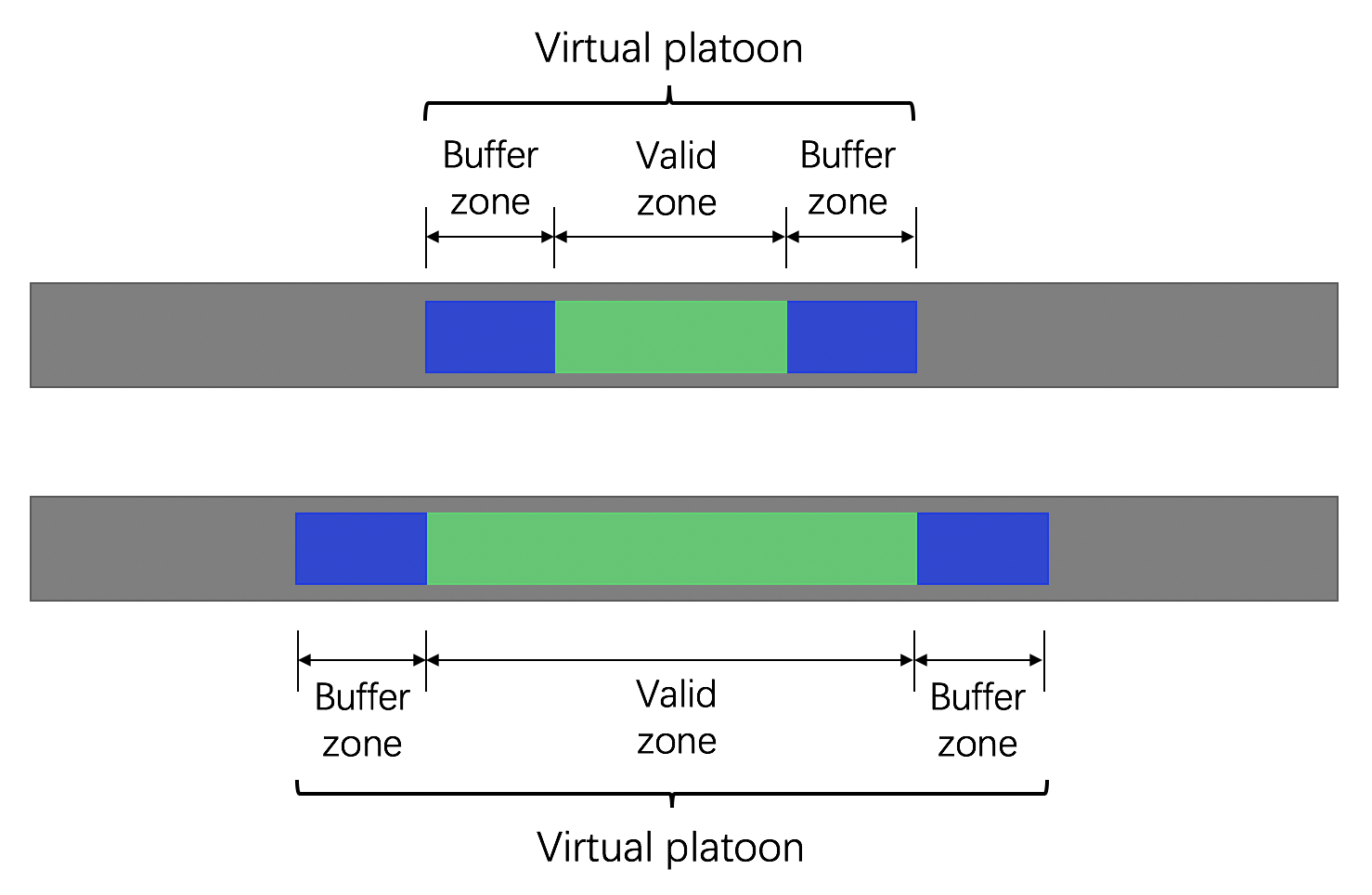}
	\caption[]{Illustration of buffers in virtual platoons.}
	\label{Buffer_fig}
\end{figure}

\begin{figure}[!ht]
	\centering
	\includegraphics[width=0.5\textwidth]{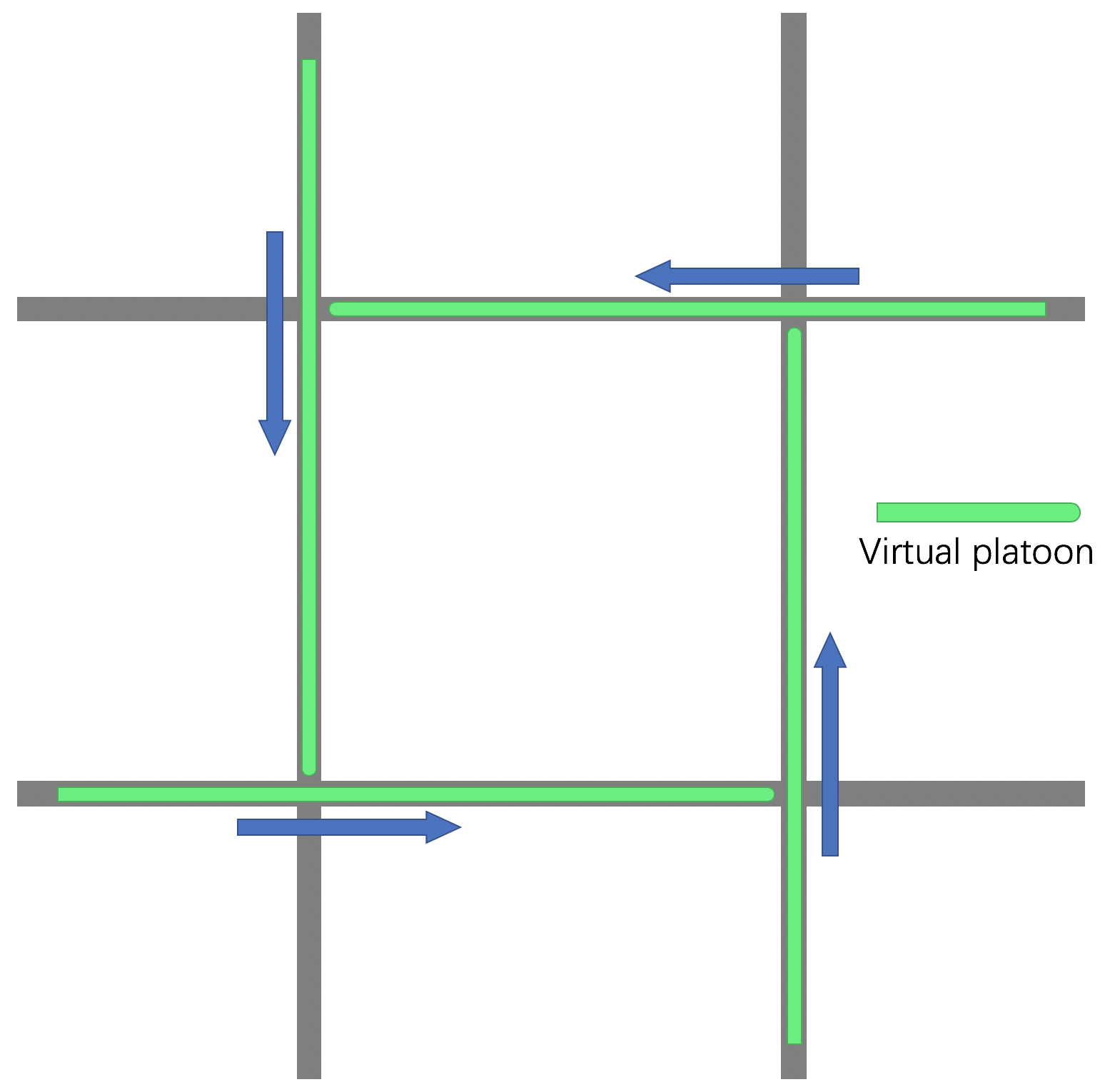}
	\caption[]{The paradox loop with excessive long virtual platoons.}
	\label{Paradox_fig}
\end{figure}
\par

Now consider the choice of network rhythm length. From Figure \ref{Buffer_fig} we can observe that, given a fixed length of buffer zone, the longer a virtual platoon is, the larger the proportion of the valid zone is; and a larger proportion of valid zone in each virtual platoon suggests a larger network traffic capacity. Meanwhile, a longer platoon requires a longer network rhythm to accommodate. Thus, it is intuitive that setting a longer network rhythm can accommodate higher travel demands. However, it is not beneficial to set up an excessively long network rhythm for expanding the service capacity due to two reasons: i) by the mechanism of RC, a long network rhythm incurs long average waiting times at entrances or junctions; and ii) excessive long virtual platoons could form a "paradox loop", which leads to gridlocks (see Figure \ref{Paradox_fig}). Therefore, the length of network rhythm is a trade-off between network capacity and vehicle delay at entrances or junctions.

\par
Here we provide a simple method for choosing a usable network rhythm length given the demand information. Suppose that the static demand rate between O-D pair $w \in \mathcal{W}$ is $\tilde{d}^w$, and $\mathcal{R}^w$ is the set of all paths connecting O-D pair $w$. We denote $\mathcal{H}$ as the set of all optional rhythm lengths, and for each $\hat{t} \in \mathcal{H}$, we denote $p_{\hat{t}}$ as the proportion of length of the valid zone in a virtual platoon; $p_{\hat{t}}$ should increase with $\hat{t}$. On each link $a \in \mathcal{A}$ on the network, the maximum allowable rate of vehicle throughput is $\frac{1}{2} l_a y_{m}$, where $l_a$ is the number of lanes on this link, and $y_{m}$ is the single-lane saturated flow rate; the term $\frac{1}{2}$ is due to the fact that RC provides each link half of its full right-of-way. Therefore, when provided a network rhythm length $\hat{t} \in \mathcal{H}$, the capacity of this link is $C_{a, \hat{t}} = \frac{1}{2} p_{\hat{t}} l_a y_{m}$. Furthermore, we define $\Omega(\tilde{\mathbf{d}},\hat{t}) = \left\{ \tilde{\mathbf{f}} \left| \begin{matrix} \tilde{d}^w = \sum_{r \in \mathcal{R}^w} \tilde{f}_r^w & \forall w \in \mathcal{W} \\ \sum_{w \in \mathcal{W}} \sum_{r \in \mathcal{R}^w} \delta_{a,r}^{w} \tilde{f}_r^w \le \gamma C_{a, \hat{t}} & \forall a \in \mathcal{A} \\ \tilde{f}_r^w \ge 0 & \forall w \in \mathcal{W}, r \in \mathcal{R}^w \end{matrix} \right. \right\}$, where $\tilde{f}_r^w$ is the vehicle demands assigned on route $r$ between O-D pair $w \in \mathcal{W}$, $\delta_{a,r}^w$ is the incidence factor between route $r$ and link $a$, and $\gamma \in (0,1)$ is a factor of robustness; it represents the set of all feasible path flow assignments given the demand information $\tilde{\mathbf{d}}$ and the network rhythm length $\hat{t}$. Then, we can formulate the problem of determining a desirable rhythm length as:

\begin{align}
& \hat{t}^* = \min_{\hat{t} \in \mathcal{H}} \left\{ \hat{t} | \Omega(\tilde{\mathbf{d}},\hat{t}) \ne \emptyset \right\} \label{choice_rhythm_eq}
\end{align}
\par

Eq.(\ref{choice_rhythm_eq}) represents that, given the demand information $\tilde{\mathbf{d}}$, we find the minimal rhythm length $\hat{t}$ under which there exists at least one feasible path flow assignment. In other words, we seek to minimize the rhythm length as long as the capacity is sufficient to accommodate the current demands. Note that if for all $\hat{t} \in \mathcal{H}$ we have $\Omega(\tilde{\mathbf{d}},\hat{t}) = \emptyset$, then we choose $\hat{t}^* = \max_{\hat{t} \in \mathcal{H}} \hat{t}$.

\subsection{Imbalanced demands} \label{ImbaD_subsec}

We discuss operational strategies to alleviate the problem of imbalanced demands on existing one-way grid networks. In Section \ref{rhythm_subsec}, we stated that in RC, “each virtual platoon shares the same capacity”; in reality, to handle the imbalances of demands, such a requirement can be relaxed. One typical situation is where the vertical vehicle demands and horizontal vehicle demands are asymmetric. In this context, we can set different capacities for the vertical and horizontal platoons without violating the conflict-free guarantee. In particular, for the network rhythm $\hat{t}$, we set the time for horizontal platoons to pass through a fixed point as $\beta \hat{t}$, and the time for vertical platoons as $(1-\beta) \hat{t}$. We set the entry time for horizontal platoons as $k \hat{t}, \forall k \in \mathbb{N}$, and for vertical platoons as $(k+\beta) \hat{t}, \forall k \in \mathbb{N}$. Then, by a procedure similar to that used in the proof of Proposition \ref{prop2}, we can ensure that no platoons collide on the network. When $0 < \beta < 0.5$, we provide more capacity for vertical demands, and when $0.5 < \beta < 1$ the case is reversed. Graphical illustrations of these concepts are shown in Figures \ref{ImbaD_fig}(a) and \ref{ImbaD_fig}(b). We note that such maneuvers can change dynamically according to the variations of demand distribution. Thus, the proposed system can handle various patterns of demand in a dynamic environment. Note that the length of network rhythm also has impact on system performance in the imbalanced demand cases, and we can adopt the method proposed in the previous subsection with slight modifications to determine it.

\begin{figure}[!ht]
	\centering
	\subfloat[][]{\includegraphics[width=0.4\textwidth]{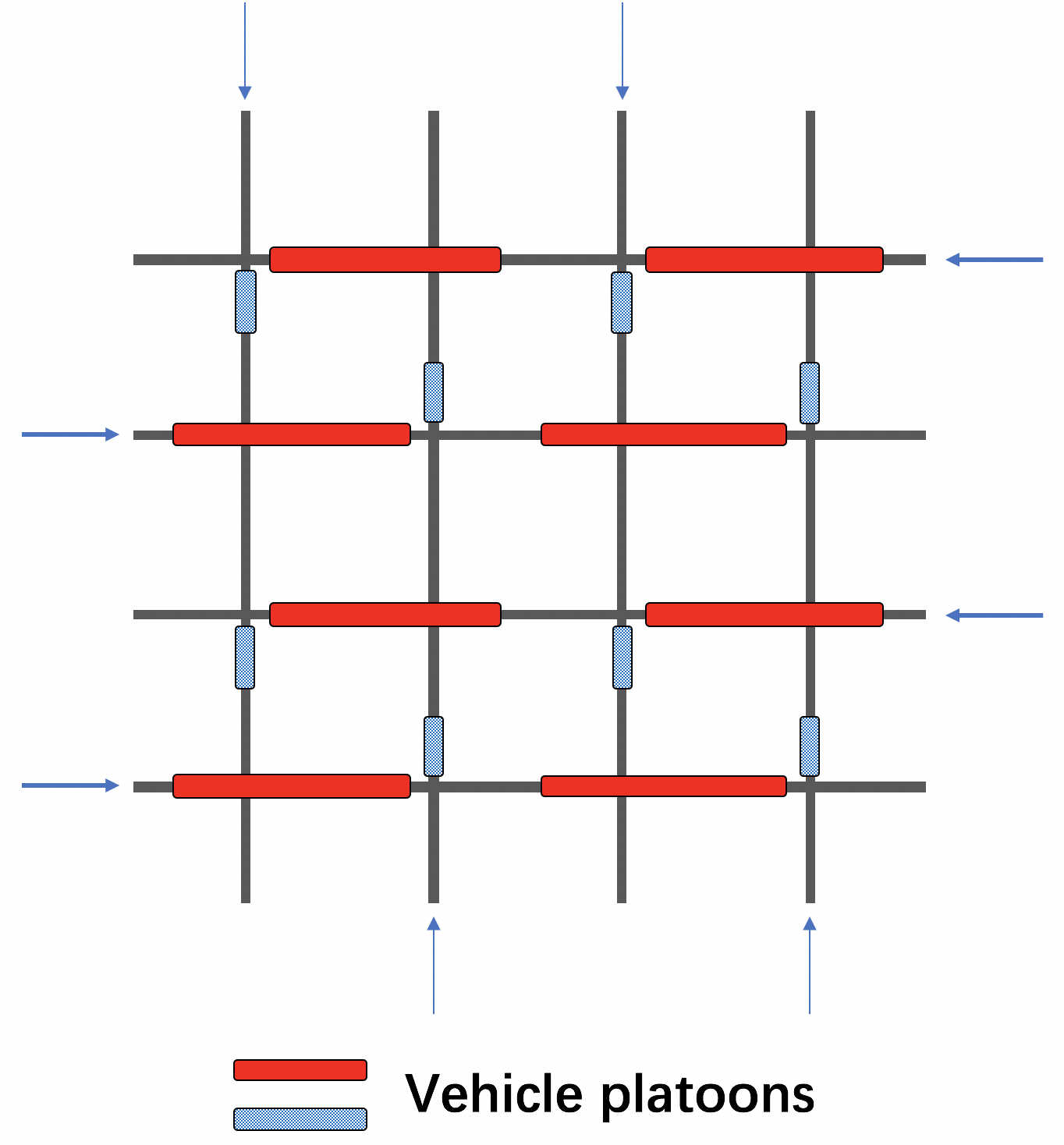}}\hspace{1em}
	\subfloat[][]{\includegraphics[width=0.4\textwidth]{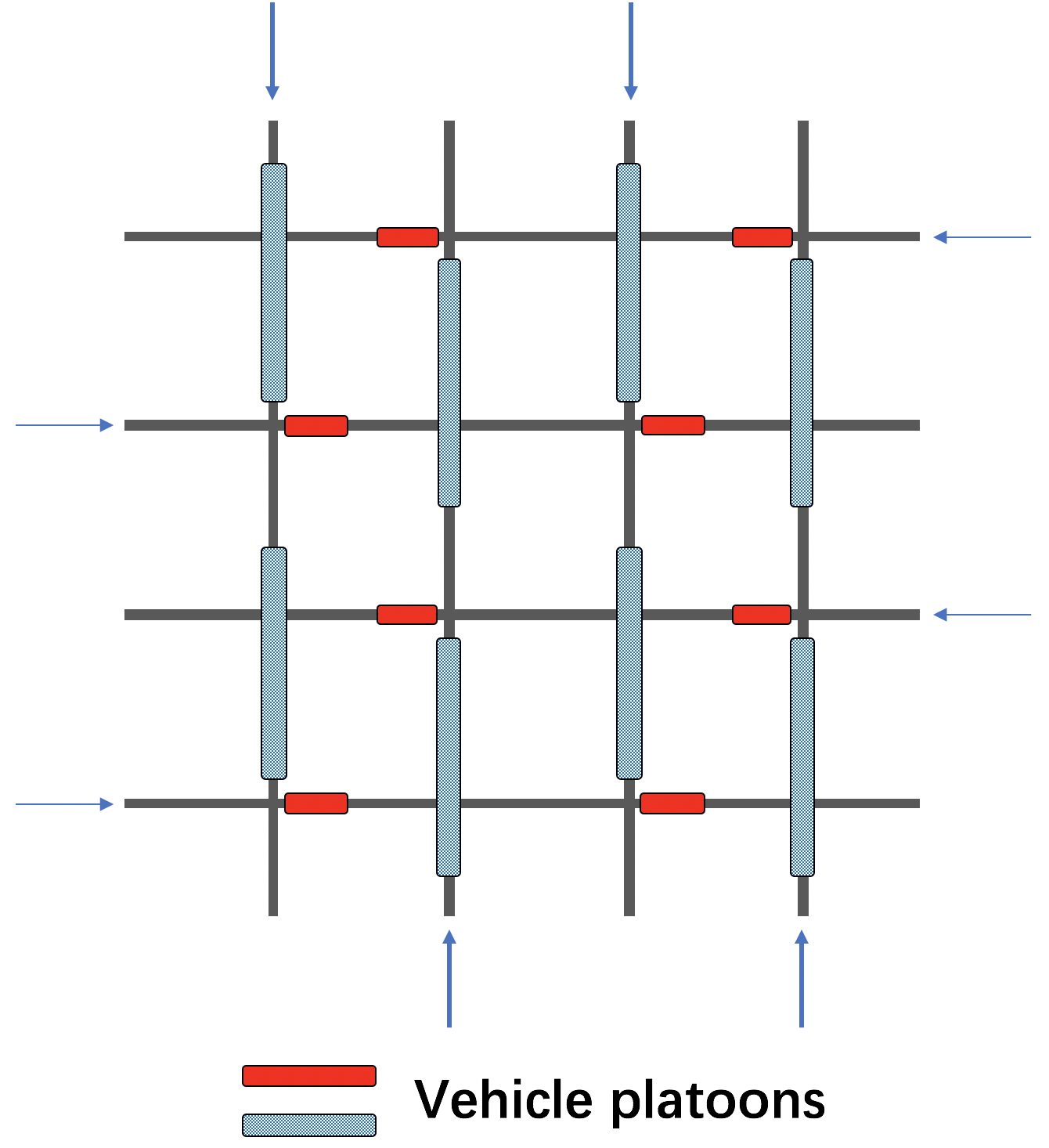}}\\
	\subfloat[][]{\includegraphics[width=0.4\textwidth]{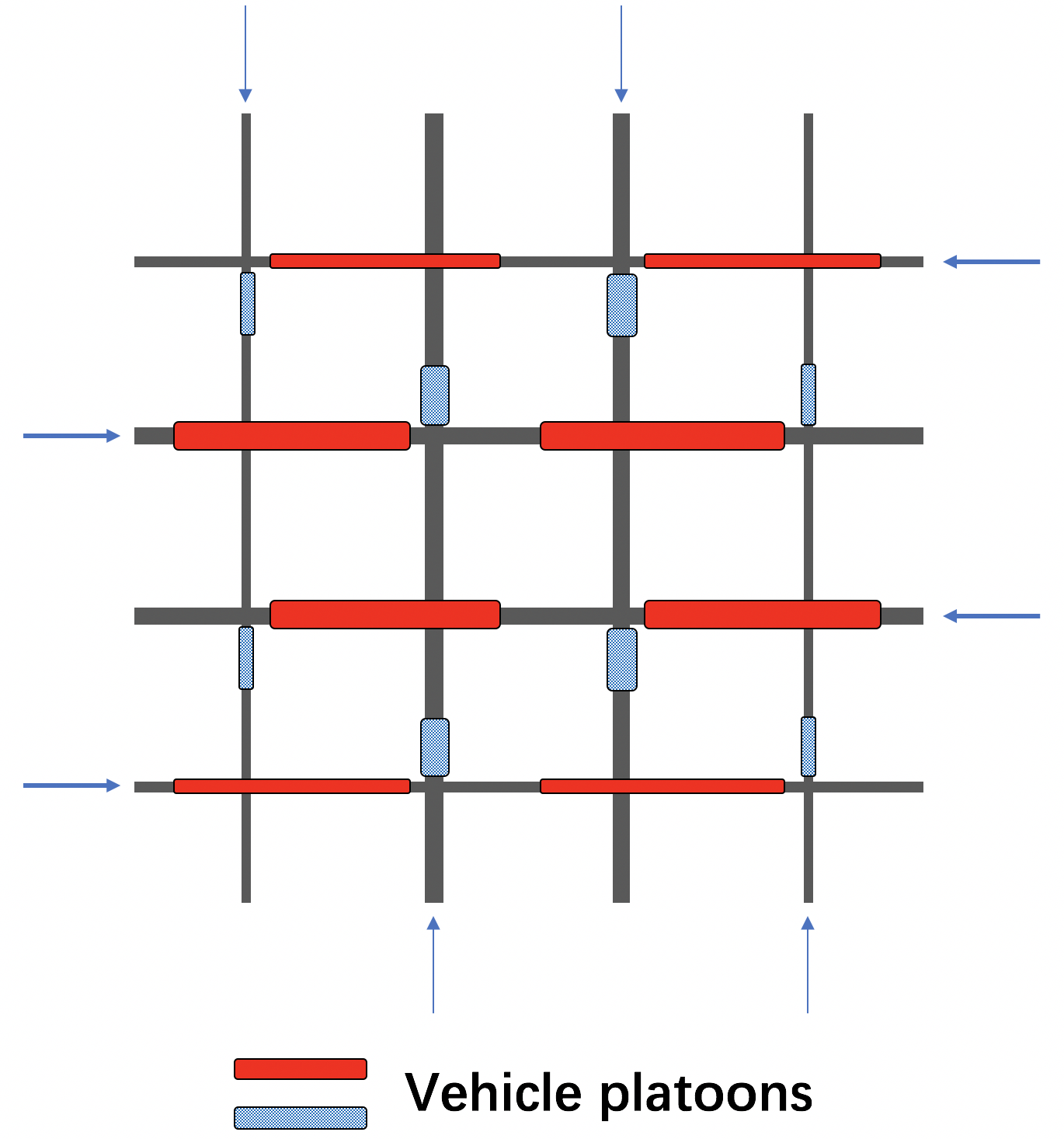}}
	\caption[]{Virtual platoons under imbalanced demands. (a) $\beta = 0.75$. (b) $\beta = 0.25$. (c) Heterogeneous street widths.}
	\label{ImbaD_fig}
\end{figure}
\par

Another way of managing imbalanced demands is to appropriately design the lane numbers on each street in the planning stages. For example, as shown in Figure \ref{ImbaD_fig}(c), different widths are assigned to different streets. Thus, even with the same length, platoons on streets with a larger width have a larger capacity. However, this way of handling imbalanced demands requires careful planning before building the network, and once determined, the infrastructure system cannot be altered in real time. Thus, this lane deployment method is not sufficiently flexible for handling time-variant demands.

\subsection{Temporary waiting within network} \label{TempWait_subsec}

In the RC framework, the vehicles entering a network will not incur any delay until they reach their destinations, and possible vehicle queuing can only occur at entrances or junctions (e.g., on perimeter gates or within a parking garage). Although such design avoids traffic gridlock on the network, undesirable congestion may occur outside. When encountering high traffic demands, it could be necessary to utilize some road space on the network for temporarily storing those vehicles. For this purpose, we provide several available options that can be used for temporary waiting lanes:

\begin{itemize}
    \item The curbside lanes used for junction entrances and exits, as shown in Figure 2(c). When some vehicles are ordered to wait temporarily on these curbside lanes, the system treats them as new trip demands generated from the corresponding junctions, and the subsequent space-time routes can be determined accordingly.
    
    \item The turning lanes from one arterial to another, as shown in Figure \ref{Components}(b). Similar to the previous case, vehicles waiting on these lanes are treated as newly generated demands, and subsequent routing orders will be given. Note that using the turning lanes as temporary waiting lanes causes temporary blockage of the traffic on this lane.
    
    \item Except for the two options above, on many road networks there exist curbside parking lanes. As the CCMS can collect the occupational information of these parking lanes, the unoccupied space can be used for temporary waiting.
\end{itemize}

When the temporary waiting within the network is considered, the online routing should be modified accordingly to accommodate such maneuvers. One simple way of modification is to split the whole path connecting an O-D pair into several parts by the temporary waiting zones it passes through. Each time a vehicle is about to approach a temporary waiting zone, the system decides whether it continues or enter the waiting zone based on the current platoon occupation information.

\subsection{Heterogeneous block sizes} \label{HetBlock_subsec}

In our current setting, all blocks are considered as identical, i.e., the intervals between two parallel neighboring streets are the same. Such a setting may be too idealized in reality, owing to the restraints in urban planning. If we design a one-way grid network that allows for heterogeneous block sizes (see Figure \ref{HetBlock_fig}), then it could be impossible to ensure the conflict-free property with a constant virtual platoon speed as proposed in Section \ref{rhythm_subsec}. This section therefore attempts to alleviate this problem, by investigating how to design variable speed curves for virtual platoons to achieve the conflict-free goal.

\begin{figure}[!ht]
	\centering
	\includegraphics[width=0.5\textwidth]{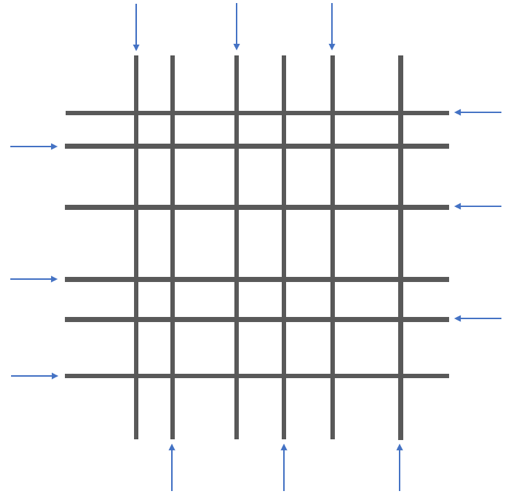}
	\caption[]{Heterogeneous block sizes}
	\label{HetBlock_fig}
\end{figure}

For convenience of discussion, we continue to assume that all platoons are able to pass through a fixed point in $\frac{1}{2}\hat{t}$. By the proof of Proposition \ref{prop2}, it is clear that the key to platoon coordination at crossroads is that the platoon travel time between two neighboring streets, i.e., $t_c$, satisfies the equation $t_c = a\hat{t}, a \in \mathbb{N}$. Inspired by such an observation, we can design the speed curve such that the travel times between two consecutive crossroads are always integers multiplied by $\hat{t}$; furthermore, to guarantee a platoon is able to pass through a fixed point within $\frac{1}{2}\hat{t}$,  we expect platoons to drive with considerable speeds at crossroads. We present the following speed curve generation (SCG) procedure to design the speed of platoons on a given street.\par

\noindent \underline{\textbf{SCG}} \\
\textbf{Step 1}: Acquire the inputs: maximum allowable speed $v_{max}$, maximum allowable acceleration $c_{max}$, maximum allowable deceleration $d_{max}$, and minimum allowable speed at crossroads $v_{min}^c$; the distances between two consecutive crossroads are $l_1,l_2,...,l_n$, ordered by sequence; the arrival time at the first crossroads is set to be $t_0 = 0$ with speed $v_0 \ge v_{min}^c$; initialize the iteration count as $i=1$. \\
\textbf{Step 2}: Compute $l_{max}^{i}(a) = \left\{ \begin{matrix} v_{i-1}a\hat{t} + \frac{1}{2}c_{max}(a\hat{t})^2, \ \textrm{if} \ a\hat{t} \le \frac{v_{max}-v_{i-1}}{c_{max}} \\ \frac{v_{max}^2-v_{i-1}^2}{2c_{max}} + v_{max}(a\hat{t} - \frac{v_{max}-v_{i-1}}{c_{max}}), \ \textrm{if} \  a\hat{t} > \frac{v_{max}-v_{i-1}}{c_{max}} \end{matrix} \right., a \in \mathbb{N}^+$, and define $a_i^* = \min_{a \in \mathbb{N}^+} \{ a: l_{max}^{i}(a) \ge l_i \}$. \\
\textbf{Step 3}: Design a speed curve for virtual platoons between crossroad $i-1$ and crossroad $i$ such that the speed at crossroad $i-1$ is $v_{i-1}$, the speed at crossroad $i$ is larger than or equal to $v_{min}^c$, the acceleration and deceleration do not exceed $c_{max}$ and $d_{max}$ respectively, and the time required to reach crossroad $i$ is exactly $a_i^*\hat{t}$. \\
\textbf{Step 4}: If $i=n$, terminate SCG and the whole speed curve is designed; otherwise, let $i \leftarrow i+1$ and return to Step 2. \par

In SCG, step 2 is aimed at choosing the shortest travel time between crossroads $i-1$ and $i$, such that the travel time is $\hat{t}$ multiples of an integer. Then, step 3 designs the speed curve based on such a travel time constraint. The following proposition provides the condition for the existence of a speed curve satisfying the requirements as stated in step 3.

\begin{prop}\label{prop4}
For a road segment between crossroads $i-1$ and $i$, there exists a speed curve satisfying the requirements in step 3 of the SCG if the platoon interval $\hat{t}$ and the length $l_i$ satisfy the following two conditions:

\begin{align*}
& \hat{t} \ge \frac{v_{max}}{d_{max}} + \frac{v_{min}^c}{c_{max}} \\
& l_i \ge \frac{v_{max}^2}{2d_{max}} + \frac{(v_{min}^c)^2}{2c_{max}}
\end{align*}

\end{prop}

\begin{proof}
See Appendix A.
\end{proof}
\par

Proposition \ref{prop4} provides a condition for the existence of an eligible speed curve, and this condition is not difficult to meet; for example, considering a case where $v_{max} = 15 \ m/s$, $v_{min}^c = 12 \ m/s$, $c_{max} = 2.5 m/s^2$, and $d_{max} = 3  \ m/s^2$, we find that $\hat{t} \ge \frac{v_{max}}{d_{max}} + \frac{v_{min}^c}{c_{max}} = 9.8 s$ and $l_i \ge \frac{v_{max}^2}{2d_{max}} + \frac{(v_{min}^c)^2}{2c_{max}} = 66.3 m$ can guarantee the existence of an eligible speed curve. Furthermore, the constructive proof of Proposition \ref{prop4} actually provides a feasible solution to step 3 in SCG. We note that such a construction may not be a unique solution to feasible SCG.

\section{Numerical tests} \label{Nume_sec}

In this section, we conduct comprehensive numerical studies to verify the effectiveness of our RC, in terms of network travel efficiency. All tests below are conducted on a $6 \times 6$ one-way grid network with homogeneous block sizes, where each street owns two lanes, and the distance between two adjacent crossroads is 150 m. The preset speed for all CAV platoons is $15 \ m/s$ ($54 \ km/h$), implying that the platoon travel time between two adjacent crossroads is $10 \ s$; by Proposition \ref{prop2}, to guarantee conflict-free environment, we set the network rhythm (as well as routing interval) to be $10 \ s$, i.e., $\hat{t} = 10 \ s$. In the virtual platoon, the minimum headway for two CAVs is set as 0.5 s. By Proposition \ref{prop2}, all virtual platoons must pass through a fixed point in $5 \ s$, so each platoon contains at most $20$ vehicles; to reserve some buffer space for robustness at crossroads, we set the capacity for platoons as 16 vehicles. Furthermore, we set the virtual platoon size on road segments between two crossroads a bit larger (i.e., 18 vehicles) because there is no risk for vertical collision on these segments, and vehicles can better utilize the additional capacity to enter or leave the network at junctions.\par

All tests and algorithms are coded with Matlab 2017b on a computer with an Intel Core i7-4790 CPU @ 3.60 GHz and a RAM of 16.0 GB. The AA-SPR and AA-MPR are solved by the off-the-shelf linear program solver in Matlab 2017b, where in AA-MPR the allowable detour time is set to be 40 seconds compared to the free-flow shortest-path travel time. The testing horizons for all scenarios are set to be 30 minutes.  

\subsection{Elementary tests} \label{EleTest_subsec}

In the elementary tests, we aim to present the performances of our proposed RC framework under different demand patterns. Specifically, the demand patterns vary in two dimensions: spatially, i.e., the demand distributions of different O-D pairs could be different; and temporally, i.e., the arrivals of CAVs could be stable or fluctuate with time $\footnote{The stable arrivals in scenarios 1-3 follow Poisson’s processes at each entrance and junction, and the demands fluctuate with multipliers ranging from 0.5 to 1.5 in scenarios 4-6.}$ To express such varieties, we design six scenarios for our tests, as shown in Figure \ref{Scenario_fig} below.

\begin{figure}[!ht]
	\centering
	\includegraphics[width=0.8\textwidth]{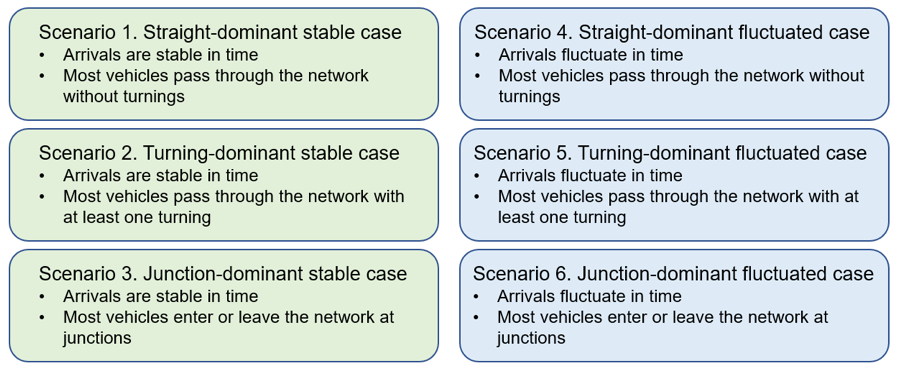}
	\caption[]{Six test scenarios}
	\label{Scenario_fig}
\end{figure}
\par

Figure \ref{EleTestResult_1} provides the results for the average trip delay with total demand rate in all six scenarios. Here, we note that the delay is composed of two parts: 1) the delay at entrances or junctions when entering the network, and 2) the additional travel time on the non-shortest paths $\footnote{We do not consider the deceleration/acceleration of crossroads turning for the delay, as turning movements require deceleration after all.}$. Moreover, the delays at entrances or junctions also consist of two parts: 1) the elementary part imposed by the incoming vehicles waiting for a virtual platoon to put themselves in, and this delay can be seen as roughly half of the network rhythm, i.e., $\frac{1}{2}\hat{t} = 5 \ s$; and ii) the additional delay caused by excessive demands. That is, when the demand is high, the incoming vehicles are very likely to wait for several rounds to enter a virtual platoon. All vehicles exhibit a basic average delay of $\hat{t} = 5 \ s$ at entrances, because routing only happens every 10 s. Figure \ref{EleTestResult_2} shows the average speed in all given scenarios, defined by the total travel distance divided by total travel time (including waiting at entrances/junctions) on the network. \par

From Figures \ref{EleTestResult_1}-\ref{EleTestResult_2}, we can draw the following observations. 

\begin{figure}[!ht]
	\centering
	\subfloat[][Scenario 1]{\includegraphics[width=0.45\textwidth]{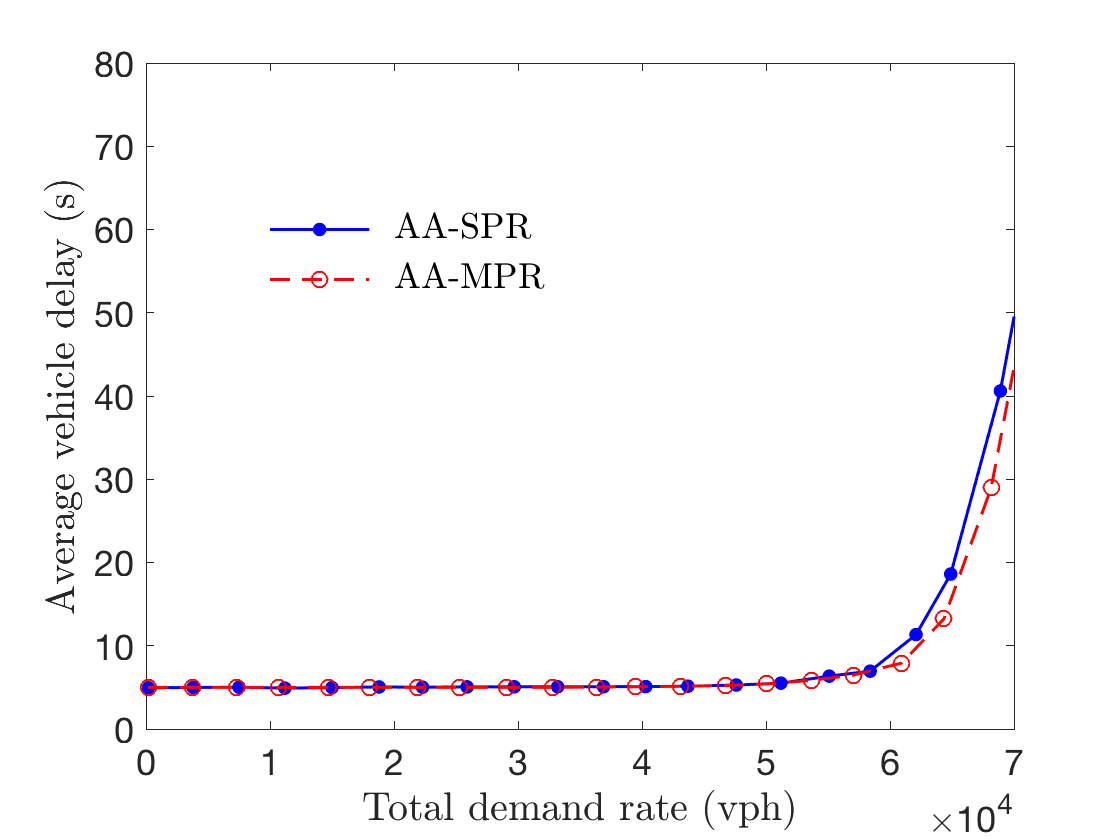}} \hspace{1em}
	\subfloat[][Scenario 4]{\includegraphics[width=0.45\textwidth]{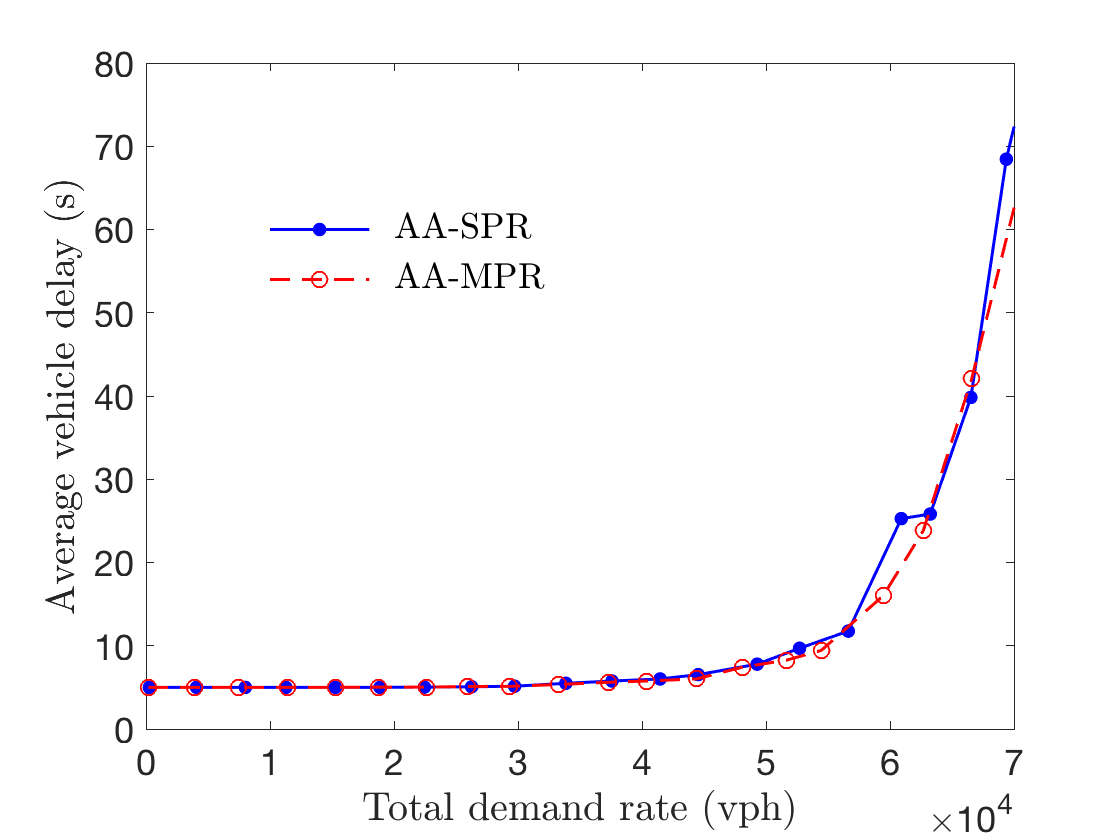}} \\
	\subfloat[][Scenario 2]{\includegraphics[width=0.45\textwidth]{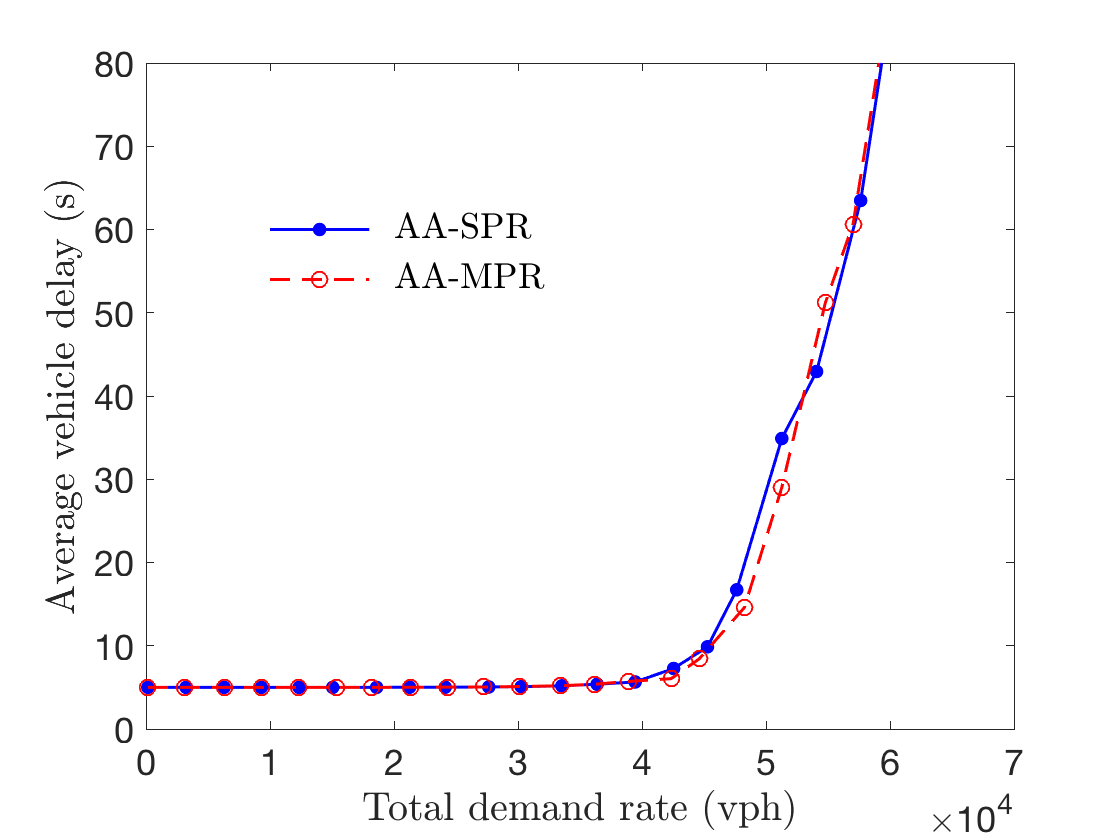}} \hspace{1em}
	\subfloat[][Scenario 5]{\includegraphics[width=0.45\textwidth]{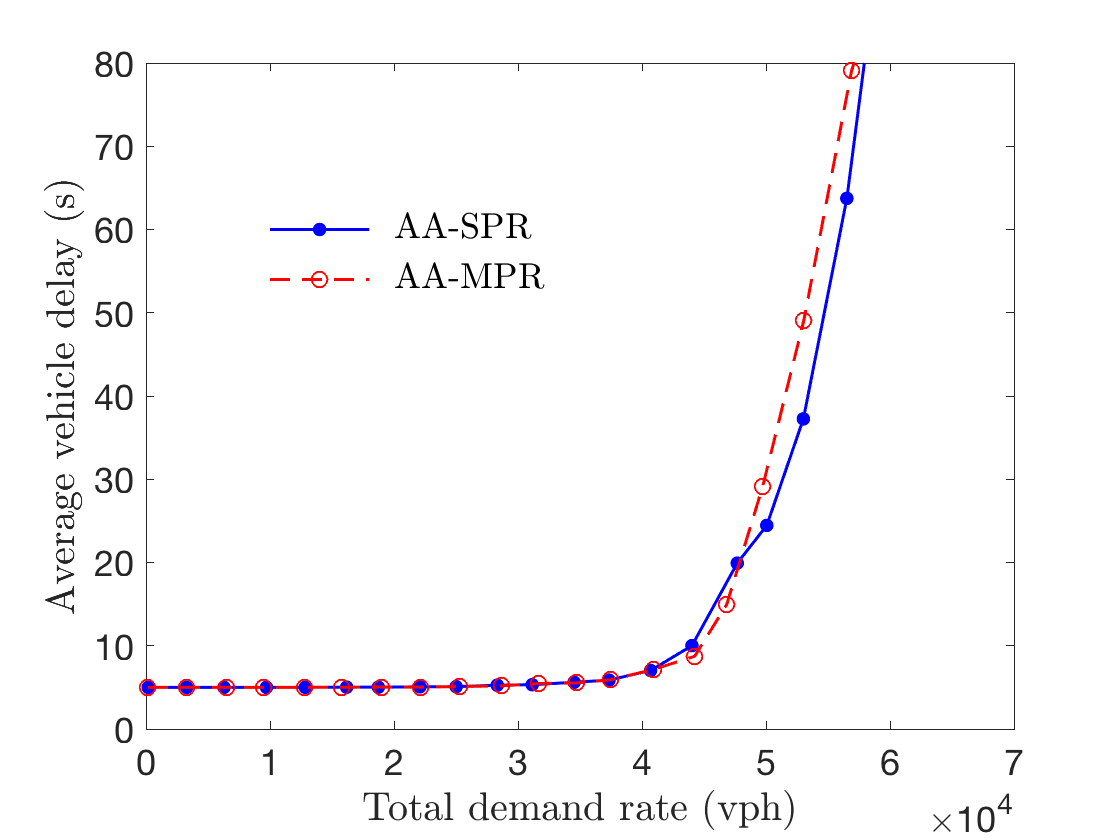}} \\
	\subfloat[][Scenario 3]{\includegraphics[width=0.45\textwidth]{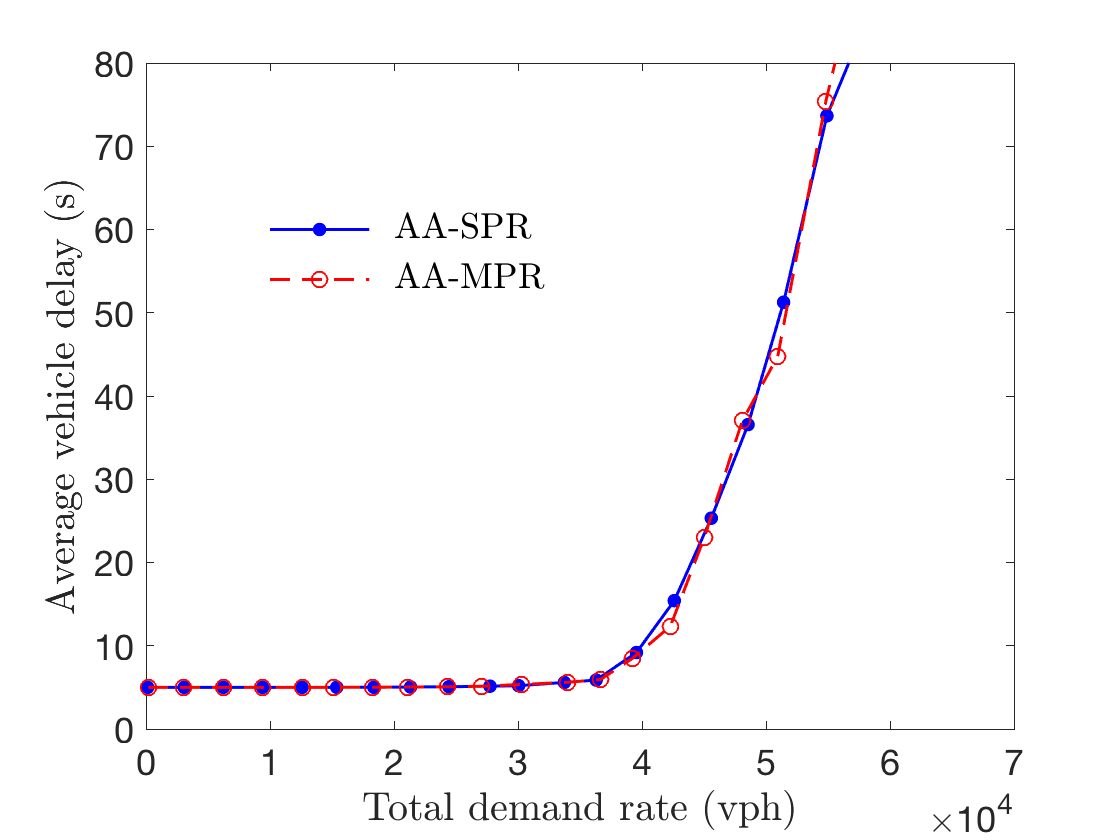}} \hspace{1em}
	\subfloat[][Scenario 6]{\includegraphics[width=0.45\textwidth]{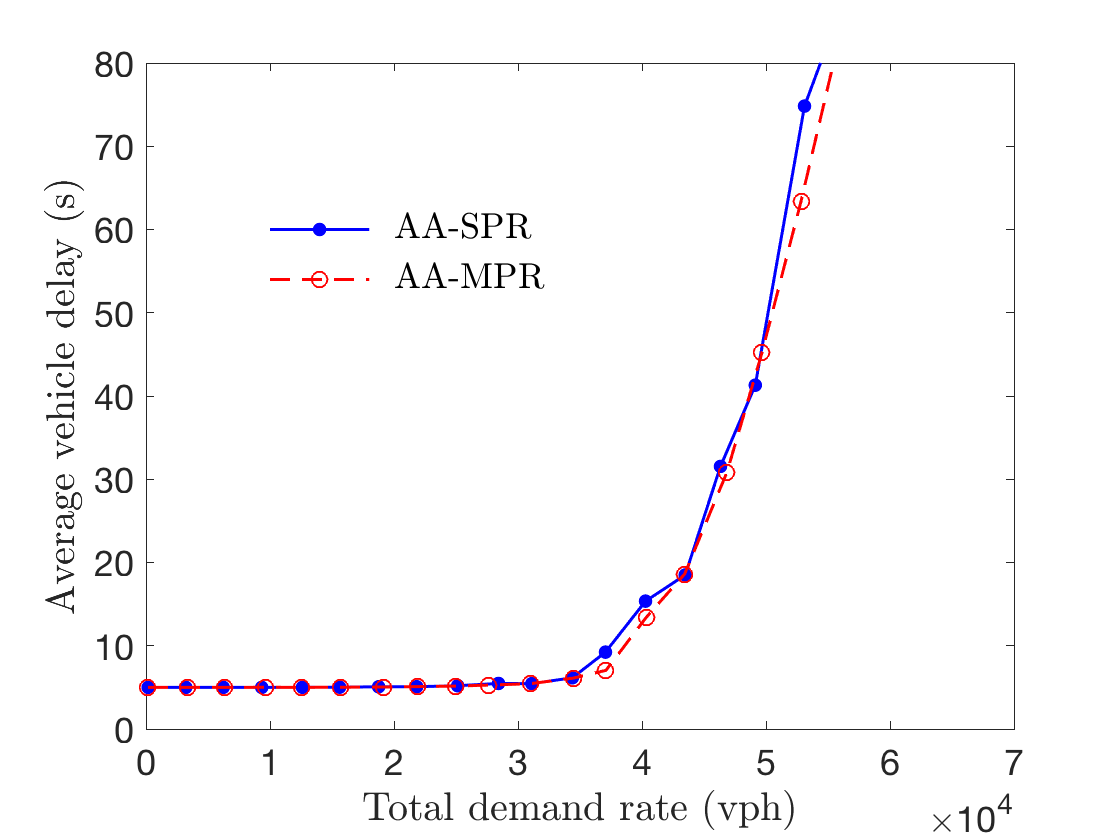}}
	\caption[]{Average vehicle delays in six scenarios}
	\label{EleTestResult_1}
\end{figure}

\begin{figure}[!ht]
	\centering
	\subfloat[][Scenario 1]{\includegraphics[width=0.45\textwidth]{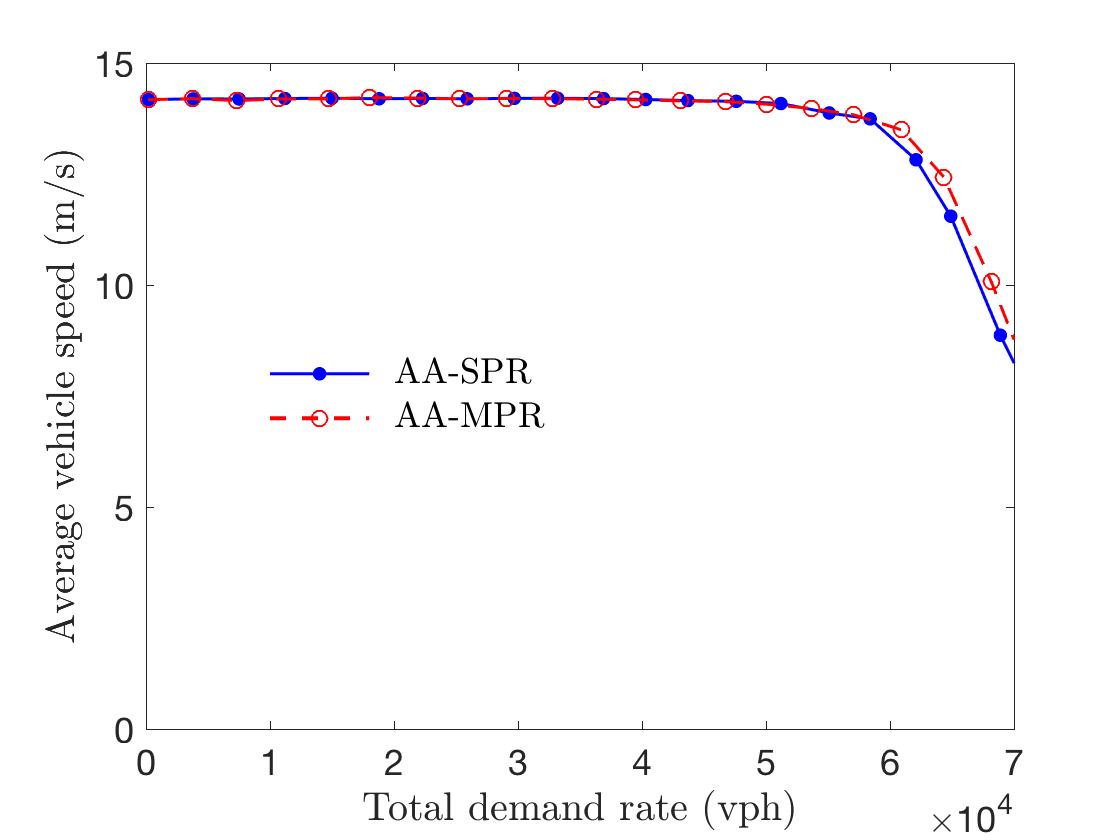}} \hspace{1em}
	\subfloat[][Scenario 4]{\includegraphics[width=0.45\textwidth]{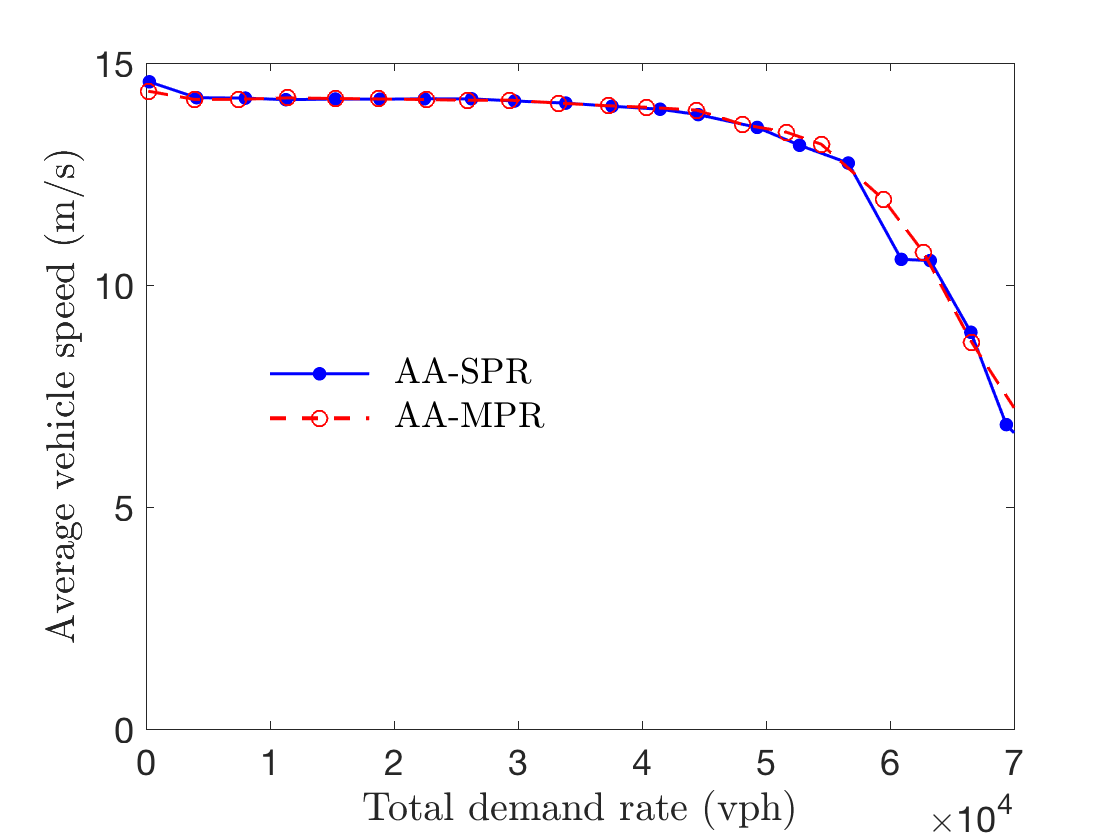}} \\
	\subfloat[][Scenario 2]{\includegraphics[width=0.45\textwidth]{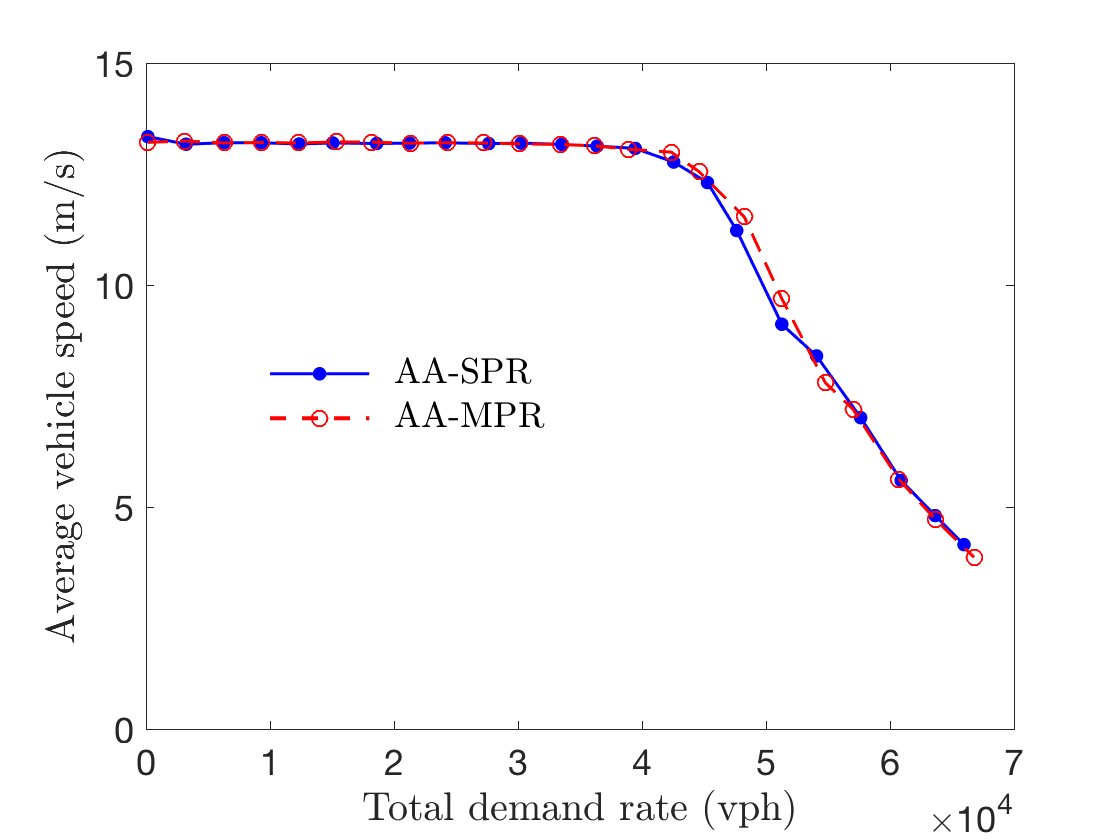}} \hspace{1em}
	\subfloat[][Scenario 5]{\includegraphics[width=0.45\textwidth]{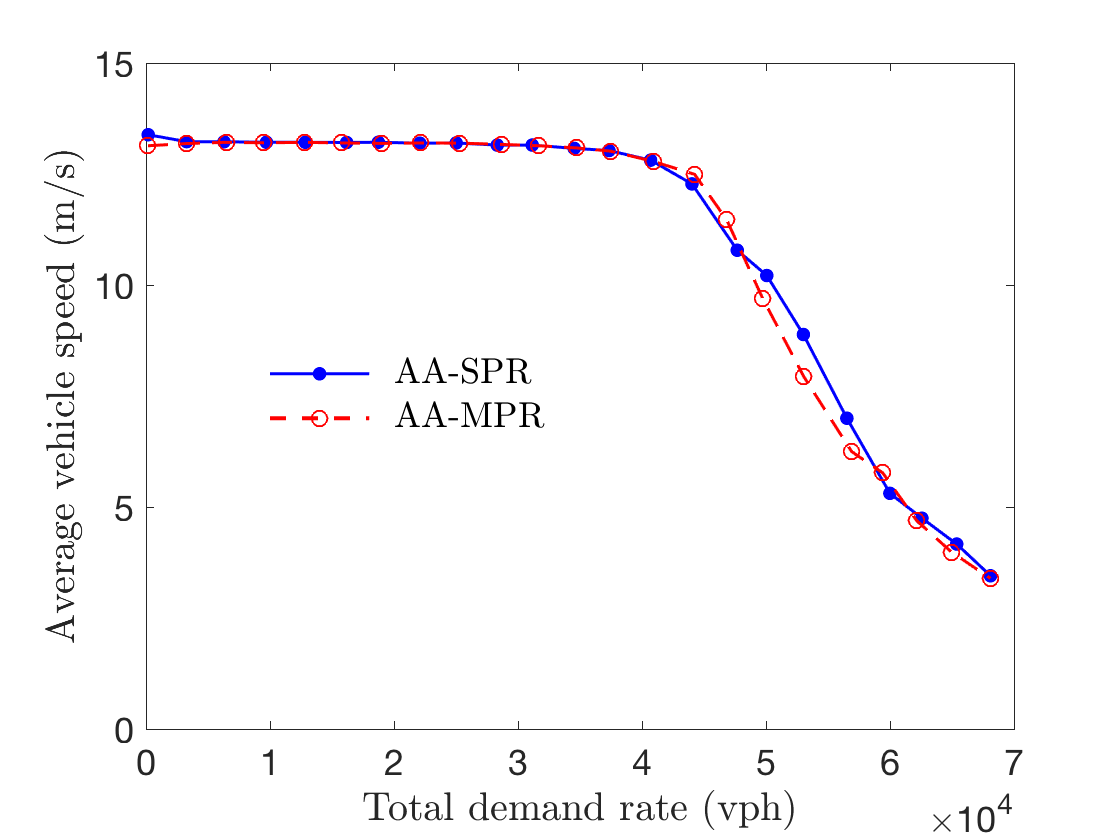}} \\
	\subfloat[][Scenario 3]{\includegraphics[width=0.45\textwidth]{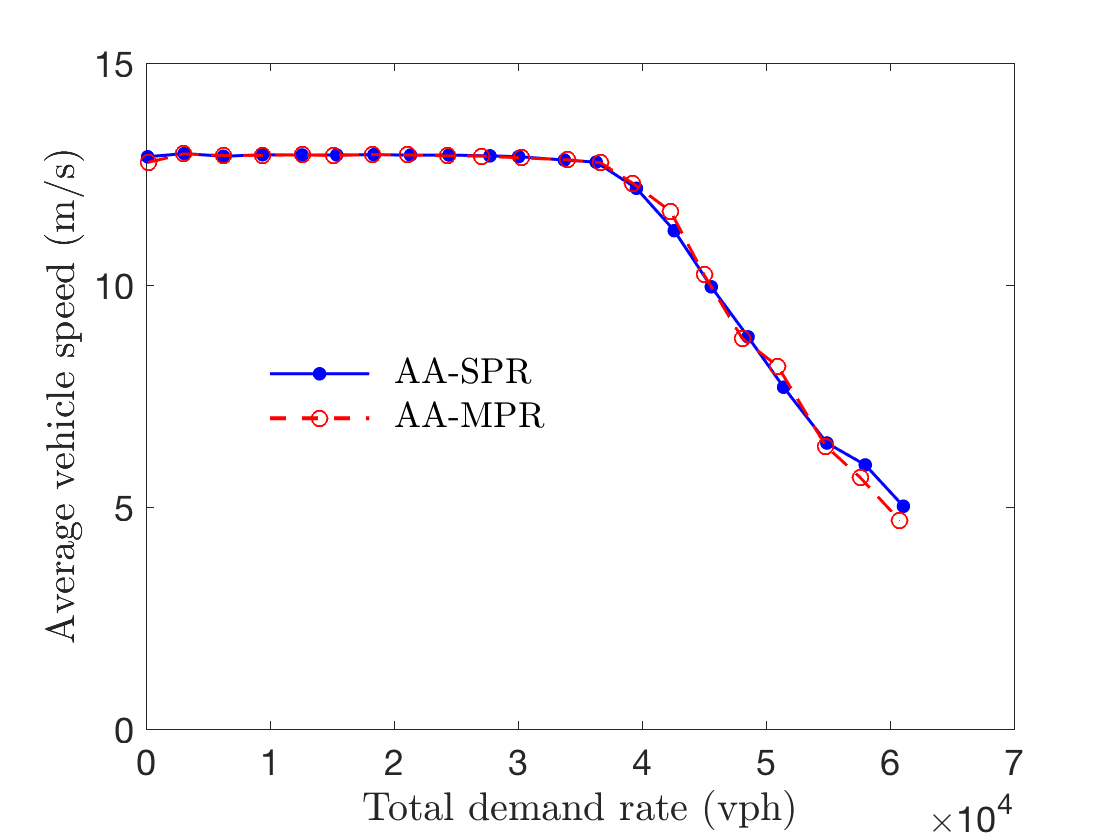}} \hspace{1em}
	\subfloat[][Scenario 6]{\includegraphics[width=0.45\textwidth]{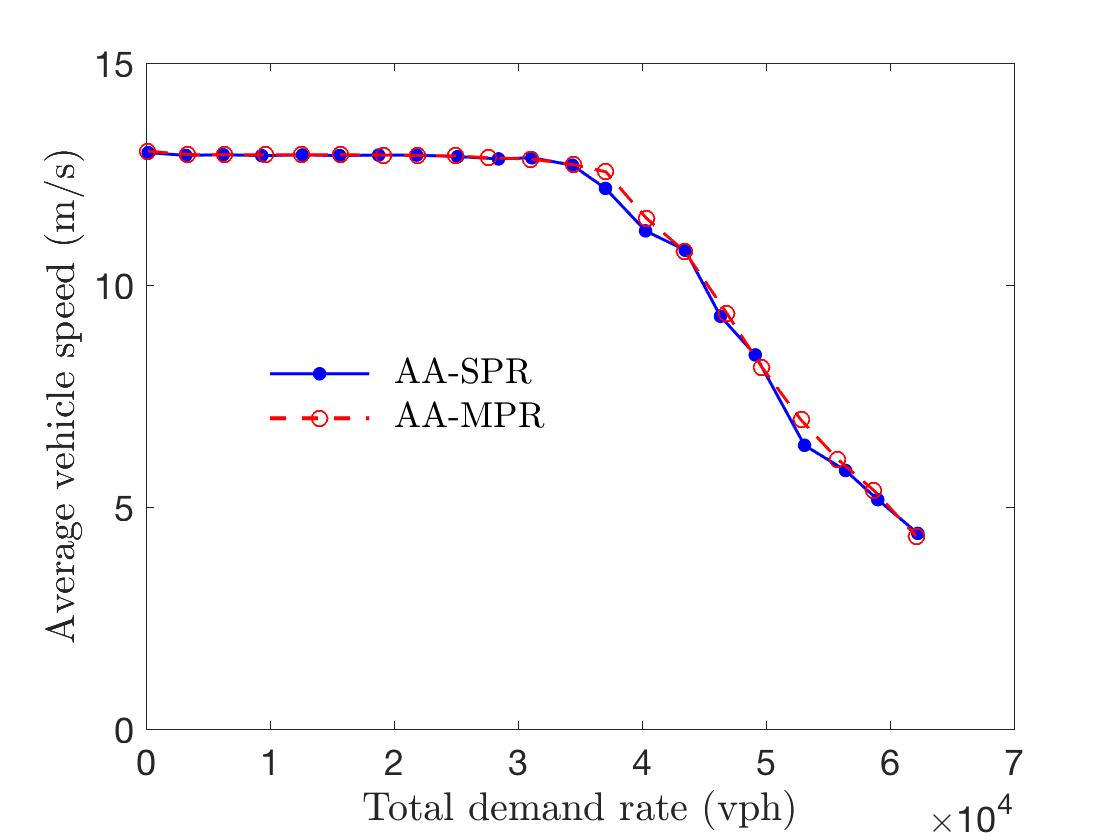}}
	\caption[]{Average vehicle speeds in six scenarios}
	\label{EleTestResult_2}
\end{figure}

\begin{itemize}
    \item The average vehicle delay through the entire network is almost negligible when the travel demand is not too high (approximately $5 s$), and the average vehicle speed approaches the maximum allowable speed (approximately $13 \ m/s$ or $47-50 \ km/h$). This result is within expectations, as Proposition \ref{prop2} indicates that under the proposed RC, vehicles only decelerate when turning at crossroads.

    \item The proposed RC framework can handle massive vehicle demands. Taking scenario 1 as an example, we see that when the total vehicle demand approaches 60,000 vehicles per hour, the average vehicle delay is only approximately $20 \ s$. The large service capacity can be partially explained by the fact that our proposed framework exhibits no stops or queuing once vehicles enter the network. However, in scenarios 2, 3, 5, and 6, we see that the service capacity of the network is undermined as compared to scenarios 1 and 4, because in these scenarios, the proportion of turning movements is increased, thereby reducing the network-wise throughput.
    
     \item We see that the RC framework is robust to fluctuations of vehicle arrivals, i.e., the fluctuations of demands do not distinctively deteriorate the average delay or average speed (see the comparisons between scenarios 1-3 and scenarios 4-6). 
     
     \item As observed in Table \ref{Compt}, both AA-SPR and AA-MPR are efficient in solving the routing models even when more than 1,000 vehicles are ready for entry at entrances and junctions. This confirms the applicability of our proposed RC in the real time implementation. In particular, AA-SPR requires much fewer computational efforts than AA-MPR when the problem size is sufficiently large.
     
     \item Lastly, regarding the performances of AA-SPR and AA-MPR, we observe that AA-MPR has some advantages over AA-SPR in general, but such advantages are quite limited; AA-MPR could sometimes perform even worse than AA-SPR (Scenario 5). One explanation for this phenomenon is that in our designed scenarios, traffic demands are distributed quite evenly in spatial dimension, and that compromises the potential of AA-MPR for further reducing vehicle delay under heavy traffic. To observe the potential of AA-MPR, we set up an additional scenario that is modified from Scenario 1 by adjusting the relative demand magnitudes on different arterials to make room for detours; meanwhile, in this test we set the allowable detour time as 60 seconds. Table \ref{Compt_2} presents the results under several demand levels; it suggests that AA-MPR indeed show distinctive advantages over AA-SPR in heavy traffic cases.

\end{itemize}

\begin{table}[!ht]
  \centering
  \caption{Average computational times of AA-SPR and AA-MPR}
  \label{Compt}
  \small
  \begin{tabular}{ c  c  c }
   \hline
   Number of vehicles & AA-SPR (sec/veh) & AA-MPR (sec/veh) \\
		\hline
		$<$200	& $<$0.1 & $<$0.1 \\
		200$\sim$500   & $<$0.1 & 0.2  \\
		500$\sim$1,000   & 0.2 & 0.6  \\
		1,000$\sim$2,000   & 0.4 & 3.0  \\
   \hline
  \end{tabular}
\end{table}

\begin{table}[!ht]
  \centering
  \caption{Average vehicle delay comparisons of AA-SPR and AA-MPR in the additional scenario}
  \label{Compt_2}
  \small
  \begin{tabular}{ c  c  c }
   \hline
  Total demand rate (vph) & AA-SPR (sec/veh) & AA-MPR (sec/veh) \\
		\hline
		40,000	& 5.27 & 5.32 \\
		45,000   & 5.63 & 6.03  \\
		50,000   & 11.96 & 7.77  \\
		55,000   & 37.27 & 16.56  \\
		60,000   & 65.26 & 49.67 \\
   \hline
  \end{tabular}
\end{table}

\subsection{Tests on rhythm length choice} \label{Sensitivity_subsec}

This subsection provides some results on choosing network rhythm length (i.e., the platoon size) with the trade-off between entrance delay and traffic capacity, discussed in Section \ref{Buffer_subsec}. In the tests, we set three available network rhythm lengths with detailed information shown in Table \ref{Net_rhythm_table}. The buffer reserved in each platoon occupies the space of 2 vehicles in both the head and the tail. Therefore, for a virtual platoon with the size of $N$ vehicles, only $N-4$ vehicles can be effectively contained. 

\begin{table}[!ht]
  \centering
  \caption{Information on available rhythm lengths}
  \label{Net_rhythm_table}
  \small
  \begin{tabular}{ c  c  c }
   \hline
   Rhythm length (sec) & Platoon size (No. Vehicles) & Valid platoon size (No. Vehicles) \\
		\hline
		10	& 20 & 16 \\
		5   & 10 & 6  \\
		$10/3$   & 6 & 2  \\
   \hline
  \end{tabular}
\end{table}
\par

Figure \ref{Sensitivity_fig} compares average vehicle delays under three different network rhythm lengths; the test scenario is chosen as Scenario 1 (i.e., the straight-dominant stable cases). As illustrated, in light traffic cases, setting the rhythm length as small as possible is beneficial for reducing delays at entrances and junctions; but in heavy traffic cases, only a relatively large rhythm length can maintain an acceptable traffic efficiency. Specifically, when the demand rate is no more than 10,000 vph, setting the rhythm length as $10/3$ seconds is optimal; when the demand rate is between 10,000 and 35,000 vph, a rhythm length of 5 seconds is optimal; and when the demand rate is above 35,000 vph, we should choose the rhythm length as 10 seconds for minimizing delay. The results are in accordance with the discussions and insights proposed in Section \ref{Buffer_subsec}.

\begin{figure}[!ht]
	\centering
	\includegraphics[width=0.6\textwidth]{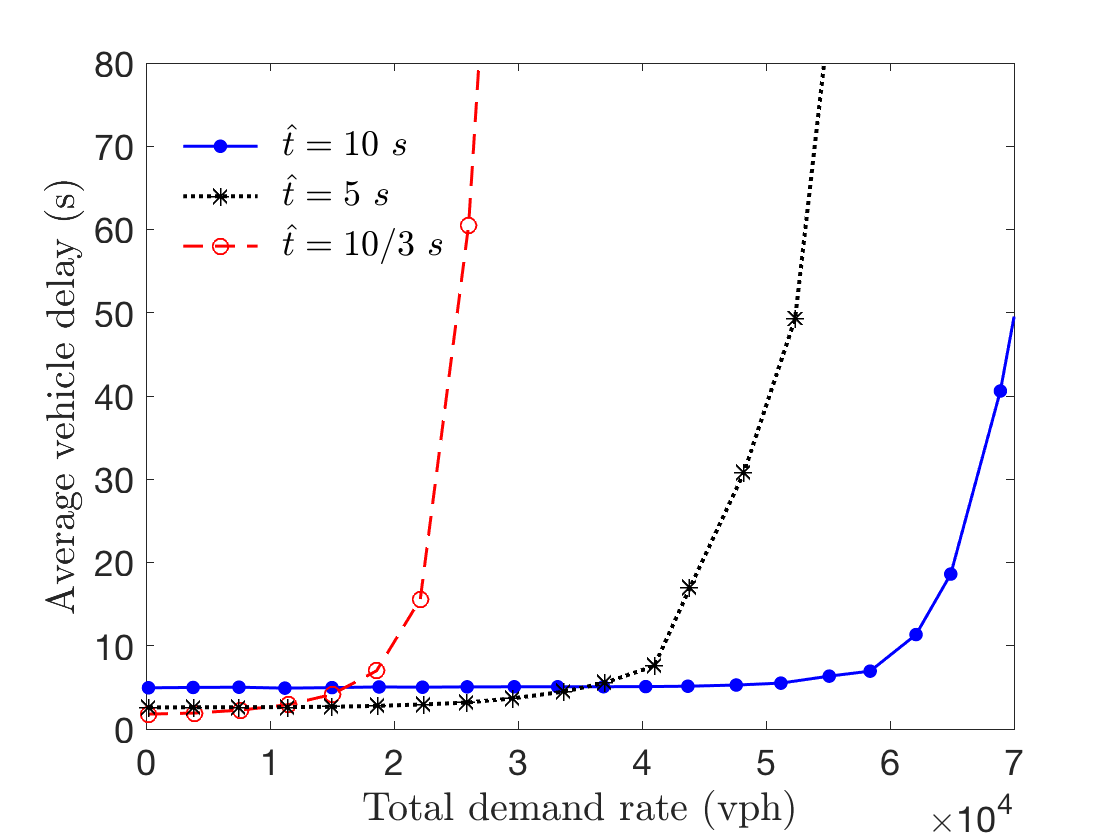}
	\caption[]{Average vehicle delays under different network rhythms}
	\label{Sensitivity_fig}
\end{figure}

\subsection{Comparative tests} \label{ComTest_subsec}

In this section we compare our RC with some other control methods in literature. The network rhythm length is chosen as a fixed 10 seconds (so the average vehicle delay in light traffic cases will be overestimated to some extent). The routing method is chosen as AA-SPR, so in the following we abbreviate our method as RC-SPR. Two well-known benchmark methods are implemented for comparison purposes, introduced as follows: \par

\begin{itemize}
    \item Max-pressure control (MP). This is a signalized control scheme proposed by Varaiya (2013). MP controls a network of signalized intersections, and is proven to maximize the network throughput under a point queue assumption without limit on storage capacity. Therefore, we choose it as a benchmark to evaluate the capacity of our proposed RC. In implementing MP, we set the length of one time slot in MP control as $5 \ s$; the impact of acceleration and deceleration is ignored. We realize two versions of MPs: one can only utilize the shortest paths, abbreviated as MP-1, and the other can dynamically adjust each vehicle's routing choice when approaching all diverging points based on the dynamically-updated travel time information (similar to the idea in Chai et al. (2017)), abbreviated as MP-R below.
    
    \item First-come-first-serve (FCFS) intersection control with real-time routing (Hausknecht et al.,2011). This is a non-signalized control scheme that guarantees conflict-free conditions at intersections by assigning non-conflicting space-time trajectories to approaching vehicles following FCFS rule, and the routing of a vehicle is dynamically updated according to the on-time path travel time information. In the following, this method is abbreviated as FCFS-R.
\end{itemize}

The simulations of benchmark methods are conducted under a cell-transmission-type environment. For fairness consideration, all related parameters (including maximum speed, saturation flow rate and safety buffer) are set as the same as in the RC. We choose scenarios 1-3 (the stable cases) for the comparative tests. The results are shown in Figure \ref{ComTestResult_1}, including the information of average vehicle delay and vehicle throughput (vehicle throughput in the tests is defined as the total number of vehicles that reach their destinations within the simulation horizon, i.e., 0-30 minutes). Similar to the above tests, the vehicle delays at entrances and junctions in RC are incorporated into the delay calculation. Table \ref{StdD} shows the standard deviation of vehicle delay. Observations are drawn as follows:

\begin{figure}[!ht]
	\centering
	\subfloat[][Scenario 1: vehicle delay]{\includegraphics[width=0.45\textwidth]{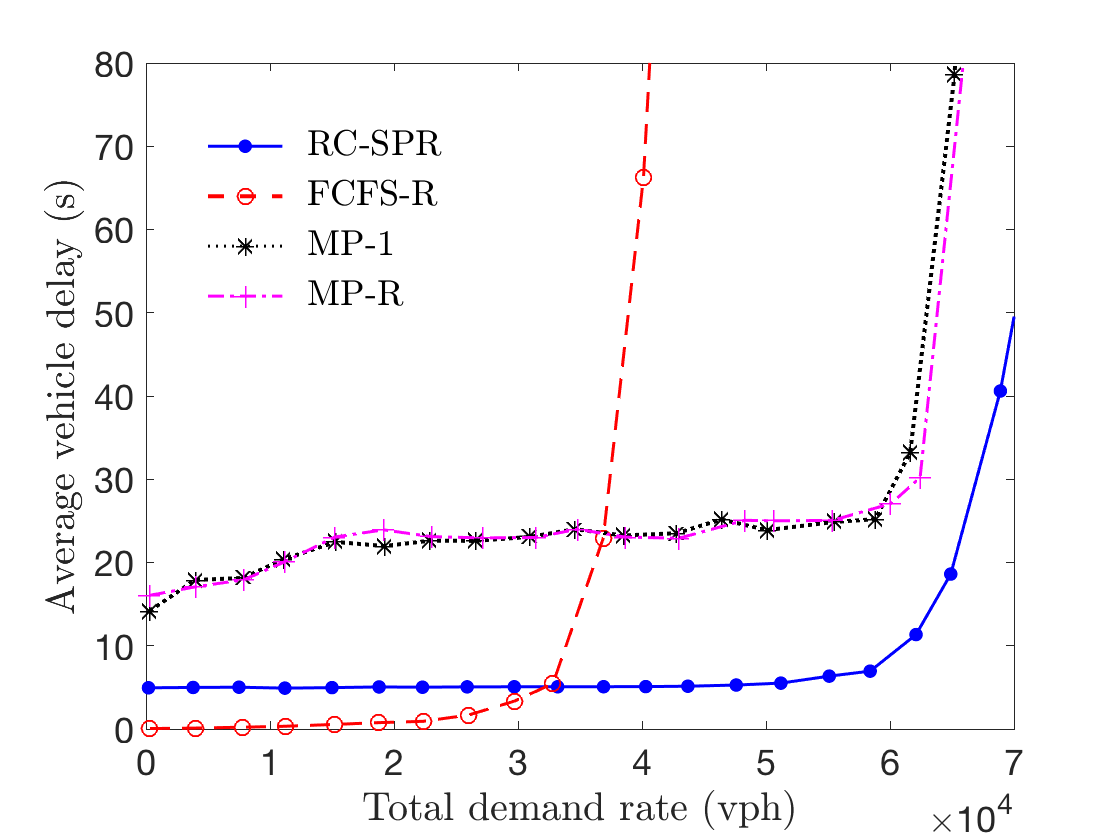}} \hspace{1em}
	\subfloat[][Scenario 1: vehicle throughput ]{\includegraphics[width=0.45\textwidth]{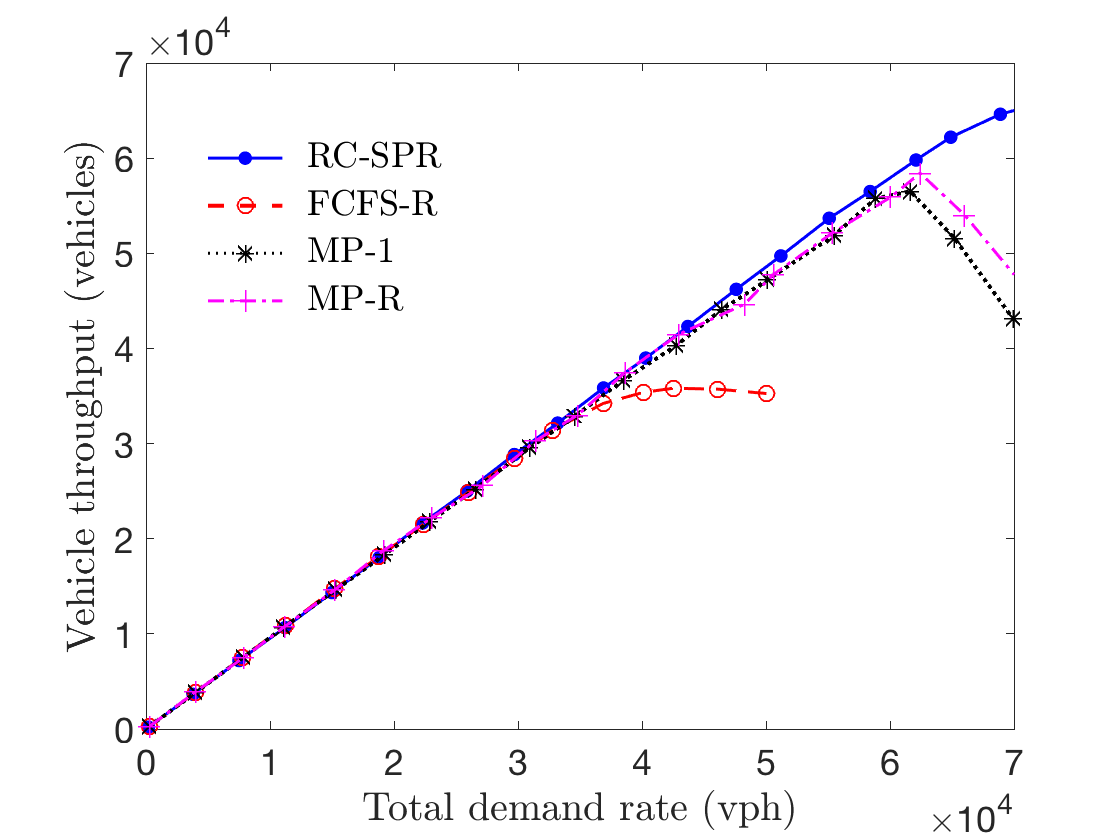}} \\
	\subfloat[][Scenario 2: vehicle delay]{\includegraphics[width=0.45\textwidth]{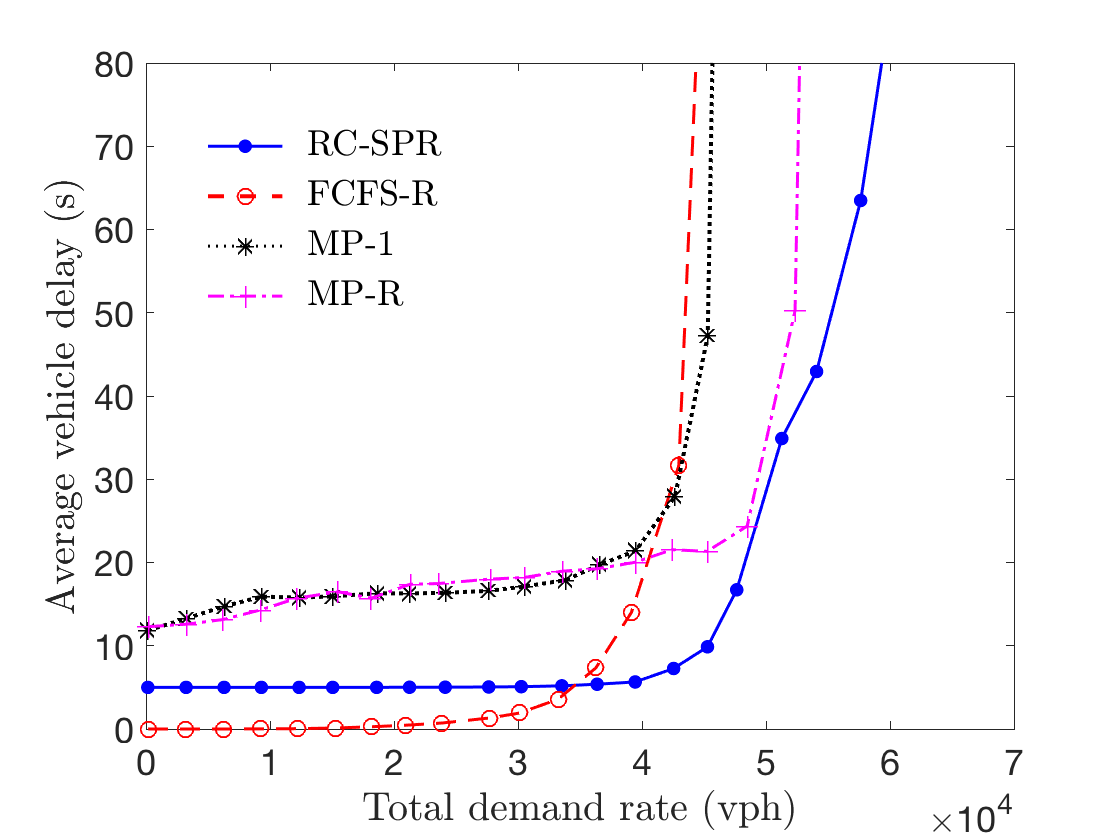}} \hspace{1em}
	\subfloat[][Scenario 2: vehicle throughput]{\includegraphics[width=0.45\textwidth]{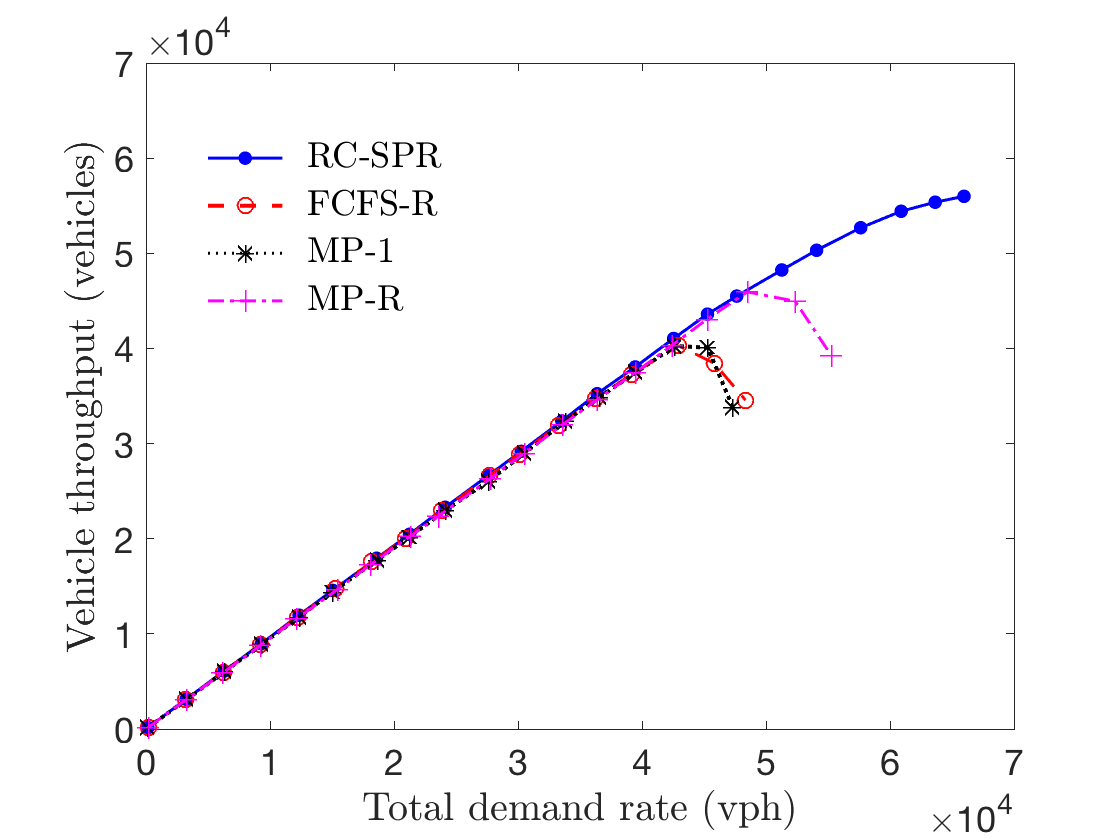}} \\
	\subfloat[][Scenario 3: vehicle delay]{\includegraphics[width=0.45\textwidth]{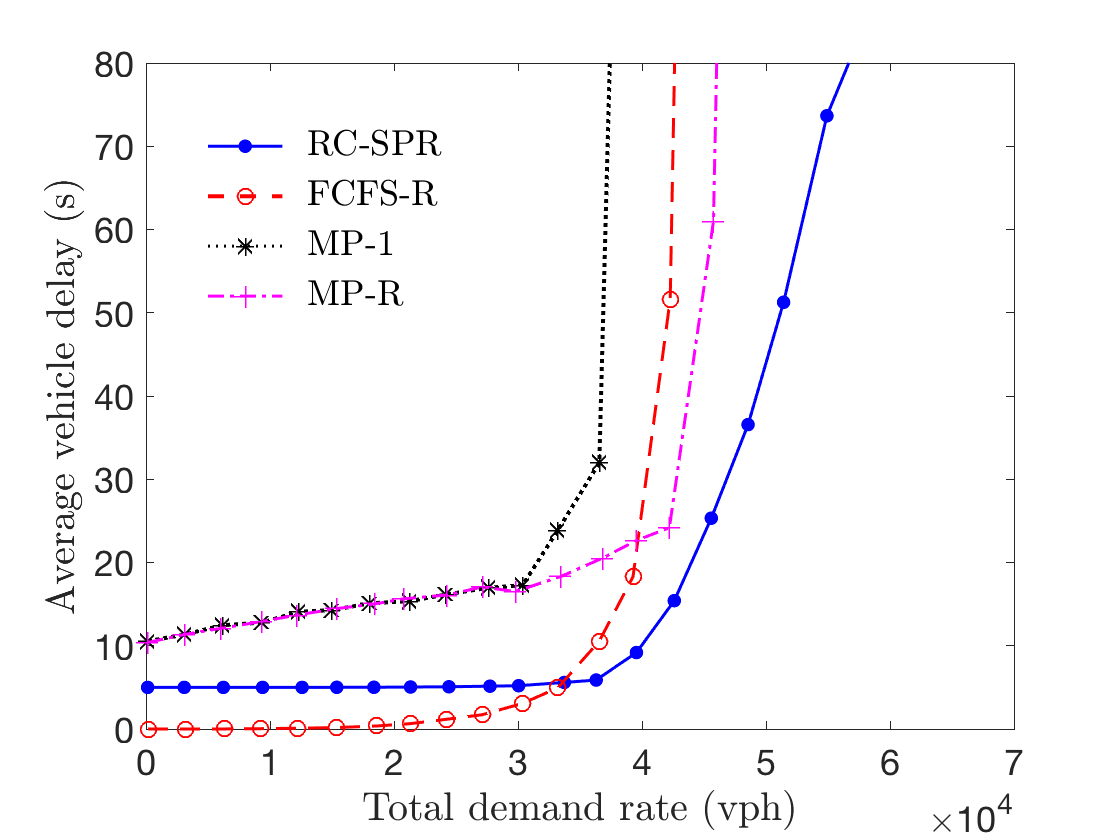}} \hspace{1em}
	\subfloat[][Scenario 3: vehicle throughput]{\includegraphics[width=0.45\textwidth]{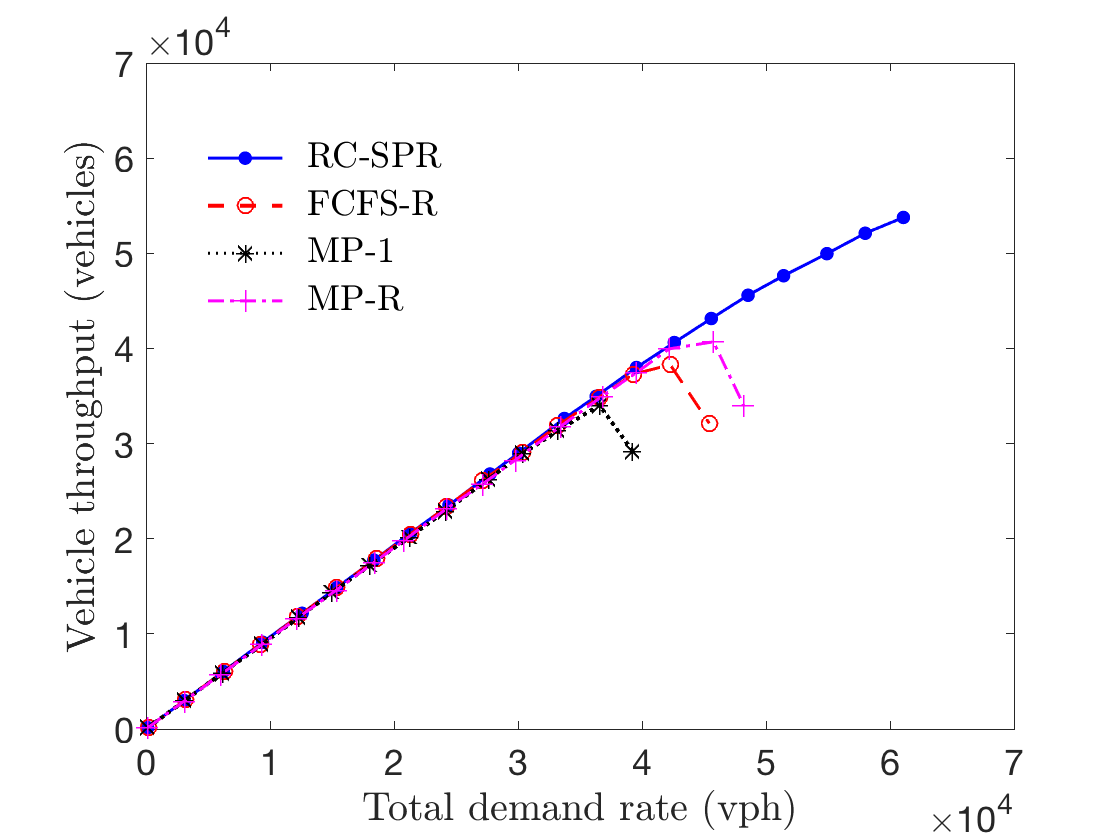}}
	\caption[]{Vehicle delay and throughput in comparative tests}
	\label{ComTestResult_1}
\end{figure}

\begin{itemize}
    \item In general, RC with AA-SPR exhibits a low average vehicle delay. In low-demand cases, the RC shows a delay of approximately $5 \ s$, only slightly higher than FCFS control. In high-demand cases, the vehicle delay of both MP controls and FCFS control sharply rise to an unacceptable level, whereas in RC, the rise is relatively mild. This is because in MP and FCFS, the network traffic under high demands starts to form gridlocks (see Long et al.(2011) for similar observations), whereas in RC, the network traffic is always efficiently organized. Moreover, from Table \ref{StdD} we can also observe that in most cases the RC with AA-SPR achieves a low standard deviation of vehicle delay, and the advantage is rather distinct under heavy traffic demands, as well.

    \item On the other hand, in view of vehicle throughput, we observe that RC holds some distinctive advantages over the benchmarks. Due to the possibility of forming gridlocks, benchmark controls show a "throughput drop" phenomenon when the demand is approaching a certain level, while in RC the throughput keeps increasing with the demand (with a gradually slowing-down trend in heavy traffic cases). It demonstrates that RC can well handle overloaded traffic. Nevertheless, such capability is acquired by "transferring" the vehicle queues to entrances and junctions, so implementing the control requires some supplementary mechanisms to manage the traffic at junctions or outside the network.

\end{itemize}

\begin{table}[!ht]
  \centering
  \caption{Standard deviation of vehicle delay under different controls}
  \label{StdD}
  \footnotesize
  \begin{tabular}{ c | c  c  c | c  c  c | c  c  c }
   \hline
   Demand rate & & Scenario 1 &  &  & Scenario 2 &  &  & Scenario 3 & \\
    (vph) & RC-SPR & MP-R & FCFS-R & RC-SPR & MP-R & FCFS-R & RC-SPR & MP-R & FCFS-R \\
     & (sec) & (sec) & (sec) & (sec) & (sec) & (sec) & (sec) & (sec) & (sec) \\
		\hline
		10,000	& 2.9 & 8.2 & 0.1 & 2.9 & 9.3 & 0.1 & 2.9 & 8.9 & 0.0 \\
		20,000   & 2.9 & 8.3 & 0.3 & 2.9 & 9.7 & 0.2 & 2.9 & 9.6 & 0.3 \\
		30,000   & 2.9 & 8.4 & 1.7 & 2.9 & 10.3 & 1.1 & 3.1 & 10.5 & 2.2 \\
		40,000   & 3.2 & 8.3 & $>$100 & 4.3 & 24.0 & $>$100 & 18.4 & 40.7 & 67.1 \\
		50,000  & 4.9 & 8.5 & $>$100 & 78.5 & $>$100 & $>$100 & $>$100 & $>$100 & $>$100 \\
   \hline
  \end{tabular}
\end{table}

Lastly, we show some results indicating the advantageous performance of one-way network in handling heavy traffic. We conduct a comparative study with MP control on a two-way grid network. The two-way grid network also has a size of $6 \times 6$, and all roads own one lane in each direction. The intersection on the two-way network contains through, right-turn, and left-turn movements on all legs, and the sum of saturation flow rates on all approaching lanes in an intersection equals that of the crossroads on the one-way grid network. The average travel time includes the average vehicle delay and the time for vehicles to traverse links on the network. Comparison results are shown in Figure \ref{Oneway_twoway_pic}, where the demand pattern is set similarly to Scenario 1. As illustrated, both the RC and the MP control on the one-way grid network exhibit a lower average delay than the MP control on the two-way grid network. This is attributed to the fact that the intersections of one-way streets own fewer conflict points and simpler streamlines than those of two-way streets. With regard to the average travel time, we observe that when the demand is low, the average travel time on the two-way grid network is smaller than that on the one-way grid network under MP control, and the situation reverses with an increase of demand magnitude. This phenomenon clearly illustrates the trade-off between the delay reduction and the additional detour brought by the one-way road networks, and it suggests that one-way roads can produce more desirable performance in areas with high traffic demands.

\begin{figure}[!ht]
	\centering
	\subfloat[][Vehicle delay]{\includegraphics[width=0.45\textwidth]{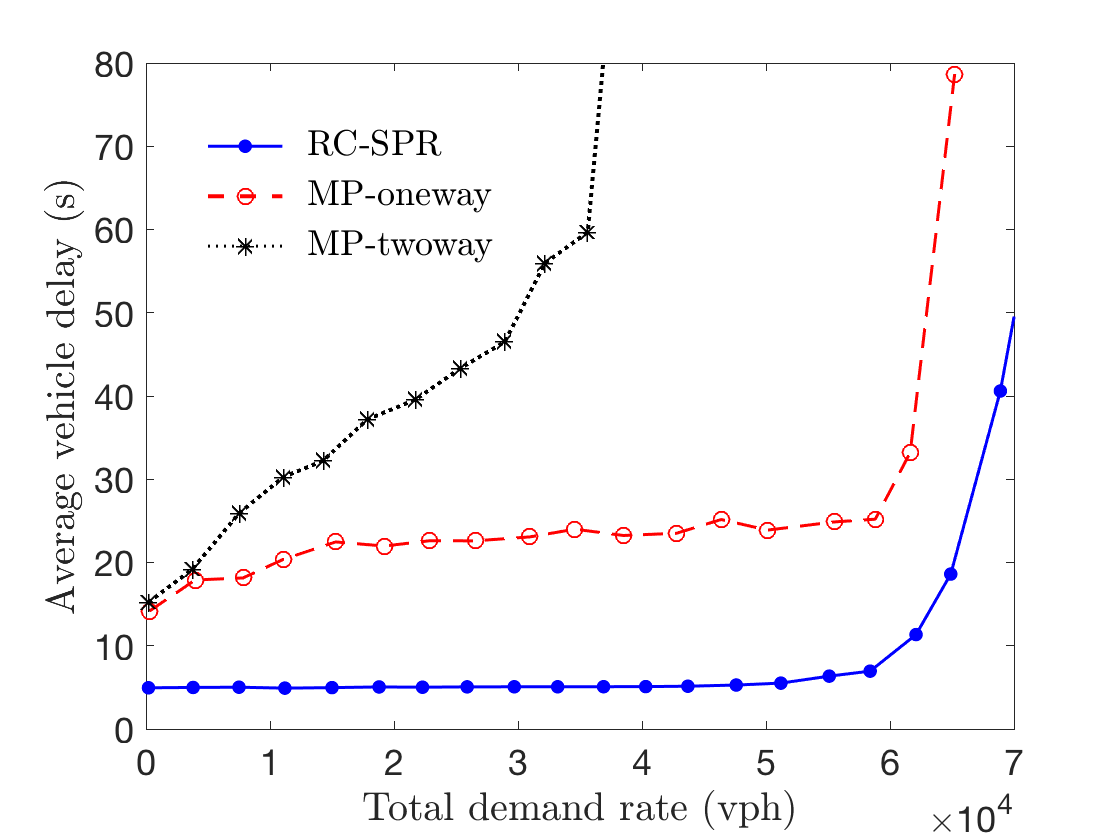}} \hspace{1em}
	\subfloat[][Travel time ]{\includegraphics[width=0.45\textwidth]{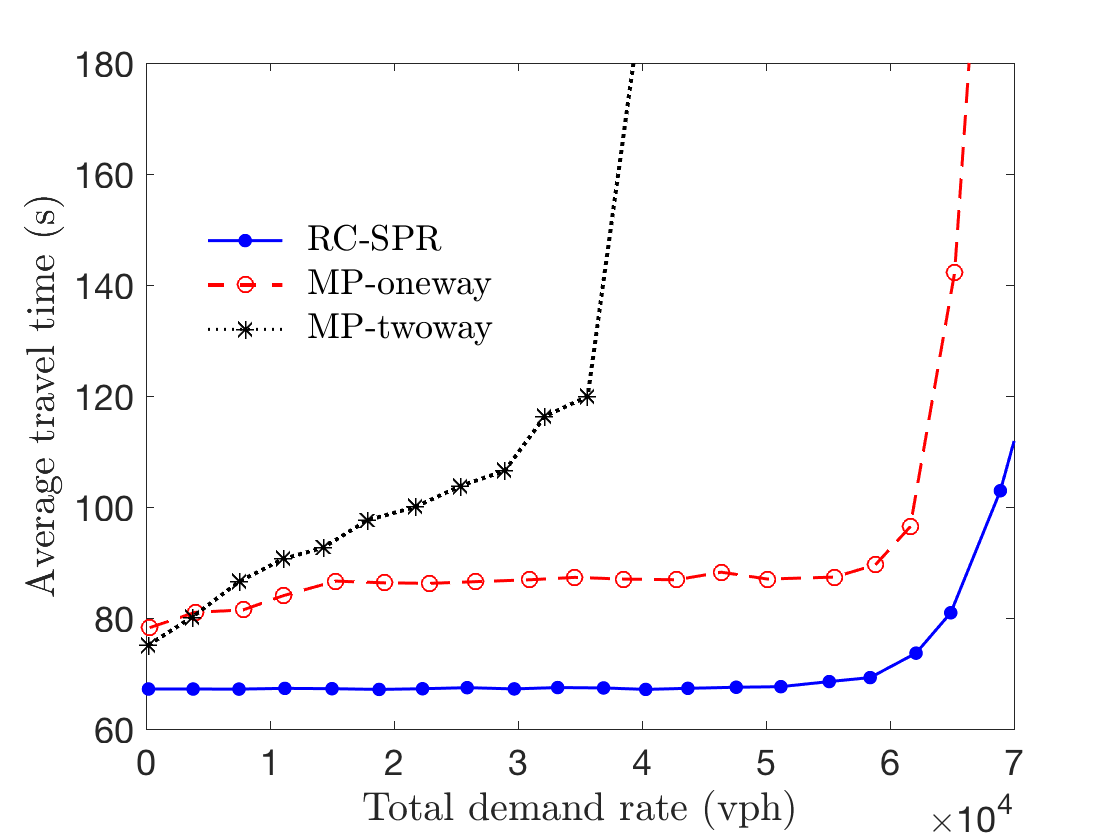}}
	\caption[]{Comparison between performances on one-way and two-way networks}
	\label{Oneway_twoway_pic}
\end{figure}

\section{Conclusions} \label{Conclusion_sec}

This study proposes an innovative and real-time rhythmic control for CAVs on grid networks, accompanied with the Part I paper (Chen et al., 2020) focusing on isolated intersections. The RC is a two-level framework. In the lower level, the entire network is endowed with a predetermined \textit{network rhythm}, and all CAVs are formed into virtual platoons that are admitted to enter the network by following the network rhythm. With the network rhythm, the CAVs experience no stop-and-go on the network, and a collision-free requirement is guaranteed by the pre-arranged platoon movements. In the upper level, and based on a pre-set network rhythm, we propose two online routing protocols, i.e., SPR and MPR, and solve them with LP-based approximation algorithms. We further provide a sufficient condition that the LP relaxation of the shortest-path online routing model yields an optimal integer solution. Some extensions are discussed for \textit{network rhythm} length choice, handling imbalanced demands, temporary within-network waiting and heterogeneous block sizes. Numerical tests are conducted to test the performances of proposed network control framework under various demand patterns, and comparative studies with MP control and FCFS control with adaptive routing are implemented to show the superiority of proposed control paradigm.\par

To the best of our knowledge, this study is among the first attempts to explore network-level control on dedicated road systems for CAVs. There are extensive further investigations following the line of this framework. First, careful exploration is required to generalize the RC into a general road network, e.g., the non-grid cases and the low-conflict network lane configuration (Eichler et al., 2013); the latter could also be suitable to implement the RC scheme. Second, the theoretical performance of the RC under stationary and non-stationary traffic deserves comprehensive investigation (e.g., the admissible demands of the control, and a more rigorous explanation for the superior performances of LP relaxation to solve the routing problems); a theoretical result can not only strengthen the practical value of the proposed control paradigm, but also provide guidance for the design of some control parameters, such as the intensity of the rhythm and the capacity of virtual platoons. Third, some real-time elements can be incorporated into the RC framework for enhancing the flexibility, such as a time-variant network rhythm and dynamically-changing platoon sizes. Fourth, since RC scheme may cause vehicle queuing at entrances and junctions, it requires associated traffic control at those places to maintain the traffic efficiencies (e.g., to manage the departure of vehicles within a parking garage). Lastly, in practical implementations, there are always some unexpected control errors and vehicle dysfunctions that may threaten the safe operation of RC. Besides reserving some safety buffers in the control, we need to design backup control protocols for handling emergencies.

\section*{Acknowledgements}

This research is supported partially by grants from National Natural Science Foundation of China (71871126, 51622807). Yafeng Yin thanks the support from the National Science Foundation (CNS-1837245 and CMMI-1904575).

\section*{References}
\noindent Balev, S., Yanev, N., Fr´eville, A., and Andonov, R. (2008). A dynamic programming based reduction
procedure for the multidimensional 0–1 knapsack problem. European Journal of Operational
Research, 186(1):63–76.\par
\noindent Bertsimas, D., Jaillet, P., and Martin, S. (2019). Online vehicle routing: The edge of optimization
in large-scale applications. Operations Research, 67(1):143–162.\par
\noindent Boyles, S. D., Rambha, T., and Xie, C. (2014). Equilibrium analysis of low-conflict network designs.
Transportation Research Record, 2467(1):129–139.\par
\noindent Chai, H., Zhang, H. M., Ghosal, D., and Chuah, C.-N. (2017). Dynamic traffic routing in a network
with adaptive signal control. Transportation Research Part C: Emerging Technologies, 85:64–85.\par
\noindent Chen, X., Li, M., Lin, X., Yin, Y., and He, F. (2020). Rhythmic control of automated traffic - part i:
Concept and properties and isolated intersections. Transportation Science, Submitted.\par
\noindent Dijkstra, E. W. (1959). A note on two problems in connexion with graphs. Numerische mathematik,
1(1):269–271.\par
\noindent Dresner, K. and Stone, P. (2008). A multiagent approach to autonomous intersection management.
Journal of artificial intelligence research, 31:591–656.\par
\noindent Eichler, D., Bar-Gera, H., and Blachman, M. (2013). Vortex-based zero-conflict design of urban
road networks. Networks and Spatial Economics, 13(3):229–254.\par
\noindent Gartner, N. H., Little, J. D., and Gabbay, H. (1975). Optimization of traffic signal settings by
mixed-integer linear programming: Part i: The network coordination problem. Transportation
Science, 9(4):321–343.\par
\noindent Hausknecht, M., Au, T.-C., and Stone, P. (2011). Autonomous intersection management: Multiintersection
optimization. In 2011 IEEE/RSJ International Conference on Intelligent Robots and
Systems, pages 4581–4586. IEEE.\par
\noindent Levin, M. W., Fritz, H., and Boyles, S. D. (2016). On optimizing reservation-based intersection
controls. IEEE Transactions on Intelligent Transportation Systems, 18(3):505–515.\par
\noindent Levin, M. W. and Rey, D. (2017). Conflict-point formulation of intersection control for autonomous
vehicles. Transportation Research Part C: Emerging Technologies, 85:528–547.\par
\noindent Li, R., Liu, X., and Nie, Y. M. (2018). Managing partially automated network traffic flow: Efficiency
vs. stability. Transportation Research Part B: Methodological, 114:300–324.\par
\noindent Long, J., Gao, Z., Zhao, X., Lian, A., and Orenstein, P. (2011). Urban traffic jam simulation based
on the cell transmission model. Networks and Spatial Economics, 11(1):43–64.\par
\noindent Lu, G., Nie, Y. M., Liu, X., and Li, D. (2019). Trajectory-based traffic management inside an
autonomous vehicle zone. Transportation Research Part B: Methodological, 120:76–98.\par
\noindent Mahmassani, H. S. (2016). 50th anniversary invited article—autonomous vehicles and connected
vehicle systems: flow and operations considerations. Transportation Science, 50(4):1140–1162.\par
\noindent Muller, E. R., Carlson, R. C., and Kraus, W. (2016). Time optimal scheduling of automated
vehicle arrivals at urban intersections. In 2016 IEEE 19th International Conference on Intelligent
Transportation Systems (ITSC), pages 1174–1179. IEEE.\par
\noindent NHTSA (2013). Preliminary statement of policy concerning automated vehicles. Washington, DC,
pages 1–14.\par
\noindent Qi, W. and Shen, Z.-J. M. (2019). A smart-city scope of operations management. Production and
Operations Management, 28(2):393–406.\par
\noindent Soleimaniamiri, S. and Li, X. (2019). Scheduling of heterogeneous connected automated vehicles
at a general conflict area. Technical report.\par
\noindent Varaiya, P. (2013). Max pressure control of a network of signalized intersections. Transportation
Research Part C: Emerging Technologies, 36:177–195.\par
\noindent Wolsey, L. A. and Nemhauser, G. L. (2014). Integer and combinatorial optimization. John Wiley and
Sons.\par
\noindent Yan, H., He, F., Lin, X., Yu, J., Li, M., and Wang, Y. (2019). Network-level multiband signal
coordination scheme based on vehicle trajectory data. Transportation Research Part C: Emerging
Technologies, 107:266–286.\par
\noindent Yu, C., Feng, Y., Liu, H. X., Ma, W., and Yang, X. (2018). Integrated optimization of traffic
signals and vehicle trajectories at isolated urban intersections. Transportation Research Part B:
Methodological, 112:89–112.\par
\noindent Gartner, N. H., Little, J. D., and Gabbay, H. (1975). Optimization of traffic signal settings by
mixed-integer linear programming: Part i: The network coordination problem. Transportation
Science, 9(4):321–343.\par
\noindent Hausknecht, M., Au, T.-C., and Stone, P. (2011). Autonomous intersection management: Multiintersection
optimization. In 2011 IEEE/RSJ International Conference on Intelligent Robots and
Systems, pages 4581–4586. IEEE.\par
\noindent Levin, M. W., Fritz, H., and Boyles, S. D. (2016). On optimizing reservation-based intersection
controls. IEEE Transactions on Intelligent Transportation Systems, 18(3):505–515.\par
\noindent Levin, M. W. and Rey, D. (2017). Conflict-point formulation of intersection control for autonomous
vehicles. Transportation Research Part C: Emerging Technologies, 85:528–547.\par
\noindent Li, R., Liu, X., and Nie, Y. M. (2018). Managing partially automated network traffic flow: Efficiency
vs. stability. Transportation Research Part B: Methodological, 114:300–324.\par
\noindent Long, J., Gao, Z., Zhao, X., Lian, A., and Orenstein, P. (2011). Urban traffic jam simulation based
on the cell transmission model. Networks and Spatial Economics, 11(1):43–64.\par
\noindent Lu, G., Nie, Y. M., Liu, X., and Li, D. (2019). Trajectory-based traffic management inside an
autonomous vehicle zone. Transportation Research Part B: Methodological, 120:76–98.\par
\noindent Mahmassani, H. S. (2016). 50th anniversary invited article—autonomous vehicles and connected
vehicle systems: flow and operations considerations. Transportation Science, 50(4):1140–1162.\par
\noindent M¨ uller, E. R., Carlson, R. C., and Kraus, W. (2016). Time optimal scheduling of automated
vehicle arrivals at urban intersections. In 2016 IEEE 19th International Conference on Intelligent
Transportation Systems (ITSC), pages 1174–1179. IEEE.\par
\noindent NHTSA (2013). Preliminary statement of policy concerning automated vehicles. Washington, DC,
pages 1–14.\par
\noindent Qi, W. and Shen, Z.-J. M. (2019). A smart-city scope of operations management. Production and
Operations Management, 28(2):393–406.\par
\noindent Soleimaniamiri, S. and Li, X. (2019). Scheduling of heterogeneous connected automated vehicles
at a general conflict area. Technical report.\par
\noindent Varaiya, P. (2013). Max pressure control of a network of signalized intersections. Transportation
Research Part C: Emerging Technologies, 36:177–195.\par
\noindent Wolsey, L. A. and Nemhauser, G. L. (2014). Integer and combinatorial optimization. John Wiley and
Sons.\par
\noindent Yan, H., He, F., Lin, X., Yu, J., Li, M., and Wang, Y. (2019). Network-level multiband signal
coordination scheme based on vehicle trajectory data. Transportation Research Part C: Emerging
Technologies, 107:266–286.\par
\noindent Yu, C., Feng, Y., Liu, H. X., Ma, W., and Yang, X. (2018). Integrated optimization of traffic
signals and vehicle trajectories at isolated urban intersections. Transportation Research Part B:
Methodological, 112:89–112.

\newpage

\section*{Appendix A}
\setcounter{figure}{0}
\renewcommand\thefigure{A-\arabic{figure}}
\setcounter{equation}{0}
\renewcommand{\theequation}{A-\arabic{equation}}

\large \textbf{Proof of Proposition \ref{prop1}} \par \normalsize

\noindent From Figure \ref{A1_fig}(a), we observe that the proposed one-way grid network is surrounded by a counterclockwise loop; this loop can be verified to be valid as we stipulate that the nethermost horizontal street is rightward and the leftmost vertical street is downward, and $m$ and $n$ are both even numbers (so the upmost street is leftward and the rightmost street is upward).  Then, for a vehicle, originating from either an entrance or a junction and destining for either an exit or a junction, to complete its trip, it can first drive to the outer loop, then circulate on it, and finally leave it for its destination or the street connecting its destination. Graphical illustrations are provided in Figures (\ref{A1_fig})-(\ref{A4_fig}). \hfill $\square$

\begin{figure}[!ht]
	\centering
	\subfloat[][]{\includegraphics[width=0.3\textwidth]{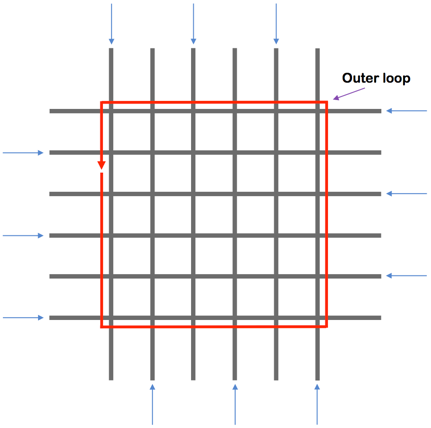}}\hspace{1em}
	\subfloat[][]{\includegraphics[width=0.3\textwidth]{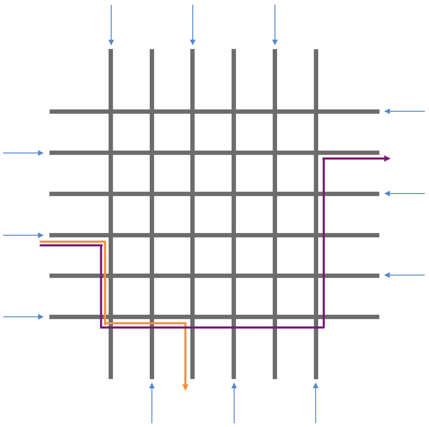}}
	\caption[]{Outer loop and entrance-to-exit routes. (a) Outer loop. (b) Typical entrance-to-exit trajectories.}
	\label{A1_fig}
\end{figure}

\begin{figure}[!ht]
	\centering
	\subfloat[][]{\includegraphics[width=0.3\textwidth]{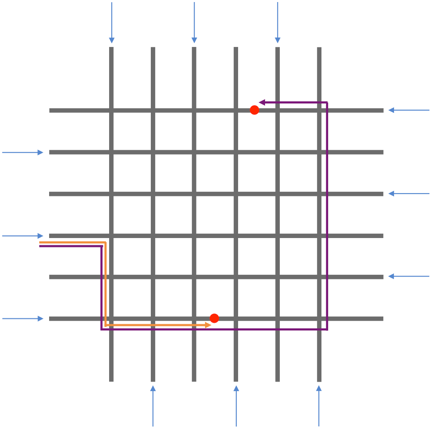}}\hspace{1em}
	\subfloat[][]{\includegraphics[width=0.3\textwidth]{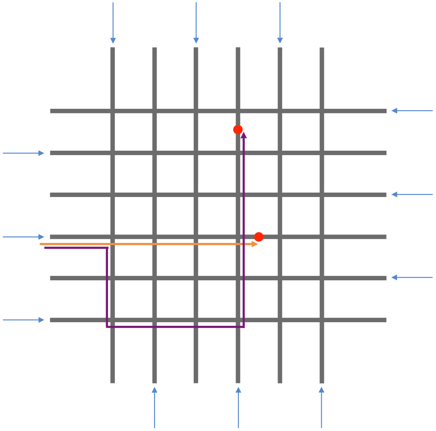}}
	\caption[]{Entrance-to-junction routes. (a) Junction on the outer loop. (b) Junction within the network.}
	\label{A2_fig}
\end{figure}

\begin{figure}[!ht]
	\centering
	\subfloat[][]{\includegraphics[width=0.3\textwidth]{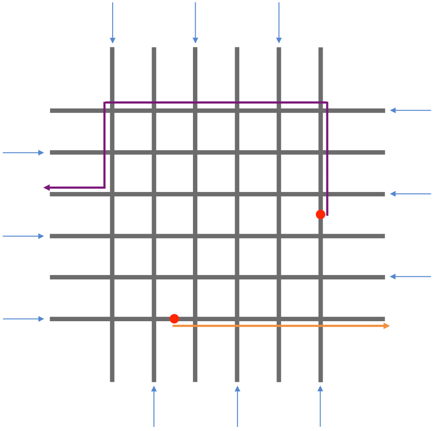}}\hspace{1em}
	\subfloat[][]{\includegraphics[width=0.3\textwidth]{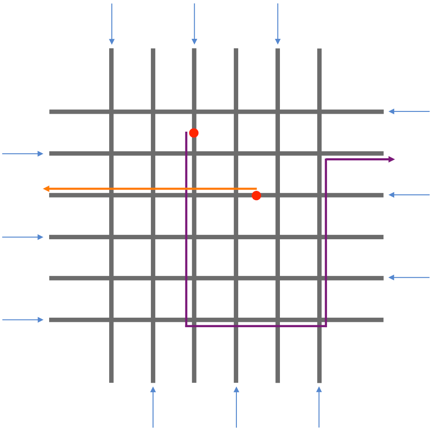}}
	\caption[]{Junction-to-exit routes. (a) Junction on the outer loop. (b) Junction within the network.}
	\label{A3_fig}
\end{figure}

\begin{figure}[!ht]
	\centering
	\includegraphics[width=0.3\textwidth]{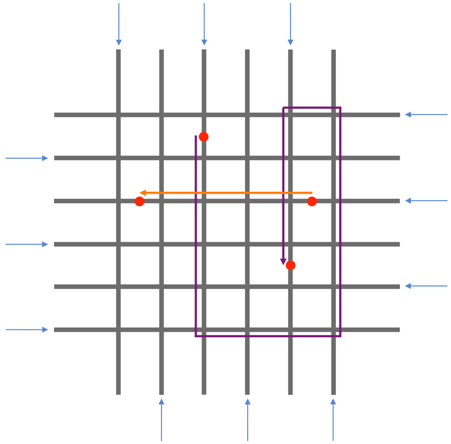}
	\caption[]{Junction-to-junction routes}
	\label{A4_fig}
\end{figure}

\par

\noindent
\large \textbf{Proof of Proposition \ref{prop3}} \par \normalsize

\noindent We prove the proposition by induction. By the unimodularity theorem, the LP-relaxation of SPR has an optimal integer solution if $\mathbf{C}_{SP}$ is totally unimodular, i.e., every square sub-matrix of $\mathbf{C}_{SP}$ is with a determinant of 1, 0 or -1 (Wolsey and Nemhauser, 2014). Hence, we only need to prove that, if $\mathbf{C}_{SP}$ is locally separable, then it is totally unimodular.\par

We first note that any square sub-matrix of $\mathbf{C}_{SP}$ with the size of 1 or 2 is always unimodular because the element is either 0 or 1. Now, assuming that any square sub-matrix of $\mathbf{C}_{SP}$ with the size less than or equal to $n_0$ is unimodular, then we prove that a square sub-matrix of $\mathbf{C}_{SP}$, denoted by $\mathbf{C}_{SP}^{n_0 + 1}$ with the size $n_0 + 1$ has a determinant of 1, 0 or -1. Consider the following two scenarios. 

\begin{itemize}
    \item i) $\mathbf{C}_{SP}^{n_0 + 1}$ has a row in Constraint (\ref{CapCons_SPR}). In such case, the associated row has only one element of 1, and then the determinant of $\mathbf{C}_{SP}^{n_0 + 1}$ can be computed by the cofactor of this element multiplied by 1 or -1. As the cofactor is the determinant of a unimodular sub-matrix, then it is 1, 0 or -1, indicating that the determinant of $\mathbf{C}_{SP}^{n_0 + 1}$ is also 1, 0 or -1. Therefore, $\mathbf{C}_{SP}^{n_0 + 1}$ is unimodular.

    \item ii) All rows in $\mathbf{C}_{SP}^{n_0 + 1}$ are in Constraint (\ref{FlowCons_SPR}). In such case, by the definition of locally separable property, if Condition (i) or (ii) is satisfied, then $\mathbf{C}_{SP}^{n_0 + 1}$ has two identical rows/columns. Hence, its determinant equals 0 because these two rows/columns are linearly dependent. If Condition (iii) holds, then there is a $k_0 \in \{ 1,2,...,n_0 \}$ such that there exist $k_0$ paths traversing at most $k_0$ temporal links, suggesting that all the elements associated with these $k_0$ paths and the remaining $n_0 + 1 - k_0$ temporal links are 0. In this context, we rearrange $\mathbf{C}_{SP}^{n_0 + 1}$ by putting the columns associated with these $k_0$ paths to the left side, and putting the rows associated with these $n_0 + 1 - k_0$ temporal links to the downside. This rearrangement yields a new matrix $\tilde{\mathbf{C}}_{SP}^{n_0 + 1}$ with the block representation as $\left[ \begin{matrix} \mathbf{C}_{k_0 \times k_0}^{11} & \mathbf{C}_{k_0 \times (n_0 + 1 - k_0)}^{12} \\ \mathbf{0}_{(n_0 + 1 - k_0) \times k_0} & \mathbf{C}_{(n_0 + 1 - k_0) \times (n_0 + 1 - k_0)}^{22} \end{matrix} \right]$. By the determinant of partitioned matrix, $\abs{\tilde{\mathbf{C}}_{SP}^{n_0 + 1}} = \abs{\mathbf{C}_{k_0 \times k_0}^{11}} \times \abs{\mathbf{C}_{(n_0 + 1 - k_0) \times (n_0 + 1 - k_0)}^{22}}$. With the assumption, $\mathbf{C}_{k_0 \times k_0}^{11}$ and $\mathbf{C}_{(n_0 + 1 - k_0) \times (n_0 + 1 - k_0)}^{22}$ are both unimodular. Then, the product of their determinants is 1, 0 or -1, implying that $\tilde{\mathbf{C}}_{SP}^{n_0 + 1}$ is unimodular. Furthermore, as the column and row maneuver of a matrix only affects the sign of its determinant, we then know $\abs{\mathbf{C}_{SP}^{n_0 + 1}}$ is 1, 0 or -1. Therefore, $\mathbf{C}_{SP}^{n_0 + 1}$ is also unimodular.
    
\end{itemize}
\par
With above discussion, we know that all square sub-matrix of size $n_0 + 1$ is unimodular. By induction we then conclude that $\mathbf{C}_{SP}$ is totally unimodular, and the LP-relaxation of SPR then exhibits an optimal integral solution. Proof completes. \hfill $\square$ \\

\noindent
\large \textbf{Proof of Proposition \ref{prop4}} \par \normalsize

\noindent We use construction to prove Proposition \ref{prop4}. Assume that the vehicle speed at crossroads $i-1$ is $v_{i-1}$, and $v_{min}^c \le v_{i-1} \le v_{max}$. We also use the notation $a_i^*$ defined in Step 2 of SCG. Now we define the following speed curve and the associated travel distance.

\begin{align}
& v(t;\theta) = \left \{ \begin{matrix} \max(-d_{max}t + v_{i-1}, \theta, c_{max}t + v_{min}^c - c_{max}a_i^*\hat{t}), 0 \le \theta \le v_{i-1} \\ \min(c_{max}t + v_{i-1}, \theta), v_{i-1} < \theta \le v_{max} \end{matrix} \right. \label{Appendix_1} \\
& D(\theta) = \int_{0}^{a_i^*\hat{t}} v(t;\theta)dt \label{Appendix_2}
\end{align}
\par

\noindent where $0 \le t \le a_i^*\hat{t}$. Then, we prove that for any $l_i \ge \frac{v_{max}^2}{2d_{max}} + \frac{(v_{min}^c)^2}{2c_{max}}$, there exists a $\theta^* \in [0, v_{max}]$ such that $D(\theta^*) = l_i$. To prove it, we first validate the following results. \par

\noindent i) $D(0)\le l_i$ and $D(v_{max})\ge l_i$. By definition, we have $D(0) = \frac{v_{i-1}^2}{2d_{max}} + \frac{(v_{min}^c)^2}{2c_{max}} \le \frac{v_{max}^2}{2d_{max}} + \frac{(v_{min}^c)^2}{2c_{max}} \le l_i$; and $D(v_{max}) = l_{max}^i(a_i^*) \ge l_i$. \par

\noindent ii) $D(\theta)$ is continuous with $\theta$ on $[0, v_{max}]$. When $0 \le \theta \le v_{i-1}$, we have: 

\begin{align*}
& \lim_{\Delta \theta \rightarrow 0} \abs{\frac{D(\theta + \Delta \theta) - D(\theta)}{\Delta \theta}} = \lim_{\Delta \theta \rightarrow 0} \frac{1}{\abs{\Delta \theta}} \int_{0}^{a_i^*\hat{t}} \abs{v(t;\theta + \Delta \theta) - v(t;\theta)}dt \\
& \qquad \qquad \qquad \qquad \qquad \quad \le \lim_{\Delta \theta \rightarrow 0} \frac{1}{\abs{\Delta \theta}} \int_{0}^{a_i^*\hat{t}} \abs{\Delta \theta}dt \\
& \qquad \qquad \qquad \qquad \qquad \quad = a_i^*\hat{t}
\end{align*}
\par

Then $D(\theta)$ is continuous on $[0, v_{i-1}]$. With a similar derivation, we can prove that $D(\theta)$ is continuous on $(v_{i-1}, v_{max}]$. Next, we show that $D(\theta)$ is continuous at $v_{i-1}$. When $v_{i-1}<v_{max}$, for $\Delta \theta > 0$, we have:

\begin{align*}
& D(v_{i-1} + \Delta \theta) = \int_{0}^{a_i^*\hat{t}} \min(c_{max}t + v_{i-1}, \theta) dt \le \int_{0}^{a_i^*\hat{t}} (v_{i-1} + \Delta \theta)dt = a_i^*\hat{t}(v_{i-1} + \Delta \theta) \\
& D(v_{i-1} + \Delta \theta) = \int_{0}^{a_i^*\hat{t}} \min(c_{max}t + v_{i-1}, \theta) dt \ge \int_{0}^{a_i^*\hat{t}} v_{i-1} dt = a_i^*\hat{t}v_{i-1}
\end{align*}
\par

Thus we have $\lim_{\Delta \theta \rightarrow 0^+} D(v_{i-1} + \Delta \theta) = a_i^*\hat{t}v_{i-1}$. Similarly, we have $\lim_{\Delta \theta \rightarrow 0^-} D(v_{i-1} + \Delta \theta) = a_i^*\hat{t}v_{i-1}$. Therefore, we conclude that $D(\theta)$ is continuous with $\theta$ on $[0, v_{max}]$.\par

With above two results, we conclude that there must exist $\theta^* \in [0, v_{max}]$ such that $D(\theta^*) = l_i$. Graphical illustrations of the above construction of speed curves are shown in Figure \ref{SCG_fig}. Proof completes. \hfill $\square$

\begin{figure}[!ht]
	\centering
	\subfloat[][]{\includegraphics[width=0.3\textwidth]{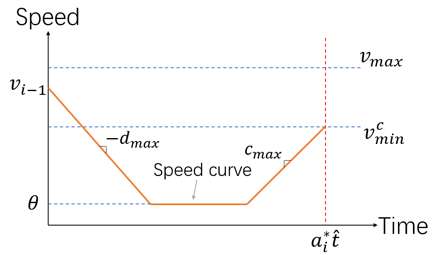}}\hspace{0.2em}
	\subfloat[][]{\includegraphics[width=0.3\textwidth]{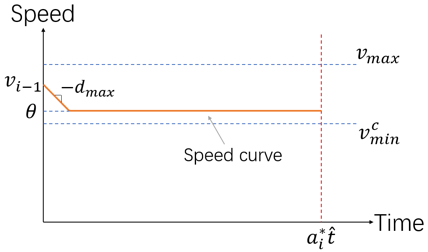}}\hspace{0.2em}
	\subfloat[][]{\includegraphics[width=0.3\textwidth]{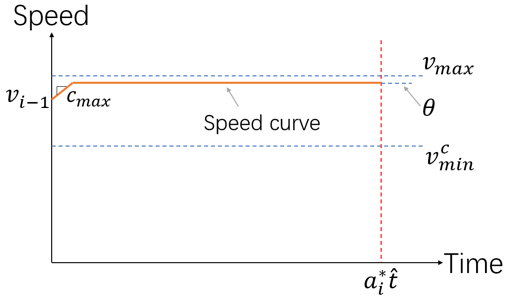}}
	\caption[]{Eligible speed curves. (a) $0 \le \theta \le v_{min}^c$ (b) $v_{min}^c \le \theta \le v_{i-1}$ (c) $v_{i-1} \le \theta \le v_{max}$}
	\label{SCG_fig}
\end{figure}

\newpage

\setcounter{figure}{0}
\renewcommand\thefigure{B-\arabic{figure}}
\setcounter{equation}{0}
\renewcommand{\theequation}{B-\arabic{equation}}
\setcounter{table}{0}
\renewcommand{\thetable}{B-\arabic{table}}


\section*{Appendix B}

\subsection*{Computational procedure for the existence probability of interconnection loops}

From Table \ref{Loop_table} we see that when the network size is no less than $6 \times 6$, the number of 3-path combination is larger than $10^{10}$; this suggests that verifying the number of loops by enumerating all combinations is intractable with a personal computer. To resolve this issue, we develop the following procedure that reduces the computational complexity. \par

\noindent \textbf{Step 1}. For each O-D pair, use Dijkstra's algorithm (Dijkstra, 1959) to extract the shortest path, and record the path along with the temporal links $(a,e)$ it passes through. \\
\textbf{Step 2}. Extract all temporal link pairs $\{ (a_1,e_1), (a_2,e_2)\}$ that are traversed by at least one path ($e_1 < e_2$); the set of these temporal link pairs is denoted by $\mathcal{A}_t$. \\
\textbf{Step 3}. For each pair in $\mathcal{A}_t$, identify all other temporal links $(a_3,e_3)$ such that $\{ (a_1,e_1), (a_3,e_3)\} \in \mathcal{A}_t$ and $\{ (a_2,e_2), (a_3,e_3)\} \in \mathcal{A}_t$, and for each $\{ (a_1,e_1), (a_2,e_2), (a_3,e_3) \}$ remove the paths passing through all these three temporal links. Then, the remaining paths associated with these three temporal link pairs constitute interconnection loops. Record all these loops and then calculate the probability of their existence. \par

Compared to the number of 3-path combinations, the number of temporal link pairs is greatly lessened. For example, on the $10 \times 10$ network, the number of 3-path combination is more than $10^{13}$, while the number of temporal link pairs is only $1.93 \times 10^5$. It takes approximately three CPU hours to extract all loops on the $10 \times 10$ network.

\subsection*{Monte-Carlo simulation for the integrality of AA-SPR}

To numerically illustrate the solution quality of AA-SPR, we design a Monte-Carlo simulation on a $6 \times 6$ one-way street network, where each road segment between two consecutive crossroads is with a junction, and the stochasticity is specified as follows. \par

\begin{itemize}
    \item The capacity for temporal link $(a,e)$ is set as $\lfloor \hat{\xi}_1 \xi_{(a,e)} \rfloor + \chi_{(a,e)}$, where $\hat{\xi}_1$ is a global random variable that obeys the uniform distribution on interval $[0,1]$, $\xi_{(a,e)}$ is a local random variable obeying the uniform distribution on interval $[0,16]$, $\chi_{(a,e)}$ is a binary random variable with the probability of 0.5 being 0 and the probability of 0.5 being 1, and $\lfloor \cdot \rfloor$ is the floor operator.
    
    \item The demand for O-D pair $w \in \mathcal{W}$ is set as $\lfloor \hat{\xi}_2 \xi_{w} \rfloor + \chi_{w}$, where $\hat{\xi}_2$ is a global random variable that obeys the uniform distribution on interval $[0,1]$, $\xi_{w}$ is a local random variable obeying the uniform distribution on interval $[0,32]$, and $\chi_{w}$ is a binary random variable with the probability of 0.5 being 0 and the probability of 0.5 being 1.
    
    \item The delay penalty for O-D pair $w \in \mathcal{W}$, i.e., $\hat{c}^w$, is set as $\bar{\chi}_{w}\bar{\xi}_{w}$, where $\bar{\chi}_{w}$ is a binary random variable with the probability of $\hat{\xi}_3$ equal to 0 and the probability of $1 - \hat{\xi}_3$ equal to 1, and $\bar{\xi}_{w}$ obeys the uniform distribution on interval $[0,50]$.
    
\end{itemize}

With the above setting, we conduct the Monte-Carlo simulation for 10,000 times. In each simulation, we solve SPR with AA-SPR, and record the solution along with the gap, computed as $\frac{S_I - S_F}{S_I} \times 100\%$, where $S_F$ and $S_I$ are the objective function values of the fractional and integral solutions, respectively. The fractional solution is from the first-time LP-relaxation of SPR, and denotes a lower bound for SPR, whereas the integral solution is obtained after AA-SPR terminates, and represents an upper bound for SPR. 

\begin{table}[!ht]
  \centering
  \caption{Summary of AA-SPR testing results}
  \label{MonteCarlo}
  \small
  \begin{tabular}{ l  c }
   \hline
		Total testing times	& 10,000 \\
		Rate of one-time integral solution   & $99.86 \%$  \\
		Maximum gap   & $0.02 \%$  \\
   \hline
  \end{tabular}
\end{table}
\par

From Table \ref{MonteCarlo}, we observe that the rate of one-time integral solution, i.e., the probability that the LP-relaxation of SPR directly provides an optimal integral solution, is $99.86\%$. This implies that for the majority of cases, AA-SPR terminates at the first iteration and provides the global optimal solution of SPR. For the remaining $0.14\%$ cases, we find that the gap between the objective function values of the solution to the LP-relaxation and the rounded integral solution is less than $0.02\%$.

\end{document}